\documentclass[a4paper,reqno]{amsart}

\usepackage{amsmath}
\usepackage{amssymb}
\usepackage{amscd}
\usepackage{amsthm}
\usepackage{type1cm}
\usepackage{tcolorbox}
\usepackage{mathrsfs}

\usepackage[colorlinks,citecolor=darkgreen,linkcolor=red]{hyperref}%arXiv用
\definecolor{darkgreen}{rgb}{0.0, 0.6, 0.0}

\usepackage[all]{xy}
\input xy
\xyoption{all}
\tcbuselibrary{breakable, skins, theorems}

\def\C{\mathcal{C}}
\def\D{\mathcal{D}}

\DeclareMathOperator{\Mod}{\mathsf{Mod}}
\DeclareMathOperator{\md}{\mathsf{mod}}
\renewcommand{\mod}{\md}
\DeclareMathOperator{\proj}{\mathsf{proj}}
\DeclareMathOperator{\fl}{\mathsf{fl}}
\DeclareMathOperator{\CM}{\mathsf{CM}}
\DeclareMathOperator{\SCM}{\mathsf{SCM}}
\DeclareMathOperator{\add}{\mathsf{add}}

\DeclareMathOperator{\refl}{\mathsf{ref}}

\DeclareMathOperator{\ind}{ind}
\DeclareMathOperator{\simp}{sim}
\DeclareMathOperator{\AR}{AR}
\DeclareMathOperator{\rank}{rank}
\DeclareMathOperator{\Hom}{Hom}
\DeclareMathOperator{\End}{End}
\DeclareMathOperator{\Ext}{Ext}
\DeclareMathOperator{\rad}{rad}
\DeclareMathOperator{\Spec}{Spec}

\DeclareMathOperator{\height}{ht}
\DeclareMathOperator{\depth}{depth}
\DeclareMathOperator{\Frac}{Frac}
\DeclareMathOperator{\ch}{char}
\DeclareMathOperator{\Ker}{Ker}
\DeclareMathOperator{\im}{Im}
\renewcommand{\Im}{\im}
\DeclareMathOperator{\op}{op}
\DeclareMathOperator{\Gal}{Gal}
\DeclareMathOperator{\Cl}{Cl}
\DeclareMathOperator{\soc}{soc}
\DeclareMathOperator{\RHom}{\mathbb{R}Hom}
\DeclareMathOperator{\MK}{MK}

\def\gl{\mathop{\rm gl.dim}\nolimits}
\def\lgl{\mathop{\rm l.gl.dim}\nolimits}

\def\pd{\mathop{\rm proj.dim}\nolimits}
\def\id{\mathop{\rm inj.dim}\nolimits}

\DeclareMathOperator{\Ind}{Ind_{\it H}^{\it G}}
\DeclareMathOperator{\Res}{Res_{\it H}}

\newtheorem{Thm}{Theorem}[section]
\newtheorem{Lem}[Thm]{Lemma}
\newtheorem{Prop}[Thm]{Proposition}
\newtheorem{Cor}[Thm]{Corollary}

\theoremstyle{definition}
\newtheorem{Def}[Thm]{Definition}
\newtheorem{Ex}[Thm]{Example}
\newtheorem{Rem}[Thm]{Remark}
\newtheorem{DefProp}[Thm]{Definition-Proposition}

\newcommand{\FRAC}[2]{\leavevmode\kern.1em\raise.5ex\hbox{\the\scriptfont0 #1}\kern-.1em/\kern-.15em\lower.25ex\hbox{\the\scriptfont0 #2}}

\title{Cohen-Macaulay representations of invariant subrings}
\author{Ryu Tomonaga}
\address{Graduate School of Mathematical Sciences, The University of Tokyo, 3-8-1 Komaba, Meguro-ku, Tokyo, 153-8914, Japan}
\email{ryu-tomonaga@g.ecc.u-tokyo.ac.jp}

\begin{document}
\begin{abstract}
We classify two-dimensional complete local rings $(R,\mathfrak{m},k)$ of finite Cohen-Macaulay type where $k$ is an arbitrary field of characteristic zero, generalizing works of Auslander and Esnault for algebraically closed case. Our main result shows that they are precisely of the form $R=l[[x_1,x_2]]^G$ where $l/k$ is a finite Galois extension and $G$ is a finite group acting on $l[[x_1,x_2]]$ as a $k$-algebra. In fact, $G$ can be linearized to become a subgroup of $GL_2(l)\rtimes\Gal(l/k)$. Moreover, we establish algebraic McKay correspondence in this general setting and completely describe its McKay quiver, which is often non-simply laced, as a quotient of another certain McKay quiver. Combining these results, we classify the quivers that may arise as the Auslander-Reiten quivers of two-dimensional Gorenstein rings of finite Cohen-Macaulay type of equicharacteristic zero. These are shown to be either doubles of (not necessarily simply-laced!) extended Dynkin diagrams or of type $\widetilde{A}_0$ or $\widetilde{CL}_n$ having loops.

More generally, we consider higher dimensional $R=l[[x_1,\cdots,x_d]]^G\ (G\subseteq GL_d(l)\rtimes\Gal(l/k))$ and show they have non-commutative crepant resolutions (NCCRs). Furthermore, we explicitely determine the quivers of the NCCRs as quotients of another certain quivers.

To accomplish these, we establish two results which are of independent interest. First, we prove the existence of $(d-1)$-almost split sequences for arbitrary $d$-dimensional Cohen-Macaulay rings having NCCR, even when their singularities are not isolated. Second, we give an explicit recipe to determine irreducible representations of skew group algebras $l*G$ in terms of those over the group algebras $lH$ where $H$ is the kernel of the action of $G$ on $l$.
\end{abstract}

\maketitle
\tableofcontents

\section*{Introduction}

%%%%%%%%%%%%%%%%%%%%%%%%%%%%%%%%%%%%%%%%%%%%%%%%%%%%%%%%%%%%%
\subsection{Classical results}
%%%%%%%%%%%%%%%%%%%%%%%%%%%%%%%%%%%%%%%%%%%%%%%%%%%%%%%%%%%%%

Auslander-Reiten theory, which gives Auslander-Reiten quivers for certain categories, is an indispensable tool to study categories of finitely generated modules over finite dimensional algebras \cite{ASS,ARS} and Cohen-Macaulay modules over orders and commutative rings \cite{I18,LW,Y}. On the other hand, quotient singularities form a class of rings which has interesting geometrical properties known as McKay correspondence \cite{AV85,Mc}. In 1980's, Auslander developed theories of Cohen-Macaulay representations for two-dimensional quotient singularities \cite{Aus86}. This says that the Auslander-Reiten quiver of a quotient singularity given by a finite group $G$ coincides with the McKay quiver of $G$, which is called algebraic McKay correspondence. In 2000's, Iyama gave its higher dimensional generalization by developing higher dimensional Auslander-Reiten theory, where modules called cluster tilting play an essential role \cite{Iya07b,Iya07a}. This concept of cluster tilting turns out to have strong relationships with cluster theories through additive categorifications of cluster algebras \cite{Am09,BMRRT,Kel} and algebraic geometry through non-commutative crepant resolutions (NCCR) \cite{IW,VdB}. Recently, NCCRs are constructed for several examples \cite{Bro,BIKR,Han23,Har,HIMO,SVdB} which often produce higher representation-infinite algebras \cite{HIO} and induce triangle equivalences between singularity categories and cluster categories \cite{AIR15,Han24a,HI22}.

The table below is a summary of the known classification results of Cohen-Macaulay complete local rings $(R,\mathfrak{m},k)$ of finite Cohen-Macaulay type where $k$ is algebraically closed with $\ch k=0$ (see \cite{BGS,LW,Y} for details). Recall that $R$ is said to be {\it of finite Cohen-Macaulay type} if the category $\CM R$ of maximal Cohen-Macaulay modules has only finitely many indecomposable objects. By the Krull-Schmidt theorem, this is equivalent to that $\CM R$ has an additive generator. Here, we call $M\in\CM R$ an additive generator if $\add_RM=\CM R$ holds where $\add_RM=\{X\in\mod R\mid X\text{ is a direct summand of }M^{\oplus n}\text{ for some }n\in\mathbb{N}\}\subseteq\CM R$. When $R$ is Gorenstein, then these rings are precisely simple singularities. For general $R$, the complete classification is known when $\dim R\leq2$. If $\dim R\geq3$ and $R$ is non-Gorenstein, then the classification problem is still open and only two examples are known.

\[
\begin{array}{c|c|c}
\dim R&\text{Gorenstein}&\text{Cohen-Macaulay}\\ \hline\hline
0&k[x]/(x^n)&k[x]/(x^n)\\
1&\text{simple singularities}&\text{rings dominating simple singularities}\\
2&\rule[1mm]{1.4cm}{0.2mm}''\rule[1mm]{1.4cm}{0.2mm}&\text{quotient singularities}\\
\geq3&\rule[1mm]{1.4cm}{0.2mm}''\rule[1mm]{1.4cm}{0.2mm}&\text{open (only two non-Gorenstein examples are known)}
\end{array}
\]

Now we focus on the case $\dim R=2$ and $R$ is not necessarily Gorenstein. The known results are summarized in Theorem \ref{classical} and \ref{classicalfrt} below.

\begin{Thm}\cite[2.2, 3.3]{Aus86}\label{classical}
Let $k$ be a field, $G$ a finite subgroup of $GL_2(k)$ with $|G|$ not divided by $\ch k$, $S:=k[[x,y]]$ the formal power series ring and $\mathfrak{n}:=(x, y)$ the maximal ideal of $S$. Let $S^G\subseteq S$ be the invariant subring.
\begin{enumerate}
\item The ring $(S^G, \mathfrak{n}\cap S^G, k)$ is a two-dimensional complete Cohen-Macaulay local ring and $\CM S^G$ has $S$ as an additive generator. In particular, $S^G$ is of finite Cohen-Macaulay type.
\item If $G\subseteq GL_2(k)$ is small (see Definition \ref{small}), then we have $\End_{S^G}(S)\cong S*G$.
\item The Auslander-Reiten quiver $\AR(\CM S^G)$ coincides with the McKay quiver of $kG$.
\end{enumerate}
\end{Thm}

This (3) says that information about Cohen-Macaulay representations of $S^G$ can be obtained from irreducible representations of $G$ over $k$. The converse to (1) holds in the following sense.

\begin{Thm}\cite[4.9]{Aus86}\cite{Esn}\label{classicalfrt}
Let $(R, \mathfrak{m}, k)$ be a two-dimensional Cohen-Macaulay complete local ring of finite Cohen-Macaulay type. If $k$ is an algebraically closed field of characteristic zero, then there exists a finite subgroup $G\subseteq GL_2(k)$ such that $R\cong k[[x,y]]^G$ holds.
\end{Thm}

\begin{comment}
From this theorem, the classification problem of two-dimensional rings of finite Cohen-Macaulay type over $\mathbb{C}$ is attributed to famous Klein's classification of finite subgroups of $GL_2(\mathbb{C})$.

This classical theory can handle, for example, $\mathbb{C}[[x^n, xy, y^n]]=\mathbb{C}[[x, y]]^G\ (G=\langle{\rm diag}(\zeta_n,\zeta_n^{-1})\rangle\subseteq GL_2(\mathbb{C}))$, but not $k[[x^n, xy, y^n]]$ if $\zeta_n\notin k$, which is in fact of finite Cohen-Macaulay type. Our results can deal with such a ring and gives how to draw its Auslander-Reiten quiver.
\end{comment}

%%%%%%%%%%%%%%%%%%%%%%%%%%%%%%%%%%%%%%%%%%%%%%%%%%%%%%%%%%%%%
\subsection{Finite Cohen-Macaulay type over arbitrary fields in dimension two}
%%%%%%%%%%%%%%%%%%%%%%%%%%%%%%%%%%%%%%%%%%%%%%%%%%%%%%%%%%%%%

In this subsection, we restrict ourselves to the two-dimensional case. Our main aim of this paper is to study two-dimensional rings $(R, \mathfrak{m}, k)$ of finite Cohen-Macaulay type where $k$ is not necessarily algebraically closed. Actually, there are many examples not covered by Theorem \ref{classical} and \ref{classicalfrt}. We investigate more general invariant subrings and see that these rings give plenty of examples of rings of finite Cohen-Macaulay type (Theorem \ref{introcmfin}). Moreover, we prove that all two-dimensional rings of finite Cohen-Macaulay type of equicharacteristic zero can be obtained as our invariant subrings, which is the answer to the classification problem (Theorem \ref{introfrt}).

Let $l$ be a field, $S:=l[[x, y]]$ the formal power series ring, $\mathfrak{n}:=(x, y)$ the maximal ideal of $S$ and $G$ a finite group acting on $S$ as a ring, not necessarily as an $l$-algebra, with $|G|$ not divided by $\ch l$. Then we obtain the ring $S^G$ of invariants. We prove the following analogue of Theorem \ref{classical}(1).

\begin{Thm}{\rm(Theorem \ref{addgen})}\label{introcmfin}
The ring $(S^G, \mathfrak{n}\cap S^G, l^G)$ is a two-dimensional Cohen-Macaulay complete local normal domain which is of finite Cohen-Macaulay type. More strongly, $S\in\CM S^G$ gives an additive generator.
\end{Thm}

We put $k:=l^G$. Then we can linearize $G$, that is, we can assume that $G$ is a subgroup of the semidirect product $GL_2(l)\rtimes\Gal(l/k)\subseteq{\rm Aut}_k^{\rm al}(S)$ (see Proposition \ref{linear}).

We prove that a converse to Theorem \ref{introcmfin} holds in the following sense, which guarantees that our new setting is essential in the representation theory of Cohen-Macaulay rings with non-algebraically closed residue fields.

\begin{Thm}{\rm(Theorem \ref{frt})}\label{introfrt}
Let $(R, \mathfrak{m}, k)$ be a two-dimensional Cohen-Macaulay complete local ring of finite Cohen-Macaulay type with $\ch k=0$. Then there exists a finite Galois extension $l/k$ and a finite subgroup $G\subseteq GL_2(l)\rtimes\Gal(l/k)$ such that $R\cong l[[x,y]]^G$ holds.
\end{Thm}

Observe that this is a generalization of Theorem \ref{classicalfrt} admitting non-algebraically closed residue fields. Later, we explain how to draw the Auslander-Reiten quiver of the category $\CM S^G$ (Theorem \ref{introARMcKay}, \ref{introalg}). Combining these results, we can characterize quivers which may appear as $\AR(\CM R)$ where $(R,\mathfrak{m},k)$ is a two-dimensional Gorenstein complete local ring of finite Cohen-Macaulay type with $\ch k=0$. When $k$ is algebraically closed, these quivers are just doubles of simply laced extended Dynkin diagrams. Admitting non-algebraically closed fields, these are mostly doubles of extended Dynkin diagrams. However, in fact, some of them have loops. In addition, not all the double of extended Dynkin diagrams appear.

\begin{Thm}{\rm(Theorem \ref{classfyquiver})}\label{introclassfyquiver}
Let $(R,\mathfrak{m},k)$ be a two-dimensional Gorenstein complete local ring of finite Cohen-Macaulay type with $\ch k=0$. Then quivers which may appear as $\AR(\CM R)$ are precisely listed below. 
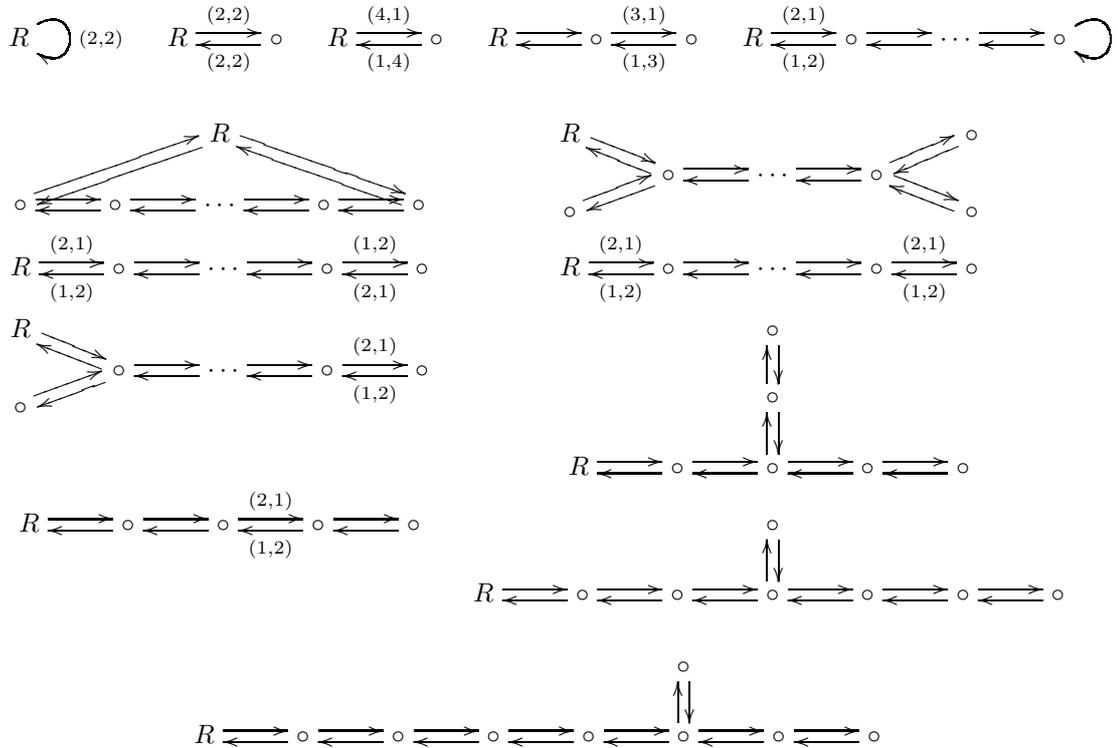
\[
\begin{array}{c c c c c}
\xymatrix{
R \ar@(ur,dr)^{(2,2)}
}&\xymatrix{
R \ar@<0.5ex>[r]^-{(2,2)} & \circ \ar@<0.5ex>[l]^-{(2,2)}
}&\xymatrix{
R \ar@<0.5ex>[r]^-{(4,1)} & \circ \ar@<0.5ex>[l]^-{(1,4)}
}&\xymatrix{
R \ar@<0.5ex>[r] & \circ \ar@<0.5ex>[r]^{(3,1)} \ar@<0.5ex>[l] & \circ \ar@<0.5ex>[l]^{(1,3)}
}&\xymatrix{
R \ar@<0.5ex>[r]^-{(2,1)} & \circ \ar@<0.5ex>[l]^-{(1,2)} \ar@<0.5ex>[r] & \cdots \ar@<0.5ex>[l] \ar@<0.5ex>[r] & \circ \ar@<0.5ex>[l] \ar@(ur,dr)
}
\end{array}
\]
\[
\begin{array}{c c}
\xymatrix@R=5mm{
& & R \ar@<0.5ex>[dll] \ar@<0.5ex>[drr] & & \\
\circ \ar@<0.5ex>[urr] \ar@<0.5ex>[r] & \circ \ar@<0.5ex>[l] \ar@<0.5ex>[r] & \cdots \ar@<0.5ex>[l] \ar@<0.5ex>[r] & \circ \ar@<0.5ex>[l] \ar@<0.5ex>[r] & \circ \ar@<0.5ex>[ull] \ar@<0.5ex>[l]
}&\xymatrix@R=1mm{
R \ar@<0.5ex>[dr] & & & & \circ \ar@<0.5ex>[dl]\\
& \circ \ar@<0.5ex>[ul] \ar@<0.5ex>[dl] \ar@<0.5ex>[r] & \cdots \ar@<0.5ex>[l] \ar@<0.5ex>[r] & \circ \ar@<0.5ex>[l]\ar@<0.5ex>[ur] \ar@<0.5ex>[dr] & \\
\circ \ar@<0.5ex>[ur] & & & & \circ \ar@<0.5ex>[ul]
}\\
\xymatrix{
R \ar@<0.5ex>[r]^-{(2,1)} & \circ \ar@<0.5ex>[l]^-{(1,2)} \ar@<0.5ex>[r] & \cdots \ar@<0.5ex>[l] \ar@<0.5ex>[r] & \circ \ar@<0.5ex>[l] \ar@<0.5ex>[r]^-{(1,2)} & \circ \ar@<0.5ex>[l]^-{(2,1)}
}&\xymatrix{
R \ar@<0.5ex>[r]^-{(2,1)} & \circ \ar@<0.5ex>[l]^-{(1,2)} \ar@<0.5ex>[r] & \cdots \ar@<0.5ex>[l] \ar@<0.5ex>[r] & \circ \ar@<0.5ex>[l] \ar@<0.5ex>[r]^-{(2,1)} & \circ \ar@<0.5ex>[l]^-{(1,2)}
}\\
\xymatrix@R=1mm{
R \ar@<0.5ex>[dr] & & & & \\
& \circ \ar@<0.5ex>[ul] \ar@<0.5ex>[dl] \ar@<0.5ex>[r] & \cdots \ar@<0.5ex>[l] \ar@<0.5ex>[r] & \circ \ar@<0.5ex>[l]\ar@<0.5ex>[r]^-{(2,1)} & \circ \ar@<0.5ex>[l]^-{(1,2)}\\
\circ \ar@<0.5ex>[ur] & & & &
}&\xymatrix@R=5mm{
&&\circ \ar@<0.5ex>[d]&&\\
&&\circ \ar@<0.5ex>[d] \ar@<0.5ex>[u]&&\\
R \ar@<0.5ex>[r] & \circ \ar@<0.5ex>[l] \ar@<0.5ex>[r] & \circ \ar@<0.5ex>[l] \ar@<0.5ex>[r] \ar@<0.5ex>[u] & \circ \ar@<0.5ex>[l] \ar@<0.5ex>[r] & \circ \ar@<0.5ex>[l]
}\\
\xymatrix{
R \ar@<0.5ex>[r] & \circ \ar@<0.5ex>[l] \ar@<0.5ex>[r] & \circ \ar@<0.5ex>[l] \ar@<0.5ex>[r]^-{(2,1)} & \circ \ar@<0.5ex>[l]^-{(1,2)} \ar@<0.5ex>[r] &\circ \ar@<0.5ex>[l]
}&\xymatrix@R=5mm{
&&&\circ \ar@<0.5ex>[d]&&&\\
R \ar@<0.5ex>[r] & \circ \ar@<0.5ex>[l] \ar@<0.5ex>[r] & \circ \ar@<0.5ex>[l] \ar@<0.5ex>[r] & \circ \ar@<0.5ex>[l] \ar@<0.5ex>[r] \ar@<0.5ex>[u] & \circ \ar@<0.5ex>[l] \ar@<0.5ex>[r] & \circ \ar@<0.5ex>[l] \ar@<0.5ex>[r] & \circ \ar@<0.5ex>[l]
}
\end{array}
\]

\[\xymatrix@R=5mm{
&&&&&\circ \ar@<0.5ex>[d]&&\\
R \ar@<0.5ex>[r] & \circ \ar@<0.5ex>[l] \ar@<0.5ex>[r] & \circ \ar@<0.5ex>[l] \ar@<0.5ex>[r] & \circ \ar@<0.5ex>[l] \ar@<0.5ex>[r] & \circ \ar@<0.5ex>[l] \ar@<0.5ex>[r] & \circ \ar@<0.5ex>[l] \ar@<0.5ex>[r] \ar@<0.5ex>[u] & \circ \ar@<0.5ex>[l] \ar@<0.5ex>[r] & \circ \ar@<0.5ex>[l]
}\]
Thus doubles of all extended Dynkin diagrams except for type $\widetilde{B}_n, \widetilde{CD}_n, \widetilde{F}_{41}$ and $\widetilde{G}_{21}$ appear. In addition, two of them have loops.

\end{Thm}

We explain details of Theorems \ref{introcmfin} and \ref{introclassfyquiver} in the rest of this introduction.

%%%%%%%%%%%%%%%%%%%%%%%%%%%%%%%%%%%%%%%%%%%%%%%%%%%%%%%%%%%%%
\subsection{Auslander-Reiten theory for invariant subrings}
%%%%%%%%%%%%%%%%%%%%%%%%%%%%%%%%%%%%%%%%%%%%%%%%%%%%%%%%%%%%%

In this subsection, we consider the general case $d\geq2$. First, as a preliminary, we develop the higher-dimensional Auslander-Reiten theory for NCCRs. Here NCCR, which was introduced in \cite[4.1]{VdB}, is what gives a virtual space of crepant resolutions of a given singularity $\Spec R$.

\begin{Def}\cite[4.1]{VdB}\cite[5.1]{IW}
Let $R$ be a Cohen-Macaulay complete local normal domain with $d:=\dim R\geq2$.
\begin{enumerate}
\item A module-finite $R$-algebra $\Gamma$ is called a {\it regular order} if $\Gamma\in\CM R$ and $\gl\Gamma=d$ hold.
\item Take a reflexive $R$-module $M\neq0$. We say $M$ gives an {\it NCCR} if $\End_R(M)$ is a regular $R$-order.
\end{enumerate}
\end{Def}

For a two-dimensional local normal domain $R$, it is classical that $0\neq M\in\CM R$ is an additive generator if and only if $\End_R(M)$ is a regular $R$-order \cite{RVdB} (see \cite[5.4]{IW} for a complete proof). Therefore this concept of NCCRs is a higher-dimensional analogue of additive generators of $\CM R$.

In the representation theory, homological properties of endomorphism rings lead to representation theoretic results. For example, if $R$ is a two-dimensional local normal domain, then a Cohen-Macaulay $R$-module whose endomorphism ring is a regular order becomes an additive generator of $\CM R$ as we stated above. We further remark that the Auslander-Iyama correspondence realizes this philosophy \cite{Au71,Iya07a}. In this paper, we prove the following theorem which deduces a representation theoretic property (the existence of $(d-1)$-almost split sequences, see Definition \ref{ARseq}) from a homological property (NCCR). 

\begin{Thm}{\rm(Theorem \ref{existAR}, Proposition \ref{ARforCT})}\label{introexistAR}
Let $R$ be a Cohen-Macaulay complete local normal domain with $d:=\dim R\geq2$, $M$ a reflexive $R$-module giving an NCCR and $X\in\ind\add_RM$.
\begin{enumerate}
\item There exists a $(d-1)$-almost split sequence in $\add_RM$ of the following form.
\[\nu X\longrightarrow C_{d-2}\longrightarrow\cdots\longrightarrow C_0\longrightarrow X\]
\item Assume $M\in\CM R$. If $X\neq R$ (respectively, $X=R$), then the complex $0\to\nu X\to C_{d-2}\to\cdots\to C_0\to X\to 0$ (respectively, $0\to\nu X\to C_{d-2}\to\cdots\to C_0\to X$) is exact.
\end{enumerate}
\end{Thm}

We remark that if $R$ is an isolated singularity and $M\in\CM R$, then $M$ gives an NCCR if and only if it is a $(d-1)$-cluster tilting object in $\CM R$, where the concept of cluster tilting is what gives a deep connection between representation theories and cluster theories \cite{BMRRT,Iya07b}. In \cite{Iya07b}, the existence of higher almost split sequences was proved for cluster tilting modules over isolated singularities. Thus Theorem \ref{introexistAR} can be viewed as a generalization of the result of \cite{Iya07b} to arbitrary Cohen-Macaulay normal domains which are not necessarily isolated singularities.

Now we start to explain our results about quotient singularities in arbitrary dimensions. We consider the ring $l[[x_1,\cdots,x_d]]^G$ of invariants where $l$ is a field and $G$ is a finite group acting on the formal power series ring $l[[x_1,\cdots,x_d]]$ as a ring, not necessarily as an $l$-algebra, with $|G|$ not divided by $\ch l$. If the action of $G$ is $l$-linear, then it is well-known that we can linearize $G$ so that $G$ is a finite subgroup of $GL_d(l)$ \cite{Car}. In the general case, thanks to the following result, we can assume $G$ to be a finite subgroup of a certain explicit group.

\begin{Prop}{\rm(Proposition \ref{linear})}
Let $l$ be a field, $G$ a finite group acting on $S:=l[[x_1,\ldots , x_d]]$ as a ring with $|G|$ not divided by $\ch l$. Observe that $G$ acts on the residue field $l$ naturally and put $k:=l^G$. Then there exists a group homomorphism $\phi\colon G\to GL_d(l)\rtimes\Gal(l/k)$ and a ring automorphism $\theta\colon S\to S$ such that for all $g\in G$, the following diagram commutes.
\[\xymatrix{
S \ar[r]^\theta \ar[d]_{\phi(g)} & S\ar[d]^{g}\\
S\ar[r]_\theta & S
}\]
In particular, we have a ring isomorphism $S^G\cong S^{\phi(G)}$.
\end{Prop}

In conclusion, our setting is the following.
\begin{enumerate}
\item[(Q1)] $l/k$ is a finite Galois extension of fields.
\item[(Q2)] $G$ is a finite subgroup of $GL_d(l)\rtimes\Gal(l/k)$ with $|G|$ not divided by $\ch k$ and $k=l^G$.
\item[(Q3)] $S:=l[[x_1,\ldots , x_d]]$ is the formal power series ring and $\mathfrak{n}:=(x_1,\ldots , x_d)\subseteq S$ is the maximal ideal.
\end{enumerate}

We consider the ring $S^G$ of invariants. Our next result is a higher-dimensional generalization of Theorem \ref{introcmfin}.

\begin{Thm}{\rm(Theorem \ref{CT})}
Under the setting $(Q1),(Q2)$ and $(Q3)$, the ring $(S^G, \mathfrak{n}\cap S^G, k)$ is a $d$-dimensional Cohen-Macaulay complete local normal domain and $S\in\CM S^G$ gives an NCCR.
\end{Thm}

Therefore, thanks to Theorem \ref{introexistAR}, we have $(d-1)$-almost split sequences in 
\[\C(S^G):=\add_{S^G}S,\] 
where $\C(S^G)=\CM S^G$ holds when $d=2$. Next, we give a complete description of the Auslander-Reiten quiver of $\C(S^G)$. As we have already stated, homological properties of endomorphism rings play an important role in the representation theory. In fact, a key step to draw $\AR(\C(S^G))$ is to determine the ring $\End_{S^G}(S)$. This ring is isomorphic to the skew group ring under a certain condition as in the classical case (Theorem \ref{classical}(2)).

\begin{Thm}{\rm(Theorem \ref{Auslanderalgebra})}
Under the setting $(Q1),(Q2)$ and $(Q3)$, let $H:=G\cap GL_d(l)$ be a normal subgroup of $G$. If $H\subseteq GL_d(l)$ is small, then we have 
\[\End_{S^G}(S)\cong S*G.\]
\end{Thm}

The assumption that $H$ is small is not restrictive because without changing $S^G$, we can replace $G$ so that $H$ is small (see Theorem \ref{makesmall}). Using this theorem, we obtain the following bijections:
\[\xymatrix@C=6em{
\ind\C(S^G)\ar@<.3ex>^-{\Hom_R(S,-)}[r]&\ind\proj S*G\ar@<.3ex>^-{(-)^G}[l]\ar@<.3ex>^-{(S/\mathfrak{n})\otimes_S-}[r]&\simp l*G\ar@<.3ex>^-{S\otimes_l-}[l]
}.\]
For $V\in\mod l*G$, we denote by
\[M(V):=(S\otimes_lV)^G\in\C(S^G)\]
the correspondent to $V$. Remark that $\ind\C(S^G)=\{M(V)\mid V\in\simp l*G\}$ holds. We can describe all $(d-1)$-almost split sequences in $\C(S^G)$ explicitly.

\begin{Thm}{\rm(Theorem \ref{ARMcKay})}\label{introARMcKay}
Under the setting $(Q1),(Q2)$ and $(Q3)$, we assume $H:=G\cap GL_d(l)\subseteq GL_d(l)$ is small.
\begin{enumerate}
\item The Auslander-Reiten quiver $\AR(\C(S^G))$ coincides with the McKay quiver of $l*G$ which we denote by $\MK(l*G)$ (see Definition \ref{McKay})$\colon$
\[\AR(\C(S^G))=\MK(l*G).\]
\item More precisely, let $V\in\simp l*G$ and $U:=\mathfrak{n}/\mathfrak{n}^2\in\mod l*G$. Then we have a $(d-1)$-almost split sequence of the following form.
\[M(\bigwedge^dU\otimes_lV)\longrightarrow M(\bigwedge^{d-1}U\otimes_lV)\longrightarrow\cdots\longrightarrow M(U\otimes_lV)\longrightarrow M(V)\]
%%
\begin{comment}
If $V$ is not the trivial representation, then
\[0\longrightarrow M(\bigwedge^dU\otimes_lV)\longrightarrow\cdots\longrightarrow M(\bigwedge^2U\otimes_lV)\longrightarrow M(U\otimes_lV)\longrightarrow M(V)\longrightarrow 0\]
is a $(d-1)$-almost split sequence in $\C(S^G)$. If $V$ is the trivial representation, then
\[0\longrightarrow M(\bigwedge^dU\otimes_lV)\longrightarrow\cdots\longrightarrow M(\bigwedge^2U\otimes_lV)\longrightarrow M(U\otimes_lV)\longrightarrow M(V)\]
is a $(d-1)$-fundamental sequence in $\C(S^G)$.
\end{comment}
%%
\end{enumerate}
\end{Thm}

Here, the McKay quiver $\MK(l*G)$ is a valued quiver whose vertex set is $\simp l*G$ and whose valuations are determined by information about $\mod l*G$.

%%%%%%%%%%%%%%%%%%%%%%%%%%%%%%%%%%%%%%%%%%%%%%%%%%%%%%%%%%%%%
\subsection{McKay quivers of skew group algebras}
%%%%%%%%%%%%%%%%%%%%%%%%%%%%%%%%%%%%%%%%%%%%%%%%%%%%%%%%%%%%%

In the previous subsection, we see that the McKay quiver $\MK(l*G)$ coincides with the Auslander-Reiten quiver $\AR(\C(S^G))$.  In this subsection, we explain how to draw $\MK(l*G)$. For this purpose, we introduce a general method for determining irreducible representations of skew group algebras. Our method is to describe $\simp l*G$ by using $\simp lH$.

We discuss in the following general setting: let $l$ be a field, $G$ a finite group acting on $l$ as a field and $k:=l^G$. Now $l/k$ is a finite Galois extension and we have the natural group isomorphism $G/H\cong\Gal(l/k)$ where $H:=\Ker(G\to\Gal(l/k))$. To determine $\simp l*G$, we use the natural action of $\Gal(l/k)$ on the set $\simp lH$. Observe that we have the restriction functor and the induction functor
\[\Res\colon\mod l*G\to\mod lH,\ \Ind\colon\mod lH\to\mod l*G.\]

\begin{Thm}{\rm(Theorem \ref{tab})}\label{introsab}
Under the setting above, the following assertions hold.
\begin{enumerate}
\item There exists a bijection
\[\Gal(l/k)\backslash\simp lH\cong\simp l*G.\]
\item Decompose $\simp lH$ into the disjoint union of the orbits with respect the $\Gal(l/k)$-action$\colon$
\[\simp lH=\bigsqcup_{i=1}^{s}\{W_{i1},\cdots, W_{it_i}\}.\]
We denote by $V_i\in\simp l*G$ the simple left $l*G$-module corresponding to $\{W_{i1},\cdots, W_{it_i}\}$. Then for each $1\leq i\leq s$, there exist integers $a_i, b_i\geq1$ satisfying
\begin{equation*}
\begin{split}
\Res V_i&\cong (W_{i1}\oplus\cdots\oplus W_{it_i})^{\oplus a_i},\\
\Ind W_{ij}&\cong V_i^{\oplus b_i}\text{for each $1\leq j\leq t_i$ and}\\
t_ia_ib_i&=[l:k].
\end{split}
\end{equation*}
\end{enumerate}
\end{Thm}

Now we return to the setting $(Q1),(Q2)$ and $(Q3)$. In this setting, $\Gal(l/k)$ also acts on the $\MK(lH)$. Using this theorem, we can draw $\MK(l*G)$ as the quotient quiver of $\MK(lH)$ except the valuations of arrows. Moreover, we can exhibit valuations of $\MK(l*G)$ explicitly.

\begin{Thm}{\rm(Theorem \ref{alg})}\label{introalg}
Under the setting $(Q1),(Q2)$ and $(Q3)$, if $H\subseteq GL_d(l)$ is small, then the following assertions hold.
\begin{enumerate}
\item The McKay quiver $\MK(l*G)$ coincides with the quotient quiver of $\MK(lH)$ with respect to the action of $\Gal(l/k)$ (see Definition \ref{quot}) except for valuations$\colon$
\[\MK(l*G)=\Gal(l/k)\backslash\MK(lH)\]
\item The valuations of $\MK(l*G)$ are determined by the $a_i$'s appeared in Theorem \ref{introsab}(2) (see Theorem \ref{alg} for the details).
\end{enumerate}
\end{Thm}

In fact, more strongly, we can determine each term of $(d-1)$-almost split (fundamental) sequences of $\C(S^G)$ in terms of those of $\C(S^H)$ (see Theorem \ref{ARseqsum}). In Section \ref{Examples}, we frequently use these results to draw $\AR(\C(S^G))$.

In addition to Auslander-Reiten quivers, we give a way to determine the divisor class group $\Cl(S^G)$. A key point is that $\Cl(S^G)$ can be identified with $\{V\in\simp l*G\mid\dim_lV=1\}$.

\begin{Thm}{\rm(Lemma \ref{refaddS}, Theorem \ref{divgrp})}
Under the setting $(Q1),(Q2)$ and $(Q3)$, if $H\subseteq GL_d(l)$ is small, then the following assertions hold.
\begin{enumerate}
\item There exists a bijection $\Cl(S^G)\cong\{V\in\simp l*G\mid\dim_lV=1\}$.
\item The map
\[\{V\in\simp l*G\mid\dim_lV=1\}\to\{W\in\simp lH\mid\dim_lW=1\};V\mapsto\Res V\]
gives an injective group homomorphism $\Cl(S^G)\to\Cl(S^H)$.
\item The group $\Cl(S^H)$ is isomorphic to the group $\Hom_\mathbb{Z}(H^{\rm ab},l^\times)$ where $H^{\rm ab}=H/[H,H]$ is the Abelianization of $H$.
\end{enumerate}
\end{Thm}

For $V\in\simp l*G$, whether $\dim_lV=1$ holds or not can be determined by its decomposition law. In conclusion, we can calculate $\Cl(S^G)$ as a subgroup of $\Cl(S^H)\cong\Hom_\mathbb{Z}(H^{\rm ab},l^\times)$ by using the decomposition laws.

\begin{Ex}(Examples \ref{typeC})
Let $n\geq2$ be an integer and $l/k$ a field extension with $\ch k$ not dividing $2n$ and $[l:k]=2$. We assume $\zeta:=\zeta_{2n}\in l$, where $\zeta_{2n}$ is the primitive $2n$-th root of unity. For example, $\mathbb{C}/\mathbb{R}$ and $\mathbb{Q}(\zeta)/\mathbb{Q}(\cos\frac{\pi}{n})$ satisfy these conditions. We denote $\Gal(l/k)=\{{\rm id},\sigma\}$. Let $G:=\Big\langle\Big(\Big[\begin{smallmatrix}\zeta&0\\0&\zeta^{-1}\end{smallmatrix}\Big], {\rm id}\Big), \Big(\Big[\begin{smallmatrix}0&1\\1&0\end{smallmatrix}\Big],\sigma\Big)\Big\rangle\subseteq GL_2(l)\rtimes\Gal(l/k)$. It is easy to see that $H=\Big\langle\Big[\begin{smallmatrix}\zeta&0\\0&\zeta^{-1}\end{smallmatrix}\Big]\Big\rangle\cong C_{2n}$. Thus we can write $\simp lH=\{W_j\}_{j\in\mathbb{Z}/{2n}\mathbb{Z}}$ and draw $\AR(\C(S^H))$ as the McKay quiver of $H$ as follows.
\[\xymatrix@R=1mm{
& W_1 \ar@<0.5ex>[dl] \ar@<0.5ex>[r] & \cdots \ar@<0.5ex>[l] \ar@<0.5ex>[r] & W_{n-1} \ar@<0.5ex>[l] \ar@<0.5ex>[dr] \\
W_0 \ar@<0.5ex>[ur] \ar@<0.5ex>[dr] & & & &W_n \ar@<0.5ex>[ul] \ar@<0.5ex>[dl] \\
& W_{-1} \ar@<0.5ex>[ul] \ar@<0.5ex>[r] & \cdots \ar@<0.5ex>[l] \ar@<0.5ex>[r] & W_{-n+1} \ar@<0.5ex>[l] \ar@<0.5ex>[ur]
}\]
Here, one can show $\sigma\cdot W_j=W_{-j}$. Thus $\Gal(l/k)$ acts on $\AR(\C(S^H))$ by swapping the top and bottom and we can write $\simp l*G=\{V_0, V_{\pm1}, \cdots, V_{\pm (n-1)},V_n\}$. The decomposition law is as follows:
\[\Res V_0=W_0,\ \Res V_{\pm j}=W_j\oplus W_{-j},\ \Res V_n=W_n.\]
Therefore we can draw $\AR(\C(S^G))$ as the quotient below. In addition, $\Cl(S^G)=\{M(V_0),M(V_n)\}\cong C_2$ holds.
\[\xymatrix{
V_0 \ar@<0.5ex>[r]^-{(2,1)} & V_{\pm1} \ar@<0.5ex>[l]^-{(1,2)} \ar@<0.5ex>[r] & \cdots \ar@<0.5ex>[l] \ar@<0.5ex>[r] & V_{\pm(n-1)} \ar@<0.5ex>[l] \ar@<0.5ex>[r]^-{(1,2)} & V_n \ar@<0.5ex>[l]^-{(2,1)}
}\]

Observe that this Auslander-Reiten quiver, which is the double of type $\widetilde{C}_n$, has non-trivial valuations which do not appear in the classical case.

\end{Ex}

%%
\begin{comment}

Finally, we give a simple relationship between special Cohen-Macaulay $S^G$-modules and $S^H$-modules. 

\begin{Thm}{\rm(Theorem \ref{special})}
Let $M\in \ind(\CM S^G)$. Remark that this corresponds to an orbit with respect to the action of $\Gal(l/k)$ on $\ind(\CM S^H)$. Now $M$ is special if and only if all modules in this orbit are special as $S^H$-modules. In fact, this is equivalent to that at least one module in the orbit is special.
\end{Thm}

There is an easy algorithm to determine specials in $\CM S^H$. Thus we can also determine ones in $\CM S^G$ by using this theorem.

\end{comment}
%%

\section*{Conventions}
For a ring $A$, we denote by $\Mod A$ the category of left $A$-modules, $\mod A$ the category of finitely generated left $A$-modules, $\proj A$ the category of finitely generated projective left $A$-modules, $\fl A$ the category of left $A$-modules of finite length, $\refl A$ the category of finitely generated reflexive left $A$-modules and $\simp A$ the set of isomorphism classes of simple left $A$-modules. For a Cohen-Macaulay local ring $R$, we write $\CM R$ for the category of maximal Cohen-Macaulay $R$-modules. For a full subcategory $\D$ of an additive category $\C$, we denote by $\add\D$ the smallest full subcategory of $\C$ containing $\D$ closed under isomorphisms, finite direct sums, and direct summands. For a skeletally small Krull-Schmidt category $\C$, we write $\ind\C$ for the set of isomorphism classes of indecomposable objects of $\C$.

\section*{Acknowledgements}
The author expresses his sincere gratitude to his supervisor Osamu Iyama for his valuable advice, suggestions and patient guidance in improving this paper. He is grateful to Norihiro Hanihara for fruitful discussions and refining some proofs. He also thanks Yuji Yoshino for providing valuable information. He is indebted to Yasuaki Gyoda and Yuya Goto for their careful reading and extensive corrections of this paper. He wishes to thank Issei Enomoto for his interests and constant encouragements. This work was supported by the WINGS-FMSP program at the Graduate School of Mathematical Sciences, the University of Tokyo.

\section{Preliminaries}
%%%%%%%%%%%%%%%%%%%%%%%%%%%%%%%%%%%%%%%%%%%%%%%%%%%%%%%%%%%%%

%%%%%%%%%%%%%%%%%%%%%%%%%%%%%%%%%%%%%%%%%%%%%%%%%%%%%%%%%%%%%
\subsection{Linearization}
%%%%%%%%%%%%%%%%%%%%%%%%%%%%%%%%%%%%%%%%%%%%%%%%%%%%%%%%%%%%%

In this subsection, we make some preparations about linearization of group actions on the power series rings, which justify our setting $(Q1),(Q2),(Q3)$ and $(Q4)$ in Section \ref{ARquotient}. 

First, we see that if a finite group acts on a complete local ring, then we can take a coefficient field which is compatible with the action.

\begin{Lem}\label{coeff}
Let $(S,\mathfrak{n},l)$ be an equicharacteristic Noetherian complete local ring and $G$ a finite group acting on $S$ as a ring. Assume $|G|$ is not divided by $\ch l$. Observe that $G$ naturally acts on the residue field $l$. Then there exists a ring homomorphism $\phi\colon l\to S$ such that
\begin{enumerate}
\item The composition $l\to S\to l$ is the identity map.
\item The ring homomorphism $\phi$ commutes with the action of $G$.
\end{enumerate}
\end{Lem}
\begin{proof}
Observe that the local ring $(S^G,\mathfrak{n}\cap S^G, l^G)$ is complete. Thus we can take a coefficient field $\psi\colon l^G\to S^G$ \cite[Theorem 60]{Mat}. Since the field extension $l/l^G$ is finite Galois, $l$ is generated by one element, say $x\in l$, over $l^G$. Take the minimal polynomial $f\in l^G[T]$ of $x\in l$. Then by the Hensel's lemma, there exists a root $s\in S$ of $\psi(f)\in S^G[T]$ whose image in $l$ is $x$. We can extend $\psi\colon l^G\to S^G$ to $\phi\colon l\to S$ by the equation $\phi(x)=s$. This $\phi$ obviously satisfies the condition (1). To see the condition (2), by the uniqueness of the lift \cite[Lemma 06RR]{stacks-project}, it is sufficient to check that for any $g\in G$, $g\cdot s=\phi(g\cdot x)$ holds. Now the images of $g\cdot s$ and $\phi(g\cdot x)$, both of which are roots of $\psi(f)$, in $l$ are both $g\cdot x$. Therefore by the uniqueness of the lifting, we get the desired result.
\end{proof}

Let $G$ be a finite group acting on $S:=l[[x_1,\ldots , x_d]]$ as a ring with $|G|$ not divided by $\ch l$. If the action of $G$ on $S$ is $l$-linear, then it is known that we can linearize this action, that is, we may assume $G\subseteq GL_d(l)$ \cite{Car}. In general case, in fact, $G$ can be assumed to be a subgroup of $GL_d(l)\rtimes\Gal(l/k)$ where $k=l^G$.

\begin{Prop}\label{linear}
Let $l$ be a field, $G$ a finite group acting on $S:=l[[x_1,\ldots , x_d]]$ as a ring with $|G|$ not divided by $\ch l$. Observe that $G$ acts on the residue field $l$ naturally and put $k:=l^G$. Then there exists a group homomorphism $\phi\colon G\to GL_d(l)\rtimes\Gal(l/k)$ and a ring automorphism $\theta\colon S\to S$ such that for all $g\in G$, the following diagram commutes.
\[\xymatrix{
S \ar[r]^\theta \ar[d]_{\phi(g)} & S\ar[d]^{g}\\
S\ar[r]_\theta & S
}\]
In particular, we have a ring isomorphism $S^G\cong S^{\phi(G)}$.
\end{Prop}
\begin{proof}
Note that the group $G$ acts on $l$, which is the residue field of $S$. We denote this action by $\rho\colon G\to\Gal(l/k)$. By Lemma \ref{coeff}, we can replace $S$ by an isomorphic one so that the natural inclusion $l\to S$ commutes with the action of $G$. Let $\mathfrak{n}:=(x_1,\ldots,x_d)\subseteq S$ be the maximal ideal. Take $g\in G$. Since $g\colon S\to S$ is a $k$-algebra automorphism, it induces a $k$-linear automorphism $\mathfrak{n}/\mathfrak{n}^2\to\mathfrak{n}/\mathfrak{n}^2$. Remark that $\rho(g)$ acts on $\mathfrak{n}/\mathfrak{n}^2$ as $\rho(g)(\lambda x_i)=\rho(g)(\lambda) x_i$. Then for $\lambda\in l$ and $v\in\mathfrak{n}/\mathfrak{n}^2$,
\[(g\rho(g)^{-1})(\lambda v)=(g\rho(g)^{-1})(\lambda)(g\rho(g)^{-1})(v)=\lambda (g\rho(g)^{-1})(v)\]
holds, so we obtain $g\rho(g)^{-1}\in GL_d(l)$.

We define $\phi\colon G\to GL_d(l)\rtimes\Gal(l/k)$ by $\phi(g):=(g\rho(g)^{-1},\rho(g))$. Then it can be easily checked that $\phi$ is a group homomorphism. We define an $l$-algebra homomorphism $\theta\colon S\to S$ by
\[\theta(x_i):=\frac{1}{|G|}\sum_{g\in G}g^{-1}\phi(g)(x_i).\]
Since $\theta(x_i)\in x_i+\mathfrak{n}^2$ holds, $\theta$ is an isomorphism. Moreover, it is easy to see that $\theta(\phi(g)(x_i))=g(\theta(x_i))$ and $\theta(\phi(g)(\lambda))=g(\theta(\lambda))\ (\lambda\in l)$ hold. Thus the assertion follows.
\end{proof}

From this proposition, we can say that it is natural to assume $G$ to be a finite subgroup of $GL_d(l)\rtimes\Gal(l/k)$.

Next, we define the smallness for subgroups of $GL_d(l)$. Let $\mathfrak{n}:=(x_1,\ldots , x_d)\subseteq S$ be the maximal ideal and $U:=\mathfrak{n}/\mathfrak{n}^2$ a $d$-dimensional $l$-linear space which can be viewed as an $l$-linear subspace of $S$. For $A\in GL_d(l)$, we let
\[U_A:=\Im(A-I_d\colon U\to U)\subseteq U\]
be an $l$-linear subspace of $U$ where $I_d\in GL_d(l)$ is the unit matrix.

\begin{Def}\label{small}
We say that $A\in GL_d(l)$ is a {\it pseudo-reflection} if $\dim_lU_A\leq 1$. A finite subgroup $H\subseteq GL_d(l)$ is called {\it small} if $H$ contains no pseudo-reflection except for $I_d$.
\end{Def}

For example, all finite subgroups of $SL_d(l)$ are small. Now we exhibit a well-known fact.

\begin{Prop}\label{allref}{\rm(Chevalley-Shephard-Todd theorem)}\cite[B.27]{LW}
Let $N\subseteq GL_d(l)$ be a finite subgroup generated by pseudo-reflections. Then the ring of invariants $l[[x_1,\ldots , x_d]]^{N}$ is regular.
\end{Prop}

Using this proposition, we can prove that $G\subseteq GL_d(l)\rtimes\Gal(l/k)$ can be replaced so that $G\cap GL_d(l)\subseteq GL_d(l)$ is small. This result can be viewed as a refined version of Proposition \ref{linear}.

\begin{Thm}\label{makesmall}
Let $l$ be a field, $G$ a finite group acting on $S:=l[[x_1,\ldots , x_d]]$ as a ring with $|G|$ not divided by $\ch l$ and $k:=l^G$ the invariant subfield of the residue field $l$ of $S$. Then there exist a group homomorphism $\phi\colon G\to GL_d(l)\rtimes\Gal(l/k)$ and a ring monomorphism $\theta\colon S\to S$ such that the following conditions are satisfied.
\begin{enumerate}
\item The subgroup $\phi(G)\cap GL_d(l)\subseteq GL_d(l)$ is small.
\item The image of $\theta$ coincides with $S^{\Ker\phi}$.
\item For all $g\in G$, the following diagram commutes.
\[\xymatrix{
S \ar[r]^\theta \ar[d]_{\phi(g)} & S\ar[d]^{g}\\
S\ar[r]_\theta & S
}\]
\end{enumerate}
In particular, we have a ring isomorphism $S^G\cong S^{\phi(G)}$.
\end{Thm}
\begin{proof}
We show by induction on $|G|$. By Proposition \ref{linear}, we may assume $G\subseteq GL_d(l)\rtimes \Gal(l/k)$. Let $N:=\langle(A, {\rm id})\in G\mid A\in GL_d(l) \text{ is a pseudo-reflection}\rangle$ be a group generated by all pseudo-reflections in $G\cap GL_d(l)$. Then $N$ is a normal subgroup of $G$. If $|N|=1$, there is nothing to prove. We assume $|N|>1$. By Lemma \ref{allref}, the ring of invariants $S^N$ is regular, so there exist $f_1,\cdots, f_d\in S$ such that $S^{N}=l[[f_1,\ldots , f_d]]$ holds. Observe that $G/N$ acts on $l[[f_1,\ldots , f_d]]$ as a $k$-algebra. Since $|G/N|<|G|,$ the assertion follows by the induction hypothesis.
\end{proof}

%%%%%%%%%%%%%%%%%%%%%%%%%%%%%%%%%%%%%%%%%%%%%%%%%%%%%%%%%%%%%
\subsection{Ramification theory of invariant subrings}
%%%%%%%%%%%%%%%%%%%%%%%%%%%%%%%%%%%%%%%%%%%%%%%%%%%%%%%%%%%%%

In this subsection, we review some basic facts on rings of invariants, which will be used to prove Theorems \ref{Auslanderalgebra} and \ref{invisol}, for the convenience of the reader. Throughout this section, we consider the following settings.
\begin{enumerate}
\item[(I1)] $A$ is a commutative Noetherian normal domain.
\item[(I2)]  $K:=\Frac(A)$ and $L/K$ is a finite Galois extension.
\item[(I3)]  $B$ is the integral closure of $A$ in $L$.
%\item[(I1)] $B$ is a commutative Noetherian normal domain.
%\item[(I2)]  $G\subseteq{\rm Aut^{al}}(B)$ is a finite subgroup.
%\item[(I3)]  $A:=B^G$ is the ring of invariants.
\end{enumerate}
Then $B$ is finitely generated as an $A$-module and thus a Noetherian normal domain. Put $G:=\Gal(L/K)$. In this setting, we have a ring homomorphism
\[\gamma\colon B*G\to\End_A(B);\ b*g\mapsto(b'\mapsto bg(b')).\]
First, we study sufficient conditions for $\gamma$ to be an isomorphism. Recall that the extension $A\subseteq B$ is called {\it separable} if $B$ is projective as a $B\otimes_AB$-module.

\begin{Prop}\label{sep}
If the extension $A\subseteq B$ is separable, then $\gamma$ is an isomorphism.
\end{Prop}
\begin{proof}
The surjectivity follows from \cite[3.6]{IT}. The injectivity also follows since the ranks of $B*G$ and $\End_A(B)$ are equal.
\end{proof}

We recall the definition of unramifiedness.

\begin{Def}
\begin{enumerate}
\item Take $\mathfrak{P}\in\Spec B$ and $\mathfrak{p}:=\mathfrak{P}\cap A\in\Spec A$. We say the extension $B/A$ is {\it unramified} at $\mathfrak{P}$ if the following two conditions hold.
\begin{enumerate}
\item $\mathfrak{p}B_\mathfrak{P}=\mathfrak{P}B_\mathfrak{P}$ holds.
\item The extension $\kappa(\mathfrak{P})/\kappa(\mathfrak{p})$ is separable.
\end{enumerate}
Here, $\kappa(\mathfrak{p})$ (respectively $\kappa(\mathfrak{P})$) denotes the residue field of the local ring $A_\mathfrak{p}$ (respectively $B_\mathfrak{P}$).
\item We say the extension $B/A$ is {\it unramified} if it is unramified at all $\mathfrak{P}\in\Spec B$.
\end{enumerate}
\end{Def}

Combining with Lemma \ref{s2isom}, we obtain the following immediately.

\begin{Thm}\label{end}
Under the settings $(I1),(I2)$ and $(I3)$, if the extension $A\subseteq B$ is unramified in codimension one, then $\gamma$ is an isomorphism.
\end{Thm}
\begin{proof}
Take $\mathfrak{p}\in\Spec A$ with $\height\mathfrak{p}=1$. Since the extension $A_\mathfrak{p}\subseteq B_\mathfrak{p}$ is unramified, it is separable (see, for example, \cite[3.8]{IT}). Thus $\gamma_\mathfrak{p}\colon B_\mathfrak{p}*G\to\End_{A_{\mathfrak{p}}}(B_\mathfrak{p})$ is an isomorphism by Proposition \ref{sep}. Therefore $\gamma$ is an isomorphism by Lemma \ref{s2isom}.
\end{proof}

Second, we see how we can judge whether the extension $A\subseteq B$ or its localization is unramified. Let $\mathfrak{P}\in\Spec B$ and $\mathfrak{p}:=\mathfrak{P}\cap A\in\Spec A$. We let $\mathfrak{P}=\mathfrak{P}_1,\ldots ,\mathfrak{P}_n$ be all the prime ideals of $S$ lying over $\mathfrak{p}$. We define the {\it decomposition group} $D(\mathfrak{P})$ and the {\it inertia group} $T(\mathfrak{P})$.
\[D(\mathfrak{P}):=\{g\in G\mid g(\mathfrak{P})=\mathfrak{P}\},\ T(\mathfrak{P}):=\{g\in G\mid g(x)-x\in\mathfrak{P}\ \text{ for all }x\in B )\}\] 
Let $f:=[\kappa(\mathfrak{P}):\kappa(\mathfrak{p})]$ and $e:={\rm length}_{B_\mathfrak{P}}(B_\mathfrak{P}/\mathfrak{p}B_\mathfrak{P})$. Now we have the following standard statements.

\begin{Prop}\cite{Bou,S}
\begin{enumerate}
\item The action of $G$ on the set $\{\mathfrak{P}_1,\ldots ,\mathfrak{P}_n\}$ is transitive. In particular, $n=[G:D(\mathfrak{P})]$ holds. 
\item The field extension $\kappa(\mathfrak{P})/\kappa(\mathfrak{p})$ is normal.
\item The natural group homomorphism $D(\mathfrak{P})\to{\rm Aut}_{\kappa(\mathfrak{p})}^{\rm al}(\kappa(\mathfrak{P}))$ is surjective.
\item The equation $\dim_{\kappa(\mathfrak{p})}(B_\mathfrak{p}/\mathfrak{p}B_\mathfrak{p})=nef$ holds.
\end{enumerate}
\end{Prop} 

\begin{comment}
\begin{proof}
Suppose that there exists no $g\in G$ with $\mathfrak{P}_1=g(\mathfrak{P}_2)$. Since the extension $A\subseteq B$ is integral, $\mathfrak{P}_1\nsubseteq g(\mathfrak{P}_2)\ (\forall g\in G)$ holds. By the prime avoidance, there exists $x\in \mathfrak{P}_1$ such that $x\notin g(\mathfrak{P}_2)\ (\forall g\in G)$ holds. Then the element $\prod_{g\in G}g(x)$ does not belong to $\mathfrak{P}_2$, but belongs to $\mathfrak{P}_1\cap A=\mathfrak{p}$, which is a contradiction.
\end{proof}
\end{comment}

Remark that the inertia group $T(\mathfrak{P})$ is the kernel of the group homomorphism $D(\mathfrak{P})\to{\rm Aut}_{\kappa(\mathfrak{p})}^{\rm al}(\kappa(\mathfrak{P}))$. Let ${\kappa(\mathfrak{p})}_s$ be the separable closure of ${\kappa(\mathfrak{p})}$ in ${\kappa(\mathfrak{P})}$ and put $f_0:=[\kappa(\mathfrak{p})_s:\kappa(\mathfrak{p})], q:=[\kappa(\mathfrak{P}):\kappa(\mathfrak{p})_s]$. Then we have the equations $f=f_0q$ and $[D(\mathfrak{P}):T(\mathfrak{P})]=f_0$.

\begin{Prop}\label{whenfree}
We have inequalities
\[|G|\leq nef, |D(\mathfrak{P})|\leq ef\text{ and }|T(\mathfrak{P})|\leq eq.\]
The equalities hold if and only if $B_\mathfrak{p}$ is $A_\mathfrak{p}$-free.
\end{Prop}
\begin{proof}
By $[G:D(\mathfrak{P})]=n$ and $[D(\mathfrak{P}):T(\mathfrak{P})]=f_0$, these three inequalities are equivalent to each other. Thus we only remark on the first one. Since $\dim_{\kappa(\mathfrak{p})}(B_\mathfrak{p}/\mathfrak{p}B_\mathfrak{p})=nef$ holds, we have a surjection $A_\mathfrak{p}^{\oplus nef}\to B_\mathfrak{p}$ by the Nakayama's lemma. This extends to a surjection $K^{\oplus nef}\to L$. By comparing the dimensions over $K$, we obtain $|G|\leq nef$. The equality part also follows from similar arguments. 
\end{proof}

Here we exhibit the following fundamental statement. Remark that this is nothing but \cite[I.4.10]{S} when $\height\mathfrak{p}=1$.

\begin{Thm}\label{inertia}\cite[3.2.5(i)]{Fu}
Under the settings $(I1),(I2)$ and $(I3)$, the following conditions are equivalent.
\begin{enumerate}
\item The extension $A\subseteq B$ is unramified at $\mathfrak{P}$.
\item $|T(\mathfrak{P})|=1$ holds.
\end{enumerate}
If this is the case, then $B_\mathfrak{p}$ is $A_\mathfrak{p}$-free.
\end{Thm}
\begin{proof}
If the extension $A\subseteq B$ is unramified at $\mathfrak{P}$, then we have $e=1$ and $q=1$. By the inequality in Proposition \ref{whenfree}, we obtain $|T(\mathfrak{P})|=1$. By the same corollary, $B_\mathfrak{p}$ is $A_\mathfrak{p}$-free. The other direction is \cite[3.2.5(i)]{Fu}.
\end{proof}

Finally, we see that if the extension $A\subseteq B$ is unramified at $\mathfrak{P}$, then $A_\mathfrak{p}$ is regular if and only if so is $B_\mathfrak{P}$.

\begin{Thm}\label{regular}
Under the settings $(I1),(I2)$ and $(I3)$, we assume the extension $A\subseteq B$ is unramified at $\mathfrak{P}$. Then the following conditions are equivalent.
\begin{enumerate}
\item The local ring $B_\mathfrak{P}$ is regular.
\item The local ring $A_\mathfrak{p}$ is regular.
\end{enumerate}
\end{Thm}
\begin{proof}
(1)$\Rightarrow$(2) By Theorem \ref{inertia}, we know $B_\mathfrak{p}$ is $A_\mathfrak{p}$-free. Thus we obtain $\gl A_\mathfrak{p}\leq\gl B_\mathfrak{p}<\infty$.

(2)$\Rightarrow$(1) The ideal $\mathfrak{P}B_\mathfrak{P}=\mathfrak{p}B_\mathfrak{P}$ is generated by elements whose number is $\height\mathfrak{p}=\height\mathfrak{P}$, which means that $B_\mathfrak{P}$ is regular.
\end{proof}

%%%%%%%%%%%%%%%%%%%%%%%%%%%%%%%%%%%%%%%%%%%%%%%%%%%%%%%%%%%%%
\section{Auslander-Reiten theory for NCCRs}\label{ARCT}
%%%%%%%%%%%%%%%%%%%%%%%%%%%%%%%%%%%%%%%%%%%%%%%%%%%%%%%%%%%%%

In this section, we prove the existence of higher almost split sequences for NCCRs (Theorem \ref{existAR}, Proposition \ref{ARforCT}). Toward this, we recall basic concepts in Cohen-Macaulay representations.

Let $(R,\mathfrak{m})$ be a $d$-dimensional Cohen-Macaulay local ring having the canonical module $\omega\in\CM R$. An $R$-{\it order} is a module-finite $R$-algebra $\Lambda$ satisfying $\Lambda\in\CM R$. We define $\CM\Lambda:=\{M\in\mod\Lambda\mid M_R\in\CM R\}$. In the branch of Cohen-Macaulay representations, the categorical structure of the category of Cohen-Macaulay modules $\CM\Lambda$ over orders is the main interest. We let
\[(-)^*:=\Hom_\Lambda(-,\Lambda),(-)^\vee:=\Hom_R(-,\omega)\colon\mod\Lambda\to\mod\Lambda^{\op}.\]
Remark that $(-)^\vee\colon\CM\Lambda\to\CM\Lambda^{\op}$ gives an anti-equivalence. 

%%%%%%%%%%%%%%%%%%%%%%%%%%%%%%%%%%%%%%%%%%%%%%%%%%%%%%%%%%%%%
\subsection{Preliminary on orders}
%%%%%%%%%%%%%%%%%%%%%%%%%%%%%%%%%%%%%%%%%%%%%%%%%%%%%%%%%%%%%

In this subsection, we summarize basic facts about orders for the convenience of the reader. First, we introduce some classes of orders using homological dimensions.

\begin{Def}
\begin{enumerate}
\item An $R$-order $\Lambda$ is called {\it regular} if $\gl\Lambda=d$ holds.
\item An $R$-order $\Lambda$ has {\it isolated singularities} if $\Lambda_\mathfrak{p}$ is a regular $R_\mathfrak{p}$-order for all $\mathfrak{p}\in\Spec R\backslash\{\mathfrak{m}\}$.
\item An $R$-order $\Lambda$ is called {\it Gorenstein} if $\id\Lambda_\Lambda=\id{}_\Lambda\Lambda=d$ holds.
\end{enumerate}
\end{Def}

Remark that $R$ itself is a regular (respectively, Gorenstein) $R$-order if and only if $R$ is a regular (respectively, Gorenstein) local ring. We have a useful criterion for $\Lambda$ to be regular. This is a consequence of Auslander-Buchsbaum formula.

\begin{Prop}\cite[2.17]{IW}\label{regproj}
An $R$-order $\Lambda$ is regular if and only if $\CM\Lambda=\proj\Lambda$ holds.
\end{Prop}

Next, we see that Gorenstein orders satisfy a condition called {\it Gorenstein condition}. We give a categorical proof using the concept of cotilting bimodules (see \cite{Mi}).

\begin{Prop}\cite[4.13]{GN}\label{Gorcond}
Let $\Lambda$ be a Gorenstein $R$-order. For $S\in\simp\Lambda$, we have
\[\Ext^i_\Lambda(S,\Lambda)=0\ (i\neq d),\ \Ext^d_\Lambda(S,\Lambda)\in\simp\Lambda^{\op}.\]
\end{Prop}
\begin{proof}
By $\depth_R\Lambda=d$, we have $\Ext^i_\Lambda(S,\Lambda)=0\ (i<d)$. Since $\id\Lambda_\Lambda=d$ holds, we have $\Ext^i_\Lambda(S,\Lambda)=0\ (i>d)$. Now since $\Lambda$ is a cotilting $(\Lambda,\Lambda)$-bimodule, $\RHom_{\Lambda}(-,\Lambda)\colon D^b(\mod\Lambda)\to D^b(\mod\Lambda^{\op})$ gives an anti-equivalence (see \cite[2.11]{Mi}). By the discussion above, we have $\RHom_{\Lambda}(S,\Lambda)\in\fl\Lambda^{\op}[-d]$. Thus $\RHom_{\Lambda}(-,\Lambda)$ restricts to an anti-equivalence $\fl\Lambda\to\fl\Lambda^{\op}[-d]$. Therefore we obtain an anti-equivalence $\Ext^d_\Lambda(-,\Lambda)\colon\fl\Lambda\to\fl\Lambda^{\op}$ and the assertion follows.
\end{proof}

Here, we see a useful basic lemma which will be used repeatedly.

\begin{Lem}\label{homdepth}
Let $X,Y\in\mod\Lambda$. If $\depth_RY\geq2$, then $\depth_R(\Hom_\Lambda(X,Y))\geq2$ holds.
\end{Lem}
\begin{proof}
Let $P_1\to P_0\to X\to 0$ be a projective presentation of $X$ in $\mod\Lambda$. Then we have an exact sequence $0\to\Hom_\Lambda(X,Y)\to\Hom_\Lambda(P_0,Y)\to\Hom_\Lambda(P_1,Y)$. Since $\depth_R(\Hom_\Lambda(P_i,Y))\geq2$ holds, the assertion follows by using the depth lemma.
\end{proof}

In the rest, we assume $d=\dim R\geq2$. We consider the following $({\rm R}_1)$-condition for $\Lambda$.
\[({\rm R}_1)\colon\Lambda_\mathfrak{p}\text{ is a regular }R_\mathfrak{p}\text{-order for all }\mathfrak{p}\in\Spec R\ \text{with}\height\mathfrak{p}\leq1\]
For example, if $\Lambda=R$, then $R$ satisfies the $({\rm R}_1)$-condition if and only if $R$ is normal (since $R$ is Cohen-Macaulay with $d\geq2$ by our assumption). Observe that if $\Lambda$ satisfies the $({\rm R}_1)$-condition, then for all $M\in\refl\Lambda$ and $\mathfrak{p}\in\Spec R$ with $\height\mathfrak{p}\leq1$, we have $M_\mathfrak{p}\in\proj\Lambda_\mathfrak{p}$ by Proposition \ref{regproj}.

Next, we consider the following $({\rm S}_n)$-condition for a module $X\in\mod R$.
\[({\rm S}_n)\colon\depth_{R_\mathfrak{p}}X_\mathfrak{p}\geq\min\{n,\height\mathfrak{p}\}\text{ for all }\mathfrak{p}\in\Spec R\]
We have the following useful lemma which states that whether a homomorphism between certain modules is an isomorphism is determined only by codimension-one information.

\begin{Lem}\cite[5.11]{LW}\label{s2isom}
Let $R$ be a commutative Noetherian ring and $f\colon M\to N$ a morphism in $\mod R$. We assume that $M$ satisfies $({\rm S}_2)$ and that $N$ satisfies $({\rm S}_1)$. If $f_\mathfrak{p}\colon M_\mathfrak{p}\to N_\mathfrak{p}$ is an isomorphism for all $\mathfrak{p}\in\Spec R$ with $\height\mathfrak{p}\leq1$, then $f$ is an isomorphism.
\end{Lem}

Using this lemma, we can prove the following basic facts for orders satisfying the $({\rm R}_1)$-condition. Remark that this (2) says that reflexive $\Lambda$-modules are also reflexive with respect to the canonical dual $(-)^\vee=\Hom_R(-,\omega)$.

\begin{Prop}
Let $\Lambda$ be an $R$-order satisfying the $({\rm R}_1)$-condition.
\begin{enumerate}
\item We have $\CM\Lambda\subseteq\refl\Lambda$.
\item For $X\in\refl\Lambda$, we have $X^\vee\in\refl\Lambda$ and $X\cong X^{\vee\vee}$.
\end{enumerate}
\end{Prop}
\begin{proof}
(1) Take $X\in\CM\Lambda$. We consider the evaluation map $f\colon X\to X^{**}$. For $\mathfrak{p}\in\Spec R$ with $\height\mathfrak{p}\leq1$, $f_\mathfrak{p}$ is an isomorphism since $X_\mathfrak{p}\in\proj\Lambda_\mathfrak{p}$ holds. Remark that not only $X$ but also $X^{**}$ satisfies $({\rm S}_2)$ by Lemma \ref{homdepth}. Thus $f$ is an isomorphism by Lemma \ref{s2isom}.

(2) Take $X\in\refl\Lambda$. First, we consider the evaluation map $g\colon X^\vee\to X^{\vee**}$. By the same argument as in (1), we can see that $g$ is an isomorphism. Thus $X^\vee\in\refl\Lambda$ holds. In the same way, we can see that the evaluation map $X\to X^{\vee\vee}$ is also an isomorphism.
\end{proof}

Thanks to this proposition, we can define the {\it Nakayama functor}
\[\nu:=(-)^{*\vee}\colon\refl\Lambda\longrightarrow\refl\Lambda.\]
This is a categorical equivalence with quasi-inverse $\nu^-:=(-)^{\vee*}$. Finally, we see the following lemma which plays a crucial role in the next subsection.

\begin{Lem}\cite[3.3]{DITW}\label{homcheck}
Let $\Lambda$ be an $R$-order satisfying the $({\rm R}_1)$-condition. For $X,Y\in\refl\Lambda$, we have
\[\Hom_\Lambda(X,\nu Y)\cong\Hom_\Lambda(Y,X)^\vee.\]
\end{Lem}
\begin{proof}
We give a complete proof here since our setting is slightly more general than \cite[3.3]{DITW}. Observe that $\Hom_\Lambda(X,\nu Y)\cong(X\otimes_\Lambda Y^*)^\vee$ holds. We have a natural $R$-homomorphism $\phi\colon X\otimes_\Lambda Y^*\to\Hom_\Lambda(Y,X);x\otimes f\mapsto(y\mapsto xf(y))$. For $\mathfrak{p}\in\Spec R$ with $\height\mathfrak{p}\leq1$, $\phi_\mathfrak{p}$ is an isomorphism since $X_\mathfrak{p},Y_\mathfrak{p}\in\proj\Lambda_\mathfrak{p}$ holds. Since $Z^\vee\in\mod R$ satisfies $({\rm S}_2)$-condition for all $Z\in\mod R$ by Lemma \ref{homdepth}, we can conclude that $\phi^\vee\colon\Hom_\Lambda(Y,X)^\vee\to(X\otimes_\Lambda Y^*)^\vee$ is an isomorphism by Lemma \ref{s2isom}.
\end{proof}

%%%%%%%%%%%%%%%%%%%%%%%%%%%%%%%%%%%%%%%%%%%%%%%%%%%%%%%%%%%%%
\subsection{Auslander-Reiten theory for NCCRs}
%%%%%%%%%%%%%%%%%%%%%%%%%%%%%%%%%%%%%%%%%%%%%%%%%%%%%%%%%%%%%

In the rest of this section, we assume $d=\dim R\geq2$ and $R$ is normal. We introduce the concept of non-commutative crepant resolutions (NCCR).

\begin{Def}\cite[4.1]{VdB}\cite{IR}\label{defNCCR}
Let $\Lambda$ be an $R$-order and $M\in\refl\Lambda$. We say $\End_\Lambda(M)$ is an {\it NCCR} of $\Lambda$ or $M$ gives an {\it NCCR} of $\Lambda$ if the following two conditions are satisfied.
\begin{enumerate}
\item The ring $\End_\Lambda(M)$ is a regular $R$-order.
\item The module $M_\mathfrak{p}\in\mod\Lambda_\mathfrak{p}$ is a generator for all $\mathfrak{p}\in\Spec R$ with $\height\mathfrak{p}\leq1$.
\end{enumerate}
\end{Def}

Observe that orders having an NCCR must satisfy the $({\rm R}_1)$-condition.

\begin{Prop}\cite[8.1]{IR}\label{NCCRR1}
If an $R$-order $\Lambda$ has an NCCR, then $\Lambda$ satisfies the $({\rm R}_1)$-condition.
\end{Prop}

Before discussing NCCR, we give a quick review on functor categories. For a preadditive category $\C$, a {\it left module} over $\C$ means a covariant additive functor from $\C$ to the category of Abelian groups. We denote by $\Mod\C$ the Abelian category of left modules over $\C$.

\begin{Prop}\cite[2.5]{Au74}\label{projectivization}
Let $\C$ be an additive category having an additive generator $X\in\C$. Then the substitution functor
\[{\rm ev}_X\colon\Mod\C^{\op}\to\Mod\End_\C(X)^{\op}\]
is a categorical equivalence. Moreover, if $\C$ is idempotent complete, then this restricts to an equivalence
\[\C(X,-)\colon\C\to\proj\End_\C(X)^{\op}.\]
\end{Prop}

Next, we give a definition of higher almost-split sequences.

\begin{Def}\label{ARseq}
Let $\C$ be a Krull-Schmidt category.
\begin{enumerate}
\item A morphism $Z\to X$ in $\C$ is called a {\it sink map} if it induces a projective cover $\C(-,Z)\to J_\C(-,X)$ in $\Mod\C^{\op}$. A complex $\cdots\to C_1\to C_0\to X$ in $\C$ is called {\it a sink sequence of $X$} if $\cdots\to\C(-,C_1)\to\C(-,C_0)\to J_\C(-,X)\to0$ gives a minimal projective resolution of $J_\C(-,X)\in\Mod\C^{\op}$.
\item Dually, a morphism $Y\to Z$ in $\C$ is called a {\it source map} if it induces a projective cover $\C(Z,-)\to J_\C(Y,-)$ in $\Mod\C$. A complex $Y\to C^0\to C^1\to\cdots$ in $\C$ is called {\it a source sequence of $Y$} if $\cdots\to\C(C^1,-)\to\C(C^0,-)\to J_\C(Y,-)\to0$ gives a minimal projective resolution of $J_\C(Y,-)\in\Mod\C$.
\item A complex $Y\to C_{d-2}\to\cdots\to C_0\to X$ in $\C$ is called a {\it $(d-1)$-almost split sequence} if it is a sink sequence of $X$ and a source sequence of $Y$ simultaneously. 
\end{enumerate}
%An exact sequence $0\to Y\to C_{d-2}\to\cdots\to C_0\to X\to 0$ (respectively, $0\to Y\to C_{d-2}\to\cdots\to C_0\to X$) in $\C$ is called a {\it $(d-1)$-almost split sequence} (respectively, {\it $(d-1)$-fundamental sequence}) if it is a sink sequence of $X$ and a source sequence of $Y$ simultaneously.
\end{Def}

Remark that these sequences are unique up to isomorphisms if exist.

From now on, we assume that an $R$-order $\Lambda$ has a module $M\in\refl\Lambda$ giving an NCCR. Then by Proposition \ref{NCCRR1}, $\Lambda$ satisfies the $({\rm R}_1)$-condition. Let $\C:=\add M\subseteq\refl\Lambda$ and $\Gamma:=\End_\Lambda(M)$. Observe that the functors ${\rm ev}_M\colon\Mod\C\to\Mod\Gamma$ and ${\rm ev}_M\colon\Mod\C^{\op}\to\Mod\Gamma^{\op}$ are categorical equivalences by Proposition \ref{projectivization}. First, we see that the Nakayama functor gives an auto-equivalence of $\C$.

\begin{Prop}\cite[3.4]{DITW}
The Nakayama functor $\nu\colon\refl\Lambda\to\refl\Lambda$ restricts to an auto-equivalence $\nu\colon\C\to\C$.
\end{Prop}
\begin{proof}
Observe that the functor $F:=\Hom_\Lambda(M,-)\colon\refl\Lambda\to\refl\Gamma^{\op}$ is a categorical equivalence (see \cite[2.4]{IR}\cite[6.2]{Han24b}). Take $X\in\C$. Now as an $R$-module, we have $F(\nu X)\cong\Hom_\Lambda(X,M)^\vee\in\CM R$ by Lemma \ref{homcheck}. This means $F(\nu X)\in\CM\Gamma^{\op}=\proj\Gamma^{\op}$ by Proposition \ref{regproj}. Since the equivalence $F\colon\refl\Lambda\to\refl\Gamma^{\op}$ restricts to an equivalence $F\colon\add M\to\proj\Gamma^{\op}$ by Proposition \ref{projectivization}, we obtain $\nu X\in\C$.
\end{proof}

In the rest, we assume that $R$ is complete local. We prove the following main theorem in this section which is a partial generalization of \cite[3.3.1,3.4.4]{Iya07b}. The proof is very different from the classical case even when $d=2$.

\begin{Thm}\label{existAR}
Let $R$ be a Cohen-Macaulay complete local normal domain with $d:=\dim R\geq2$, $\Lambda$ an $R$-order and $M\in\refl\Lambda$ with $\Gamma:=\End_\Lambda(M)$ is an NCCR. We let $\C:=\add M$. Then for any $X\in\ind\C$, there exists a $(d-1)$-almost split sequence in $\C$ of the following form.
\[\nu X\longrightarrow C_{d-2}\longrightarrow\cdots\longrightarrow C_0\longrightarrow X.\]
\end{Thm}
\begin{proof}
Let $Y\longrightarrow C_{d-2}\longrightarrow\cdots\longrightarrow C_0\longrightarrow X$ be a complex which induces a minimal projective resolution
\begin{gather}
0\longrightarrow\C(-,Y)\longrightarrow\C(-,C_{d-2})\longrightarrow\cdots\longrightarrow\C(-,C_0)\longrightarrow\C(-,X)\longrightarrow S_X\longrightarrow0\label{simminproj}
\end{gather}
of $S_X:=\C(-,X)/J_\C(-,X)\in\simp\C^{\op}$. Remark that this sequence exists since $\gl\Gamma=d$ holds and $\Gamma$ is semiperfect by the completeness of $R$. By Proposition \ref{Gorcond}, we have
\[\Ext^i_{\C^{\op}}(S_X,\C)=0\ (i\neq d),\ \Ext^d_{\C^{\op}}(S_X,\C)\in\simp\C.\]
Thus we obtain the following exact sequence.
\[0\longrightarrow\C(X,-)\longrightarrow\C(C_0,-)\longrightarrow\cdots\longrightarrow\C(C_{d-2},-)\longrightarrow\C(Y,-)\longrightarrow\Ext^d_{\C^{\op}}(S_X,\C)\longrightarrow0\]
This is the minimal projective resolution of $\Ext^d_{\C^{\op}}(S_X,\C)\in\mod\C$. Thus if we can prove $\Ext^d_{\C^{\op}}(S_X,\C)\cong\C(\nu X,-)/J_\C(\nu X,-)$, then $Y\cong\nu X$ holds and so the assertion follows. Since $\Ext^d_{\C^{\op}}(S_X,\C)\in\simp\C$ holds, it is enough to show $\Ext^d_{\C^{\op}}(S_X,\C(-,\nu X))\neq0$. If we rewrite the minimal projective resolution \eqref{simminproj} of $S_X$ as $\C(-,D_\bullet)\longrightarrow S_X\longrightarrow0$, we have
\begin{equation*}
\begin{split}
\Ext^d_{\C^{\op}}(S_X,\C(-,\nu X))&\cong H^d(\Hom_{\C^{\op}}(\C(-,D_\bullet),\C(-,\nu X)))\cong H^d(\C(D.,\nu X))\\
&\cong H^d(\C(X,D_\bullet)^\vee)\ \text{($\because$ Lemma \ref{homcheck})}\\
&\cong\Ext^d_R(\C(X,X)/J_\C(X,X),\omega)\ (\because\C(X,D_\bullet)\in\CM R).
\end{split}
\end{equation*}
Since $\C(X,X)/J_\C(X,X)\in\fl R$ holds, we have $\Ext^d_R(\C(X,X)/J_\C(X,X),\omega)\neq0$. This completes the proof.
\end{proof}

By the uniqueness of sink and source sequences, we obtain the following corollary.

\begin{Cor}\label{sinksource}
Let $Y\to C_{d-2}\to\cdots\to C_0\to X$ be a complex in $\C$ with $X\in\ind\C$. Then it is a sink sequence of $X$ if and only if it is a source sequence of $Y$ if and only if it is a $(d-1)$-almost split sequence. If these conditions hold, then we have $Y=\nu X$.
\end{Cor}

%%%%%%%%%%%%%%%%%%%%%%%%%%%%%%%%%%%%%%%%%%%%%%%%%%%%%%%%%%%%%
\subsection{CT modules}
%%%%%%%%%%%%%%%%%%%%%%%%%%%%%%%%%%%%%%%%%%%%%%%%%%%%%%%%%%%%%

In this section, we introduce the concept of CT modules as \cite[5.1]{IW}. Later, we see that every CT module gives an NCCR (Theorem \ref{CTEnd}(1)).

\begin{Def}\cite[5.1]{IW}\label{CTdef}
Let $R$ be a Cohen-Macaulay local normal domain and $\Lambda$ an $R$-order. A module $M\in\CM\Lambda$ is called a {\it CT module} if the following holds.
\[\add M=\{X\in\CM\Lambda\mid\Hom_\Lambda(M,X)\in\CM R\}=\{X\in\CM\Lambda\mid\Hom_\Lambda(X,M)\in\CM R\}\]
\end{Def}

Remark that if $M\in\CM\Lambda$ is a CT module, then $\add M$ contains $\Lambda$ and $\Lambda^\vee$. When $d=2$, a module $M\in\CM\Lambda$ is CT if and only if it is an additive generator of $\CM\Lambda$ by Lemma \ref{homdepth}.

If $\Lambda$ is an isolated singularity, then the concept of CT modules coincides with that of $(d-1)$-cluster tilting modules \cite[2.2]{Iya07b}, i.e. a module $M\in\CM\Lambda$ satisfying
\begin{equation*}
\begin{split}
\add M&=\{X\in\CM\Lambda\mid\Ext_\Lambda^i(M,X)=0\ (1\leq\forall i\leq d-2)\}\\&=\{X\in\CM\Lambda\mid\Ext_\Lambda^i(X,M)=0\ (1\leq\forall i\leq d-2)\}.
\end{split}
\end{equation*}

\begin{Prop}\cite[2.5.1]{Iya07b}
Assume that $\Lambda$ is an isolated singularity. For $X,Y\in\CM\Lambda$, the condition $\Hom_\Lambda(X,Y)\in\CM R$ is equivalent to $\Ext_\Lambda^i(X,Y)=0\ (1\leq\forall i\leq d-2)$. In particular, a module $M\in\CM\Lambda$ is CT if and only if it is $(d-1)$-cluster tilting.
\end{Prop}

We see a characterization of CT modules by a homological property of its endomorphism ring. In particular, we can say that every CT module gives an NCCR.

\begin{Thm}\cite[5.4]{IW}\label{CTEnd}
A module $M\in\CM\Lambda$ is CT if and only if $\Lambda\in\add M$ holds and $\End_\Lambda(M)$ is a regular $R$-order.
\end{Thm}

As a summary, we obtain the following implications for a reflexive module.
\[
\xymatrix@C=25mm{
(d-1)\text{-cluster tilting} \ar@{<=>}[r]^-{\Lambda\colon\text{isolated singularity}} & \text{CT} \ar@{<=>}[r] & \text{NCCR+generator}
}\]

Now we assume that $R$ is complete local and $M\in\CM\Lambda$ is a CT module. Let $\C:=\add M$ and $X\in\ind\C$. By Theorem \ref{CTEnd}, we can apply Theorem \ref{existAR} and thus there exists a $(d-1)$-almost split sequence $\nu X\to C_{d-2}\to\cdots\to C_0\to X$ in $\C$. In this situation, we can deduce the exactness of this complex.

\begin{Prop}\label{ARforCT}
Let $\nu X\overset{f_{d-1}}{\longrightarrow}C_{d-2}\overset{f_{d-2}}{\longrightarrow}\cdots\overset{f_1}{\longrightarrow}C_0\overset{f_0}{\longrightarrow}X$ be a $(d-1)$-almost split sequence in $\C$.
\begin{enumerate}
\item If $X\notin\proj\Lambda$, then the following complex is exact.
\[0\longrightarrow\nu X\overset{f_{d-1}}{\longrightarrow}C_{d-2}\overset{f_{d-2}}{\longrightarrow}\cdots\overset{f_1}{\longrightarrow}C_0\overset{f_0}{\longrightarrow}X\longrightarrow 0.\]
\item If $X\in\proj\Lambda$, then the following complex is exact.
\[0\longrightarrow\nu X\overset{f_{d-1}}{\longrightarrow}C_{d-2}\overset{f_{d-2}}{\longrightarrow}\cdots\overset{f_1}{\longrightarrow}C_0\overset{f_0}{\longrightarrow}X\longrightarrow X/\rad X\longrightarrow0.\]
\end{enumerate}
\end{Prop}
\begin{proof}
We obtain the assertions by substituting $\Lambda\in\C$ for the exact sequence
\[0\longrightarrow\C(-,\nu X)\longrightarrow\C(-,C_{d-2})\longrightarrow\cdots\longrightarrow\C(-,C_0)\to J_\C(-,X)\longrightarrow0.\qedhere\]
\end{proof}

In this setting, the exact sequence $0\to\nu X\to C_{d-2}\to\cdots\to C_0\to X\to0$ is called $(d-1)$-almost split sequence when $X\notin\proj\Lambda$. If $X\in\proj\Lambda$, then the exact sequence $0\to\nu X\to C_{d-2}\to\cdots\to C_0\to X$ is called {\it $(d-1)$-fundamental sequence} (see \cite[3.1]{Iya07b}).

%%%%%%%%%%%%%%%%%%%%%%%%%%%%%%%%%%%%%%%%%%%%%%%%%%%%%%%%%%%%%
\subsection{Auslander-Reiten quivers}
%%%%%%%%%%%%%%%%%%%%%%%%%%%%%%%%%%%%%%%%%%%%%%%%%%%%%%%%%%%%%

In this subsection, we assume that $R$ is complete local. We introduce the Auslander-Reiten quivers for NCCRs, which play an important role in the rest of this paper. 

\begin{Def}\cite[3.11]{INP}\label{valquiv}
\begin{enumerate}
\item A {\it valued quiver} is a triple $Q=(Q_0,d,d')$ consisting of a set $Q_0$ and maps $d,d'\colon Q_0\times Q_0\to\mathbb{Z}_{\geq0}\cup\{\infty\}$. We often visualize $Q$ by regarding elements of $Q_0$ as vertices and drawing a valued arrow $\xymatrix@C=5em{X \ar^{(d_{XY},d'_{XY})}[r] & Y}$ for each $(X,Y)\in Q_0\times Q_0$ with $(d_{XY},d'_{XY})\neq(0,0)$.
\item Let $\C$ be a skeletally small Krull-Schmidt category. For $X,Y\in\ind\C$, we let
\[D_X:=(\C/J_\C)(X,X),\ {\rm Irr}(X,Y):=(J_\C/J_\C^2)(X,Y),\]
\[d_{XY}:=\dim{\rm Irr}(X,Y)_{D_X},\ d'_{XY}:=\dim{}_{D_Y}{\rm Irr}(X,Y).\]
We define the {\it Auslander-Reiten quiver} $\AR(\C)$ of $\C$ as the valued quiver $(\ind\C,d,d')$.
\item If $\C=\add M$ where $M\in\refl\Lambda$ gives an NCCR, then we draw a dotted arrow $\xymatrix{X \ar@{.>}[r] & \nu X}$ in $\AR(\C)$ for each $X\in\ind\C$ with $X\notin\proj\Lambda$.
\end{enumerate}
\end{Def}

If $\Lambda$ is a symmetric order (i.e. $\Lambda\cong\Lambda^\vee$ holds as $(\Lambda,\Lambda)$-bimodules), then we have $\nu={\rm id_\C}$. Thus in this case we often omit the dotted arrows in (3). To draw the Auslander-Reiten quivers, it is essential to find sink or source maps.

\begin{Prop}\cite[3.12]{INP}
Let $\C$ be a skeletally small Krull-Schmidt category, $X,Y\in\ind\C$ and $Z\in\C$.
\begin{enumerate}
\item If $a\colon Z\to X$ is a sink map, then it induces an isomorphism $a\circ-\colon(\C/J_\C)(-,Z)\overset{\cong}\to(J_\C/J_\C^2)(-,X)$. Thus we have $Z\cong\bigoplus_{W\in\ind\C}W^{\oplus d_{WX}}$.
\item If $b\colon Y\to Z$ is a source map, then it induces an isomorphism $-\circ b\colon(\C/J_\C)(Z,-)\overset{\cong}\to(J_\C/J_\C^2)(Y,-)$. Thus we have $Z\cong\bigoplus_{W\in\ind\C}W^{\oplus d'_{YW}}$.
\end{enumerate}
\end{Prop}

For example, if $\C=\add M$ where $M\in\refl\Lambda$ gives an NCCR, then we can draw $\AR(\C)$ by determining all $(d-1)$-almost split sequences.

%%%%%%%%%%%%%%%%%%%%%%%%%%%%%%%%%%%%%%%%%%%%%%%%%%%%%%%%%%%%%
\section{Auslander-Reiten theory for quotient singularities admitting field extensions}\label{ARquotient}
%%%%%%%%%%%%%%%%%%%%%%%%%%%%%%%%%%%%%%%%%%%%%%%%%%%%%%%%%%%%%

Let $d\geq2$ be an integer. We assume the following settings.
\begin{enumerate}
\item[(Q1)] $l/k$ is a finite Galois extension of fields.
\item[(Q2)] $G$ is a finite subgroup of $GL_d(l)\rtimes\Gal(l/k)$ with $|G|$ not divided by $\ch k$ and $k=l^G$.
\item[(Q3)] $S:=l[[x_1,\ldots , x_d]]$ is the formal power series ring and $\mathfrak{n}:=(x_1,\ldots , x_d)\subseteq S$ is the maximal ideal and $R:=S^G$ is the ring of invariants.
\end{enumerate}
These settings are justified by Proposition \ref{linear}. Remark that the action of $(A,\sigma)\in GL_d(l)\rtimes\Gal(l/k)$ on $S$ as a $k$-algebra is defined by
\[(A,\sigma)(\lambda):=\sigma(\lambda)\ (\lambda\in l)\ {\rm and}\ (A,\sigma)(x_j):=a_{1j}x_1+\cdots+a_{dj}x_d\ (A=[a_{ij}]_{i,j}).\]
Now we have the {\it Reynolds operator}
\[\rho\colon S\to R;s\mapsto\frac{1}{|G|}\sum_{g\in G}g(s).\]
Observe that this is a retraction of the inclusion map $R\hookrightarrow S$ as an $R$-module homomorphism.

\begin{Prop}
The ring $(R, \mathfrak{m}:=\mathfrak{n}\cap R, k)$ is a $d$-dimensional Cohen-Macaulay complete local normal domain and $S\in\CM R$ holds.
\end{Prop}
\begin{proof}
The proof is just a routine (see, for example, \cite[Chapter 5]{LW}). For the convenience of the reader, we prove $R$ is Cohen-Macaulay and complete and $S\in\CM R$.

Since $S\in\mod R$ holds, we have $\depth_RS=\depth_SS=d$. Thus $S\in\CM R$ holds. Moreover, $\depth_RR=d$ follows since $R$ is a direct summand of $S$. Since $(S/\mathfrak{m}S,\mathfrak{n}/\mathfrak{m}S)$ is a finite dimensional local $k$-algebra, there exists $n\in\mathbb{N}$ such that $\mathfrak{n}^n\subseteq\mathfrak{m}S$ holds. Thus the $\mathfrak{m}$-adic topology on $S$ coincides with the $\mathfrak{n}$-adic one. Thus $R$ is complete since it is a submodule of a finitely generated $R$-module $S$ which is $\mathfrak{m}$-adic complete.
\end{proof}

%Note that $l[x_1,\cdots,x_d]^G$ is a finitely generated $k$-algebra since is an $\mathbb{N}$-graded Noetherian ring. Thus we can say that $S$ is a completion of a finitely generated $k$-algebra.

%%%%%%%%%%%%%%%%%%%%%%%%%%%%%%%%%%%%%%%%%%%%%%%%%%%%%%%%%%%%%
\subsection{$S$ is a CT module}
%%%%%%%%%%%%%%%%%%%%%%%%%%%%%%%%%%%%%%%%%%%%%%%%%%%%%%%%%%%%%

In this section, we prove that $S\in\CM R$ is a CT module (Theorem \ref{CT}). First, we show the two-dimensional version of Theorem \ref{CT} since it is easier than the higher dimensional case and is independent of Theorem \ref{Auslanderalgebra}.

\begin{Thm}\label{addgen}
Under the setting $(Q1),(Q2)$ and $(Q3)$, assume $d=2$. Then $S$ is an additive generator of $\CM R$.
\end{Thm}
\begin{proof}
Take $M\in\CM R$. Observe that $M\cong\Hom_R(R,M)$ can be seen as a direct summand of $\Hom_R(S,M)$ as an $R$-module through a split monomorphism $\Hom_R(\rho,M)$. We know that $\Hom_R(S,M)\in\CM R$ holds by Lemma \ref{homdepth}(1). Thus as an $S$-module, $\Hom_R(S,M)\in\CM S=\proj S$ follows.
\end{proof}

Our goal in this subsection is to show the following theorem which is a generalization of \cite[5.7]{IW} or \cite[2.5]{Iya07b} by using Proposition \ref{CTEnd}.

%%
\begin{comment}

To investigate an idempotent complete additive category $\C$ with an additive generator $X\in\C$, the endomorphism ring $\End_\C(X)$ is indispensable because $\C$ is categorically equivalent to $\proj\End_\C(X)^{\op}$ (see Proposition \ref{projectivization}). The importance of $\End_\C(X)$ can be seen from the following fact: when $\End_\C(X)$ is a finite dimensional algebra over a fixed field, whether $\C$ is the module category $\mod A$ over a finite dimensional algebra $A$ is equivalent to whether $\End_\C(X)$ is an Auslander algebra, i.e. algebras satisfying certain homological properties (see \cite{Au71}). This says that there exists a bijection (called Auslander correspondence) between the set of Morita-equivalence class of finite dimensional algebras of finite representation type and that of Auslander algebras. In summary, representation theories have deep connections with homological algebras through considering endomorphism rings.

In summary, studying endomorphism rings is very important for representation theory.

\end{comment}
%%

\begin{Thm}\label{CT}
Under the setting $(Q1),(Q2)$ and $(Q3)$, 
\begin{enumerate}
\item The module $S\in\CM R$ is CT.
\item In particular, if $R$ is an isolated singularity, then $S\in\CM R$ is a $(d-1)$-cluster tilting module.
\end{enumerate}
\end{Thm}

Towards this theorem, first, we investigate ramification theory for the extension $R\subseteq S$. The following description of inertia group is our starting point.

\begin{Prop}\label{invinertia}
For $\mathfrak{P}\in\Spec S$ and $g=(A,\sigma)\in G\subseteq GL_d(l)\rtimes\Gal(l/k)$, the following conditions are equivalent.
\begin{enumerate}
\item The map $\sigma$ is the identity and $U_AS\subseteq\mathfrak{P}$ holds.
\item  The inertia group $T(\mathfrak{P})$ contains $g$.
\end{enumerate}
\end{Prop}
\begin{proof}
The condition $U_AS\subseteq\mathfrak{P}$ is just a rephrasing of $(A,{\rm id})\in T(\mathfrak{P})$. Thus we have only to prove $\sigma={\rm id}$ under (2). Assume $g\in T(\mathfrak{P})$. For $\lambda\in l$, we have $\sigma(\lambda)-\lambda=g(\lambda)-\lambda\in\mathfrak{P}\subseteq\mathfrak{n}$, which means that $\sigma(\lambda)=\lambda$ holds. Thus $\sigma$ is the identity.
\end{proof}

Let $H:=G\cap GL_d(l)$ be a normal subgroup of $G$. Using Proposition \ref{invinertia}, we can characterize when the extension $R\subseteq S$ is unramified in codimension one.

\begin{Prop}\label{whenunram}
The extension $R\subseteq S$ is unramified in codimension one if and only if the subgroup $H\subseteq GL_d(l)$ is small.
\end{Prop}
\begin{proof}
By Theorem \ref{inertia}, we have only to show the following: for $g=(A,\sigma)\in G$, the following conditions are equivalent.
\begin{enumerate}
\item $A$ is a pseudo-reflection and $\sigma$ is the identity.
\item There exists $\mathfrak{P}\in\Spec S$ with $\height\mathfrak{P}=1$ such that $g\in T(\mathfrak{P})$ holds.
\end{enumerate}

(1)$\Rightarrow$(2) We may assume $U_A\neq0$. Take $f\neq0\in U_A$ and let $\mathfrak{P}:=(f)\in\Spec S$. Then $g\in T(\mathfrak{P})$ holds by Proposition \ref{invinertia} since $U_A=lf$ holds.

(2)$\Rightarrow$(1) Take $\mathfrak{P}\in\Spec S$ with $\height\mathfrak{P}=1$ and $g\in T(\mathfrak{P})$. By Proposition \ref{invinertia}, we have $\sigma={\rm id}$ and $U_AS\subseteq\mathfrak{P}$. This forces $\dim_l(U_A)\leq1$ since $\mathfrak{P}$ is generated by one element.
\end{proof}

Then we can give a description of $\End_R(S)$ which is crucial to Theorem \ref{CT}.

\begin{Thm}\label{Auslanderalgebra}
Under the setting $(Q1),(Q2)$ and $(Q3)$, if $H\subseteq GL_d(l)$ is small, then the canonical ring homomorphism 
\[S*G\to\End_R(S);\ s*g\mapsto(s'\mapsto sg(s'))\]
is an isomorphism.
\end{Thm}
\begin{proof}
By Proposition \ref{whenunram}, the extension $R\subseteq S$ is unramified in codimension one. Thus the assertion follows from Theorem \ref{end}.
\end{proof}

Using Theorem \ref{Auslanderalgebra}, we can prove Theorem \ref{CT}. This immediately follows if $H\subseteq GL_d(l)$ is small. If not, we replace $G$ so that $H\subseteq GL_d(l)$ is small.
\begin{proof}[Proof of Theorem \ref{CT}]
By the existence of the Reynolds operator, we have $R\in\add S$. Thus we have only to show that $\End_R(S)$ is a regular $R$-order.

We can take a group homomorphism $\phi\colon G\to GL_d(l)\rtimes\Gal(l/k)$ and an $l$-algebra monomorphism $\theta\colon S\to S$ as in Theorem \ref{makesmall}. Through the induced isomorphism $S^{\phi(G)}\cong S^G$, we have
\[\add_{S^{\phi(G)}}S\simeq\add_{S^G}\Im\theta=\add_{S^G}S.\]
This means that $\End_R(S)$ is Morita equivalent to $\End_{S^{\phi(G)}}(S)$, which is isomorphic to $S*\phi(G)$ by Theorem \ref{Auslanderalgebra}. Thus the assertion holds by Proposition \ref{gldim}.
\end{proof}

Finally, we determine the singular locus of $R$ under the assumption that $H\subseteq GL_d(l)$ is small.

\begin{Thm}\label{invisol}
Under the setting $(Q1),(Q2)$ and $(Q3)$, if $H\subseteq GL_d(l)$ is small, then the singular locus of $R$ is the closed subset of $\Spec R$ defined by the ideal
\[I:=\bigcap_{I_d\neq A\in H}(U_AS\cap R)\subseteq R.\]
\end{Thm}
\begin{proof}
Let $\mathfrak{p}\in\Spec R$. Take $\mathfrak{P}\in\Spec S$ which is lying over $\mathfrak{p}$. If the extension $R\subseteq S$ is unramified at $\mathfrak{P}$, then $R_\mathfrak{p}$ is regular by Theorem \ref{regular}. Conversely, if $R_\mathfrak{p}$ is regular, then the extension $R\subseteq S$ is unramified at $\mathfrak{P}$ by the purity of the branch locus (see \cite[1.4]{Au62}\cite[B.12]{LW}). Therefore, combining with Theorem \ref{inertia}, the ring $R_\mathfrak{p}$ is regular if and only if $|T(\mathfrak{P})|=1$. This is equivalent to that there exists no $A\neq I_d\in H$ with $U_AS\subseteq\mathfrak{P}$ by Proposition \ref{invinertia}.

If $\mathfrak{p}\nsupseteq I$, then there exists no $A\neq I_d\in H$ with $U_AS\subseteq\mathfrak{P}$. Thus $R_\mathfrak{p}$ is regular. If $\mathfrak{p}\supseteq I$, then there exists $A\neq I_d\in H$ with $U_AS\cap R\subseteq\mathfrak{p}$. Then by the going-up theorem, we can retake $\mathfrak{P}$ so that $U_AS\subseteq\mathfrak{P}$ holds. Thus $R_\mathfrak{p}$ is singular.
\end{proof}

As a corollary, we get a necessary and sufficient condition for $R$ to be an isolated singularity.

\begin{Cor}
Under the setting $(Q1),(Q2)$ and $(Q3)$, if $H\subseteq GL_d(l)$ is small, then the following conditions are equivalent.
\begin{enumerate}
\item The ring $R$ is an isolated singularity.
\item Any element $A\neq I_d$ of $H$ does not have eigenvalue one.
\end{enumerate}
\end{Cor}

%%%%%%%%%%%%%%%%%%%%%%%%%%%%%%%%%%%%%%%%%%%%%%%%%%%%%%%%%%%%%
\subsection{Auslander-Reiten sequences}\label{ARsequences}
%%%%%%%%%%%%%%%%%%%%%%%%%%%%%%%%%%%%%%%%%%%%%%%%%%%%%%%%%%%%%

From now on, we denote $\C(R):=\add_RS\subseteq\CM R$. In the rest of this paper, we further assume the following.
\begin{enumerate}
\item[(Q4)] $H:=G\cap GL_d(l)\subseteq GL_d(l)$ is small.
\end{enumerate}
Remark that even if $H\subseteq GL_d(l)$ is small, there may exist an element $(A,\sigma)\in G$ with $A\neq I_d$ pseudo-reflection (see Example \ref{typeCL}). Under the assumption $(Q4)$, we have $\End_R(S)\cong S*G$ by Theorem \ref{Auslanderalgebra}. Since $S*G\to(S*G)^{\rm op};s*g\mapsto g^{-1}(s)*g^{-1}$ gives a ring isomorphism, we have $\End_R(S)\cong(S*G)^{\op}$. Combining with Proposition \ref{projectivization}, we obtain the following.

\begin{Prop}\label{projskew}
We have a categorical equivalence 
\[\Phi:=\Hom_R(S,-)\colon\C(R)\to\proj S*G\]
where the left $S*G$-module structure on $\Hom_R(S,M)$ is given by $((s*g)f)(s'):=f(g^{-1}(ss'))$.
\end{Prop}

Now we consider about a quasi-inverse of $\Phi$. Let
\[\Psi:=(-)^G\colon\proj S*G\to\C(R)\]
be a functor taking the $G$-invariant part. To see the image of $\Psi$ is contained in $\C(R)$, it is enough to show $\Psi(S*G)\cong S$, which will be proved in the following proposition since $S*G\cong\Phi(S)$ holds.

\begin{Prop}
The functor $\Psi$ is a quasi-inverse of $\Phi$.
\end{Prop}
\begin{proof}
Let $\rho\colon S\to R$ be the Reynolds operator. For $M\in\C(R)$, we define $\eta_M:=-\circ\rho:\Hom_R(R,M)\to\Hom_R(S,M)^G$. This is well-defined since for $f\in\Hom_R(R,M)$, $g\in G$ and $s\in S$, we have
\[(g\cdot\eta_M(f))(s)=f(\rho(g^{-1}\cdot s))=f(\rho(s))=\eta_M(f)(s).\]
It is obvious that $\eta:{\rm Id}\Rightarrow\Psi\circ\Phi$ defines a natural homomorphism. We shall show that $\eta$ is a natural isomorphism. 

The injectivity of $\eta_M$ follows from the surjectivity of $\rho$. We show the surjectivity of $\eta_M$. For $h\in\Hom_R(S,M)^G$, we have
\[(\eta_M(h|_R))(s)=h(\rho(s))=h\Bigg(\frac{1}{|G|}\sum_{g\in G}g\cdot s\Bigg)=\frac{1}{|G|}\sum_{g\in G}(g^{-1}\cdot h)(s)=\frac{1}{|G|}\sum_{g\in G}h(s)=h(s).\]
Thus $h=\eta_M(h|_R)\in\Im\eta_M$ holds.
\end{proof}

From this equivalence, we can see that $\ind\C(R)$ corresponds to $\ind\proj S*G$ bijectively. Moreover, $\ind\proj S*G$ bijectively corresponds to $\simp S*G$, which is equal to $\simp l*G$ by the next lemma.

\begin{Lem}\label{Jac}
The Jacobson radical $J_{S*G}$ of $S*G$ is just $\mathfrak{n}(S*G)$. Thus $(S*G)/J_{S*G}=l*G$ holds.
\end{Lem}
\begin{proof}
Since $(S/\mathfrak{m}S, \mathfrak{n}/\mathfrak{m}S)$ is a finite dimensional local $R/\mathfrak{m}$-algebra, there exists $i\in\mathbb{N}$ such that $\mathfrak{n}^i\subseteq\mathfrak{m}S$ holds. Thus $\mathfrak{n}^i(S*G)\subseteq\mathfrak{m}(S*G)$ follows. Since $S*G$ is a module-finite $R$-algebra, we have $\mathfrak{m}(S*G)\subseteq J_{S*G}$. Combining these results, we obtain $\mathfrak{n}(S*G)\subseteq J_{S*G}$ since in $(S*G)/\mathfrak{n}^i(S*G)(=(S/\mathfrak{n}^i)*G)$, the ideal $\mathfrak{n}(S*G)/\mathfrak{n}^i(S*G)$ is nilpotent. We can deduce $\mathfrak{n}(S*G)=J_{S*G}$ because $(S*G)/\mathfrak{n}(S*G)=l*G$ is semisimple by Proposition \ref{gldim}.
\end{proof}

In summary, we obtain the following bijections.

\begin{DefProp}
There exist bijections as follows:
\[\xymatrix@C=6em{
\ind\C(R)\ar@<.3ex>^-{\Hom_R(S,-)}[r]&\ind\proj S*G\ar@<.3ex>^-{(-)^G}[l]\ar@<.3ex>^-{(S/\mathfrak{n})\otimes_S-}[r]&\simp l*G\ar@<.3ex>^-{S\otimes_l-}[l]
}.\]
For $V\in\mod l*G$, we denote by
\[M(V):=(S\otimes_lV)^G\in\C(R)\]
the correspondent to $V$.
\end{DefProp}

For example, for the trivial representation $l\in\simp l*G$ (see Example \ref{trivrep}), we have $M(l)=(S\otimes_ll)^G=R$.

\begin{Rem}\label{rank}
Remark that the rank of $M(V)$ as an $R$-module coincides with the dimension of $V$ as an $l$-linear space since $\dim_lV=\rank_S\Hom_R(S,M(V))=\rank_RM(V)$.
\end{Rem}

Recall that $(d-1)$-almost split sequences in $\C(R)$ correspond to minimal projective resolutions of simple left $S*G$-modules (=simple left $l*G$-modules) (see Corollary \ref{sinksource}). This is given explicitly as follows.

\begin{Prop}\label{minproj}
We view $U:=\mathfrak{n}/\mathfrak{n}^2$ as a left $l*G$-module through the natural $G$-action.
The minimal projective resolution of $V\in\mod l*G$ in $\mod S*G$ is given by applying $-\otimes_l V$ to the Koszul complex
\[0\longrightarrow S\otimes_l\bigwedge^dU\overset{f_d}{\longrightarrow}\cdots\overset{f_3}{\longrightarrow}S\otimes_l\bigwedge^2U\overset{f_2}{\longrightarrow}S\otimes_lU\overset{f_1}{\longrightarrow}S\overset{f_0}{\longrightarrow}l\longrightarrow0.\]
That is, 
\[0\longrightarrow S\otimes_l\bigwedge^dU\otimes_lV\overset{f_d\otimes V}{\longrightarrow}\cdots\overset{f_3\otimes V}{\longrightarrow}S\otimes_l\bigwedge^2U\otimes_lV\overset{f_2\otimes V}{\longrightarrow}S\otimes_lU\otimes_lV\overset{f_1\otimes V}{\longrightarrow}S\otimes_lV\overset{f_0\otimes V}{\longrightarrow}V\longrightarrow0\]
is the minimal projective resolution of $V$ in $\mod S*G$.
\end{Prop}

This proposition forces us to define the McKay quivers as follows.

\begin{Def}\label{McKay}
Take $V_1,V_2\in\simp l*G$. Let $d_G(V_1,V_2)\geq0$ be the number of $V_1$ appearing as direct summands of $U\otimes_lV_2$ where $U=\mathfrak{n}/\mathfrak{n}^2\in\mod l*G$. Let $d'_G(V_1,V_2)\geq0$ be the number of $V_2$ appearing as direct summands of $\bigwedge^{d-1}U\otimes_lV'_1$ where $V'_1\in\simp l*G$ is the irreducible representation satisfying $V_1=\bigwedge^dU\otimes_lV'_1$. We define the {\it McKay quiver of $l*G$} (we denote $\MK(l*G)$) as a valued quiver (see Definition \ref{valquiv}).
\[\MK(l*G):=(\simp l*G,d_G,d'_G)\]
\end{Def}

Combining with Corollary \ref{sinksource}, we obtain the following theorem.

\begin{Thm}\label{ARMcKay}
Under the setting $(Q1),(Q2)$ and $(Q3)$, we assume the condition $(Q4)$. Let $V\in\simp l*G$ and $U:=\mathfrak{n}/\mathfrak{n}^2\in\mod l*G$. Then $\nu(M(V))=M(\bigwedge^dU\otimes_lV)$ holds. Moreover, we have a $(d-1)$-almost split sequence of the following form.
\[M(\bigwedge^dU\otimes_lV)\longrightarrow M(\bigwedge^{d-1}U\otimes_lV)\longrightarrow\cdots\longrightarrow M(U\otimes_lV)\longrightarrow M(V)\]
Thus the Auslander-Reiten quiver $\AR(\C(R))$ coincides with $\MK(l*G)$.
\end{Thm}
\begin{proof}
By Proposition \ref{minproj}, the given complex is a sink sequence of $M(V)$. Thus the assertion follows from Corollary \ref{sinksource}.
\end{proof}

\begin{Cor}\label{canonical}
Under the setting $(Q1),(Q2)$ and $(Q3)$, if we assume the condition $(Q4)$, then we have
\[M(\bigwedge^dU)=\omega,\]
where $\omega\in\ind\C(R)$ denotes the canonical module.
\end{Cor}
\begin{proof}
By Theorem \ref{ARMcKay}, we have the following $(d-1)$-almost split sequence.
\[M(\bigwedge^dU)\longrightarrow M(\bigwedge^{d-1}U)\longrightarrow\cdots\longrightarrow M(U)\longrightarrow M(l)\]
Therefore we have
\[M(\bigwedge^dU)=\nu(M(l))=\nu(R)=\omega.\qedhere\]
\end{proof}

%%%%%%%%%%%%%%%%%%%%%%%%%%%%%%%%%%%%%%%%%%%%%%%%%%%%%%%%%%%%%
\section{The McKay quiver of $lH$ determines that of $l*G$}\label{ARquivers}
%%%%%%%%%%%%%%%%%%%%%%%%%%%%%%%%%%%%%%%%%%%%%%%%%%%%%%%%%%%%%

In the previous section, we proved $\AR(\C(R))=\MK(l*G)$ (see Theorem \ref{ARMcKay}). Our main theorem in this section (Theorem \ref{alg}) describes $\MK(l*G)$ by relating them to $\MK(lH)$.

%%%%%%%%%%%%%%%%%%%%%%%%%%%%%%%%%%%%%%%%%%%%%%%%%%%%%%%%%%%%%
\subsection{Preliminaries on skew group rings}
%%%%%%%%%%%%%%%%%%%%%%%%%%%%%%%%%%%%%%%%%%%%%%%%%%%%%%%%%%%%%

In this subsection, we summarize basic properties of skew group rings for the convenience of the reader. Let $A$ be a ring and $G$ a group acting on $A$ as a ring. Observe the following basic fact about modules over the skew group ring $A*G$.

\begin{Prop}
For an Abelian group $M$, there exists a bijection between the following.
\begin{enumerate}
\item Left $A*G$-module structures on $M$.
\item Pairs of a left $A$-module structure on $M$, and an action of $G$ on $M$ as an Abelian group satisfying $g\cdot(am)=(g\cdot a)(g\cdot m)$.
\end{enumerate}
\end{Prop}

\begin{Ex}\label{trivrep}
We can view $A$ as a left $A*G$-module through the given $G$-action. We call this the {\it trivial representation} as in the case of usual group rings.
\end{Ex}

Especially when $M$ is a free left $A$-module of finite rank, we obtain the following.

\begin{Prop}\label{extend}
Left $A*G$-module structures on $A^m$ extending the natural left $A$-module structure bijectively correspond to group homomorphisms $G\to GL_m(A)\rtimes G$ which make the following diagram commute.
\[\xymatrix{
G \ar[d] \ar[rd]^{\rm id}\\
GL_m(A)\rtimes G \ar[r] & G
}\]
\end{Prop}

Next, under the assumption that $A$ is commutative, we define left $A*G$-module structures on tensor products $-\otimes_A-$ and homomorphisms $\Hom_A(-,-)$ as in the case of usual group rings. This can be summarized as the category $\Mod A*G$ admits a closed symmetric monoidal structure.

\begin{Def}\label{tenhomdef}
Assume $A$ is commutative. Let $M, N$ be left $A*G$-modules.
\begin{enumerate}
\item We define a left $A*G$-module structure on $M\otimes_AN$ by
\[g\cdot(m\otimes n):=(g\cdot m)\otimes(g\cdot n).\]
Similarly, we define a left $A*G$-module structure on $\bigwedge^m M$.
\item We define a left $A*G$-module structure on $\Hom_A(M,N)$ by
\[(g\cdot f)(m):=g\cdot(f(g^{-1}\cdot m)).\]
\end{enumerate}
\end{Def}

We see that Hom-tensor adjoint holds as usual.

\begin{Prop}
Assume $A$ is commutative. Let $L, M, N$ be left $A*G$-modules. Then there exists a natural isomorphism
\[\Hom_{A*G}(L\otimes_A M,N)\cong\Hom_{A*G}(L,\Hom_A(M,N)).\]
\end{Prop}
\begin{proof}
We have the usual adjointness $\Hom_A(L\otimes_A M,N)\cong\Hom_A(L,\Hom_A(M,N))$. By taking the $G$-invariant parts, we obtain the assertion.
\end{proof}

In the rest of this section, we compare the left global dimension of $A*G$ with that of $A$.

\begin{Prop}
The inequality $\lgl A*G\geq\lgl A$ holds.
\end{Prop}
\begin{proof}
We may assume $\lgl A*G$ is finite. Take $M\in\Mod A$. Let
\[0\to P_n\to\cdots\to P_0\to(A*G)\otimes_AM\to 0\]
be a projective resolution in $\Mod A*G$ with $n=\pd_{A*G}((A*G)\otimes_AM)$. Here, this exact sequence also gives a projective resolution in $\Mod A$ since $A*G$ is free as a left $A$-module. Thus we can conclude $\pd_AM\leq n$ since $M$ is a direct summand of $(A*G)\otimes_AM$ as a left $A$-module.
\end{proof}

To obtain the equality in this proposition, we assume that $|G|\cdot 1_A$ is invertible in $A$.

\begin{Lem}\label{invex}
If $|G|\cdot 1_A\in A^\times$ holds, then the functor $(-)^G:\Mod A*G\to\Mod A^G$ is exact.
\end{Lem}
\begin{proof}
Let $0\to L\xrightarrow{\phi} M\xrightarrow{\psi} N\to 0$ be an short exact sequence in $\Mod A*G$. We prove
$0\to L^G\to M^G\to N^G\to 0$ is also exact.

Only the exactness at $N^G$ is non-trivial. Take $z\in N^G$. By the surjectivity of $\psi$, there exists $y'\in M$ such that $\psi(y')=z$ holds. Let $y:=\frac{1}{|G|}\sum_{g\in G}g\cdot y'$, then $y\in M^G$ and $\psi(y)=z$ hold.
\end{proof}

From this lemma, we obtain the following.

\begin{Prop}\label{gldim}
Assume $|G|\cdot 1_A\in A^\times$. For $n\in\mathbb{N}$ and $M,N\in\Mod A*G$, we have
\[\Ext_A^n(M,N)^G=\Ext_{A*G}^n(M,N).\]
Especially, $\lgl A*G=\lgl A$ holds.
\end{Prop}
\begin{proof}
Let $P_\bullet\to M\to0$ be a projective resolution in $\Mod A*G$. Then this is also gives a projective resolution in $\Mod A$. Thus $\Ext_A^n(M,N)$ is the $n$th cohomology of the complex $\Hom_A(P_\bullet,N)$. By the above lemma, $\Ext_A^n(M,N)^G$ is the $n$th cohomology of the complex $\Hom_A(P_\bullet,N)^G=\Hom_{A*G}(P_\bullet,N)$.
\end{proof}

%%%%%%%%%%%%%%%%%%%%%%%%%%%%%%%%%%%%%%%%%%%%%%%%%%%%%%%%%%%%%
\subsection{Irreducible representations of skew group algebras}
%%%%%%%%%%%%%%%%%%%%%%%%%%%%%%%%%%%%%%%%%%%%%%%%%%%%%%%%%%%%%

In this subsection, we see how we can determine irreducible representations of skew group algebras. We discuss in a general setting: let $l$ be a field, $G$ a finite group acting on $l$ as a field and $k:=l^G$. We do not impose any assumption on $\ch l$ or $|G|$. By Artin's theorem, we know that $l/k$ is a finite Galois extension and the natural group homomorphism $G\to\Gal(l/k)$ is surjective. We let
\[H:=\Ker(G\to\Gal(l/k))\]
be a normal subgroup of $G$. Our strategy is to relate $\mod l*G$ to $\mod lH$ since we can investigate modules over $lH$ more easily.

\begin{Def}
We define the {\it induction functor} $\Ind$ and the {\it restriction functor} $\Res$ by
\[\Ind:=(l*G)\otimes_{lH}-\colon\mod lH\to\mod l*G,\]
\[\Res:={}_{lH}(-)\colon\mod l*G\to\mod lH.\]
\end{Def}

For $g\in G$ and $W\in\mod lH$, we define an left $lH$-module $gW$ by twisting the action of $H$ on $W$ by the group automorphism $g^{-1}\cdot-\cdot g\colon H\to H$. Then this gives an auto-equivalence $g-\colon\mod lH\to\mod lH$ and the group $G$ acts on the category $\mod lH$. On the other hand, we define
\[g\otimes W:=\{g\otimes w\mid w\in W\}\subseteq(l*G)\otimes_{lH}W.\]
Then we have $g\otimes W\cong gW$ naturally. Since $h\otimes W\cong W$ holds for $h\in H$, we can conclude that $\Gal(l/k)\cong G/H$ acts on the set $\simp lH$.

\begin{Thm}\label{tab}
\begin{enumerate}
\item There exists a bijection
\[\Gal(l/k)\backslash\simp lH\cong\simp l*G.\]
\item Decompose $\simp lH$ into the disjoint union of the orbits with respect the $\Gal(l/k)$-action$\colon$
\[\simp lH=\bigsqcup_{i=1}^{s}\{W_{i1},\cdots, W_{it_i}\}.\]
We denote by $V_i\in\simp l*G$ the simple left $l*G$-module corresponding to $\{W_{i1},\cdots, W_{it_i}\}$. Then for each $1\leq i\leq s$, there exist integers $a_i, b_i\geq1$ satisfying
\begin{equation*}
\begin{split}
\Res V_i&\cong (W_{i1}\oplus\cdots\oplus W_{it_i})^{\oplus a_i},\\
\Ind W_{ij}&\cong V_i^{\oplus b_i}\text{for each $1\leq j\leq t_i$ and}\\
t_ia_ib_i&=[l:k].
\end{split}
\end{equation*}
\end{enumerate}
\end{Thm}

To prove Theorem \ref{tab}, we investigate $\Res\Ind\colon\mod lH\to\mod lH$ and $\Ind\Res\colon\mod l*G\to\mod l*G$. The first one is easy.

\begin{Prop}\label{resind}
Let $T\subseteq G$ be a subset with $G=\bigsqcup_{g\in T}gH$. We have a natural isomorphism
\[\Res\Ind=\bigoplus_{g\in T}g\otimes-\colon\mod lH\to\mod lH.\]
\end{Prop}
\begin{proof}
We can decompose $l*G$ into $\bigoplus_{g\in T}l*(gH)$ as an $(lH, lH)$-bimodule. Therefore we have
\[\Res\Ind=\bigoplus_{g\in T}(l*(gH))\otimes_{lH}-=\bigoplus_{g\in T}g\otimes-.\qedhere\]
\end{proof}

Next, we consider $\Ind\Res$. Unlike Proposition \ref{resind}, the discussions depend strongly on our setting. We view $l*(G/H)$ as a left $l*G$-module through the natural ring homomorphism $l*G\to l*(G/H)$.

\begin{Lem}\label{nontrivial}
There exists an isomorphism
\[l*(G/H)\cong l^{\oplus[l:k]}\]
as a left $l*G$-module where $l$ in the right hand side is the trivial representation.
\end{Lem}
\begin{proof}
We have ring isomorphisms
\[l*(G/H)\cong l*\Gal(l/k)\cong\End_k(l),\]
where the last one is a special case of Proposition \ref{sep}. Therefore we have
\[_{l*G}(l*(G/H))\cong\ _{l*G}\End_k(l)=\Hom_k(l_k,\ _{l*G}l_k)\cong\Hom_k(k^{\oplus[l:k]}_k,\ _{l*G}l_k)\cong\ _{l*G}l^{\oplus[l:k]}.\qedhere\]
\end{proof}

\begin{Prop}\label{indres}
We have a natural isomorphism
\[\Ind\Res\cong (-)^{\oplus[l:k]}\colon\mod l*G\to\mod l*G.\]
\end{Prop}
\begin{proof}
For $V\in\mod l*G$, we have an $l*G$-homomorphism $1\otimes-\colon V\to(l*(G/H))\otimes_lV$. This induces an $lH$-homomorphism $\Res V\to\Res((l*(G/H))\otimes_lV)$. By the adjointness, we obtain an $l*G$-homomorphism
\[\Ind\Res V=(l*G)\otimes_{lH}V\to(l*(G/H))\otimes_lV;(\lambda*g)\otimes v\mapsto(\lambda*(gH))\otimes g(v).\]
This is obviously surjective. The injectivity also follows since $\dim_l((l*G)\otimes_{lH}V)=[l:k]\dim_lV=\dim_l((l*(G/H))\otimes_lV)$ holds. Thus the assertion holds by Lemma \ref{nontrivial}.
\end{proof}

Finally, we see an analogue of Clifford's theorem \cite{Cli}. Observe that if $\ch l$ does not divide $|H|$, then this is obvious.

\begin{Lem}\label{Clifford}
For $V\in\simp l*G$, $\Res V\in\mod lH$ is semisimple.
\end{Lem}
\begin{proof}
Take a simple $lH$-submodule $W\subseteq\Res V$. For any $g\in G$, $gW\subseteq\Res V$ is also a simple $lH$-submodule. Put $V':=\sum_{g\in G}gW\subseteq\Res V$. This is obviously an $l*G$-submodule and thus $V'=V$ holds. Since $V'$ is a sum of simple $lH$-modules, we get the assertion.
\end{proof}

Under these preparations, we can prove Theorem \ref{tab}.
\begin{proof}[Proof of Theorem \ref{tab}]
Take $V\in\simp l*G$. By Lemma \ref{Clifford}, we have a decomposition $\Res V\cong W_1^{\oplus a^1}\oplus\cdots\oplus W_t^{\oplus a^t}$ into simple $lH$-modules. By Proposition \ref{indres}, we have
\[V^{\oplus[l:k]}\cong\Ind\Res V\cong \Ind W_1^{\oplus a^1}\oplus\cdots\oplus \Ind W_t^{\oplus a^t}.\]
Thus there exists $b^j\in\mathbb{N}$ such that $\Ind W_j\cong V^{\oplus b^j}$ by the Krull-Remak-Schmidt theorem. Observe that we have
\[\Res\Ind W_1\cong\Res V^{\oplus b^1}\cong W_1^{\oplus a^1b^1}\oplus\cdots\oplus W_t^{\oplus a^tb^1}.\]
Thus by Proposition \ref{resind}, the set $\{W_1,\cdots, W_t\}$ coincides with an orbit with respect to the action of $\Gal(l/k)$ on $\simp lH$ and $a:=a^1=\cdots=a^t$ holds. In particular, we obtain $\dim_lW_1=\cdots=\dim_lW_t$. By $\Ind W_j\cong V^{\oplus b^j}$, we have $b^j\dim_lV=[l:k]\dim_lW_j$. Thus we can conclude $b:=b^1=\cdots=b^t$. Therefore we can write
\[\Res\Ind W_1\cong (W_1\oplus\cdots\oplus W_t)^{\oplus ab}.\]
Since $\dim_l\Res\Ind W_1=[l:k]\dim_l W_1$, we obtain $tab=[l:k]$.
\end{proof}

In Section \ref{Examples}, we calculate $a_i$ by using the equation $t_ia_ib_i=[l:k]$.

%%%%%%%%%%%%%%%%%%%%%%%%%%%%%%%%%%%%%%%%%%%%%%%%%%%%%%%%%%%%%
\subsection{McKay quivers}
%%%%%%%%%%%%%%%%%%%%%%%%%%%%%%%%%%%%%%%%%%%%%%%%%%%%%%%%%%%%%

In this subsection, we adopt the settings $(Q1),(Q2),(Q3)$ and assume the condition $(Q4)$. Using Theorem \ref{tab}, we can describe an algorithm to determine $(d-1)$-almost split (fundamental) sequences in $\C(R)$ and draw the McKay quiver of $l*G$. Before that, as an application of the results in the previous subsection, we prove the following generalization of Watanabe's result \cite{W}.

\begin{Thm}\label{Gor}
Under the setting $(Q1),(Q2)$ and $(Q3)$, we assume the condition $(Q4)$. Then the ring $R=S^G$ is Gorenstein if and only if $H\subseteq SL_d(l)$ holds.
\end{Thm}
\begin{proof}
We have $M(l)=R$ and $M(\bigwedge^dU)=\omega$ by Corollary \ref{canonical}. Thus $R$ is Gorenstein if and only if $\bigwedge^dU\cong l$ as a left $l*G$-module. By Proposition \ref{indres}, this is equivalent to $\Res\bigwedge^dU\cong\Res l$. Here, for $A\in H$, we have $A(x_1\wedge\cdots\wedge x_d)=(\det A)\ x_1\wedge\cdots\wedge x_d$. Thus this is equivalent to that $\det A=1$ holds for all $A\in H$.
\end{proof}

Remark that Theorem \ref{Gor} fails if we drop the assumption that $H$ is small. In fact, for any finite subgroup $G\subseteq GL_d(l)$ generated by pseudo-reflections, $S^G$ is regular by Chevalley-Shephard-Todd Theorem (Proposition \ref{allref}). In addition, note that the condition $H\subseteq SL_d(l)$ is weaker than $G\subseteq SL_d(l)\rtimes\Gal(l/k)$ (see Example \ref{typeCL}).

Now we consider the McKay quivers. To state our main theorem in this subsection, we need the concept of quotient quivers.

\begin{Def}\label{quot}
Let $Q$ be a quiver and $G$ a group acting on $Q$. We define the quotient quiver $G\backslash Q$ as follows. The vertices of $G\backslash Q$ is the set of the orbits with respect to the action of $G$ on the vertices of $Q$. For vertices $i,j$ of $G\backslash Q$, there exists an arrow from $i$ to $j$ in $G\backslash Q$ if there exists an arrow from an element of $i$ to an element of $j$ in $Q$.
\end{Def}

We see that if we forget about valuations, the quiver $\MK(l*G)$ coincides with the quotient quiver $\Gal(l/k)\backslash\MK(lH)$. Moreover, we can determine valuations of $\MK(l*G)$ by those of $\MK(lH)$. We prepare some notations. Let $\MK(lH)=(\simp lH,d_H,d'_H)$ and fix $1\leq i,i'\leq s$. Choose $1\leq j\leq t_i$ and let
\[d_H(i,i'):=\sum_{j'=1}^{t_{i'}}d_H({W_{ij},W_{i'j'}})\]
which does not depend on the choice of $j$ because of the action of $\Gal(l/k)$ on $\MK(lH)$. Similarly, choose $1\leq j\leq t_i$ and let
\[d'_H(i,i'):=\sum_{j'=1}^{t_{i'}}d'_H({W_{ij},W_{i'j'}}).\]

\begin{Thm}\label{alg}
Observe that the group $\Gal(l/k)$ acts on $\MK(lH)$.
\begin{enumerate}
\item We have an equation $\MK(l*G)=\Gal(l/k)\backslash\MK(lH)$ as quivers without valuations.
\item Let $\MK(l*G)=(\simp l*G,d_G,d'_G)$ and $\MK(lH)=(\simp lH,d_H,d'_H)$. For $1\leq i,i'\leq s$, we have
\[d_G(V_i,V_{i'})=\dfrac{a_{i'}}{a_i}d_H(i,i')\ {\rm and}\ d'_G(V_i,V_{i'})=\dfrac{a_i}{a_{i'}}d'_H(i,i').\]
\end{enumerate}
\end{Thm}
\begin{proof}
We let $U:=\mathfrak{n}/\mathfrak{n}^2\in\mod l*G$. By using $(\Res U)\otimes_lW_{i'j'}=\bigoplus_{i=1}^s\bigoplus_{j=1}^{t_i}W_{ij}^{\oplus d_H({W_{ij},W_{i'j'}})}\ (1\leq j'\leq t_{i'})$ and $\Res V_{i'}\cong (W_{i'1}\oplus\cdots\oplus W_{i't_{i'}})^{\oplus a_{i'}}$, we have 
\[\Res(U\otimes_lV_{i'})\cong\bigoplus_{j'=1}^{t_{i'}}(\Res U)\otimes_lW_{i'j'}^{\oplus a_{i'}}=\bigoplus_{i=1}^s\bigoplus_{j=1}^{t_i}\bigoplus_{j'=1}^{t_{i'}}W_{ij}^{\oplus a_{i'}d_H({W_{ij},W_{i'j'}})}=\bigoplus_{i=1}^s\bigoplus_{j=1}^{t_i}W_{ij}^{\oplus a_{i'}d_H(i,i')}\]%\\
%=&\bigoplus_{i=1}^s\bigoplus_{j=1}^{t_i}W_{ij}^{\oplus a_{i'}d_H(i,i')}\cong\bigoplus_{i=1}^s\Res V_i^{\oplus\frac{a_{i'}}{a_i}d_H(i,i')}.
On the other hand, we have
\[\Res(U\otimes_lV_{i'})=\Res\bigg(\bigoplus_{i=1}^sV_i^{d_G(V_i,V_{i'})}\bigg)\cong\bigoplus_{i=1}^s\bigoplus_{j=1}^{t_i}W_{ij}^{\oplus a_id_G(V_i,V_{i'})}.\]
Thus we obtain $a_id_G(V_i,V_{i'})=a_{i'}d_H(i,i')$. The other equality can be proved in the same way.
\end{proof}

In fact, more strongly, we can determine each term of $(d-1)$-almost split (fundamental) sequences in $\C(R)$. In this view point, we can interpret Theorem \ref{alg} as a determination of the second and the second last terms of $(d-1)$-almost split (fundamental) sequences. Let $V\in\mod l*G$. By Proposition \ref{minproj}, the minimal projective resolution of $V$ in $\mod S*G$ is given by
\[0\longrightarrow S\otimes_l\bigwedge^dU\otimes_l V\longrightarrow\cdots\longrightarrow S\otimes_lU\otimes_lV\longrightarrow S\otimes_lV\longrightarrow V\longrightarrow0,\]
where $U:=\mathfrak{n}/\mathfrak{n}^2\in\mod l*G$. We view this exact sequence as one in $\mod S*H$, that is
\[0\longrightarrow S\otimes_l\bigwedge^d\Res U\otimes_l\Res V\longrightarrow\cdots\longrightarrow S\otimes_l\Res U\otimes_l\Res V\longrightarrow S\otimes_l\Res V\longrightarrow\Res V\longrightarrow0.\]
This is the minimal projective resolution of $\Res V$ in $\mod S*H$ by Proposition \ref{minproj}. Therefore, if $\Res V=\bigoplus_{j}W_j\ (W_j\in\simp lH)$, the minimal projective resolution of $V$ in $\mod S*G$ is obtained by taking a direct sum of minimal projective resolutions of $W_j$'s in $\mod S*H$. By this observation, we obtain the following.

\begin{Thm}\label{ARseqsum}
Take $V\in\simp l*G$ which is not the trivial representation and let $\Res V=(W_1\oplus\cdots\oplus W_t)^{\oplus a}$ where $W_1,\cdots,W_t\in\simp lH$ are non-isomorphic to each other. Assume
\[0\longrightarrow N(W_{j,d-1})\longrightarrow\cdots\longrightarrow N(W_{j,1})\longrightarrow N(W_{j,0})\longrightarrow N(W_j)\longrightarrow 0\]
is a $(d-1)$-almost split sequence in $\C(S^H)$ where $N(W):=(S\otimes_lW)^H\in\C(S^H)$ for $W\in\mod lH$. For each $0\leq p\leq d-1$, there exists a unique $V_p\in\mod l*G$ up to isomorphism such that $\Res V_p\cong(W_{1,p}\oplus\cdots\oplus W_{t,p})^{\oplus a}$ holds. Under this notation, there is a $(d-1)$-almost split sequence in $\C(R)$ of the form
\[0\longrightarrow M(V_{d-1})\longrightarrow\cdots\longrightarrow M(V_1)\longrightarrow M(V_0)\longrightarrow M(V)\longrightarrow 0.\]
As for $(d-1)$-fundamental sequences, the same assertion holds.
\end{Thm}

%%%%%%%%%%%%%%%%%%%%%%%%%%%%%%%%%%%%%%%%%%%%%%%%%%%%%%%%%%%%%
\subsection{Divisor class groups}
%%%%%%%%%%%%%%%%%%%%%%%%%%%%%%%%%%%%%%%%%%%%%%%%%%%%%%%%%%%%%

In this subsection we explain how to determine the divisor class group $\Cl(R)$. Here, we view $\Cl(R)$ as the set of isomorphism classes of reflexive $R$-modules whose ranks are one.

\begin{Lem}\label{refaddS}
\begin{enumerate}
\item For $N\in\mod S$, $N^*=\Hom_R(N,R)$ is reflexive as an $S$-module.
\item Every reflexive $R$-module of rank one is a direct summand of $S\in\mod R$.
\end{enumerate}
\end{Lem}
\begin{proof}
(1) Take a projective presentation $P_1\to P_0\to N\to 0$ in $\mod S$. Then we have an exact sequence $0\to N^*\to P_0^*\to P_1^*$. Here $P_i^*\in\proj S$ holds by $\End_R(S)\cong S*G$ (see Theorem \ref{Auslanderalgebra}). Thus the assertion follows.

(2) Take a reflexive module $M\in\mod R$ with $\rank_RM=1$. By the existence of the Reynolds operator, $M$ is a direct summand of $M\otimes_RS$ as an $R$-module. Thus $M\cong M^{**}$ is a direct summand of $(M\otimes_RS)^{**}$ as an $R$-module where $(-)^*:=\Hom_R(-,R)$. By (1), $(M\otimes_RS)^{**}$ is reflexive as an $S$-module. Since $\rank_S(M\otimes_RS)^{**}=1$, we can conclude $(M\otimes_RS)^{**}\in\Cl(S)=\{[S]\}$.
\end{proof}

By this lemma, we can view $\Cl(R)$ as a subset of $\ind\C(R)\cong\simp l*G$. Next, we see how to calculate the product in $\Cl(R)$ by using the functor $M\colon\mod l*G\to\C(R)$.

\begin{Prop}\label{tenhomcor}
Take $V,V'\in\mod l*G$.
\begin{enumerate}
\item We have $\Hom_R(M(V),M(V'))\cong M(\Hom_l(V,V'))$ (see Definition \ref{tenhomdef}).
\item We have $(M(V)\otimes_RM(V'))^{**}\cong M(V\otimes_lV')$ where $X^*:=\Hom_R(X,R)\in\mod R$ for $X\in\mod R$.
\end{enumerate}
\end{Prop}
\begin{proof}
(1) This follows from
\begin{equation*}
\begin{split}
\Hom_R((S\otimes_lV)^G,(S\otimes_lV')^G)&\cong\Hom_{S*G}(S\otimes_lV,S\otimes_lV')\\
\cong\Hom_S(S\otimes_lV,S\otimes_lV')^G&\cong(S\otimes_l\Hom_l(V,V'))^G.
\end{split}
\end{equation*}

(2) This follows from
\begin{equation*}
\begin{split}
(M(V)\otimes_RM(V'))^{**}&\cong\Hom_R(M(V),\Hom_R(M(V'),R))^*\cong\Hom_R(M(V),M(\Hom_l(V',l)))^*\\
&\cong(M(\Hom_l(V,\Hom_l(V',l))))^*\cong\Hom_R(M(\Hom_l(V\otimes_lV',l)),R)\\
&\cong M(\Hom_l(\Hom_l(V\otimes_lV',l),l))\cong M(V\otimes_lV').\qedhere
\end{split}
\end{equation*}
\end{proof}

As a consequence, we obtain the following theorem.

\begin{Thm}\label{divgrp}
Under the setting $(Q1),(Q2)$ and $(Q3)$, we assume the condition $(Q4)$.
\begin{enumerate}
\item We have a group isomorphism $\Cl(S^H)\cong\Hom_\mathbb{Z}(H^{\rm ab},l^\times)$ where $H^{\rm ab}$ is the Abelianization of $H$.
\item There exists an injective group homomorphism $\Cl(R)\to\Cl(S^H)$ induced by $\Res$. %More explicitly, 
%\[\Cl(R)\to\Cl(S^H);\ [M]\mapsto[M\otimes_RS^H]\]
%gives an injective group homomorphism.
\end{enumerate}
\end{Thm}
\begin{proof}
By Lemma \ref{refaddS} and Remark \ref{rank}, we have a bijection
\[\{V\in\simp l*G\mid\dim_lV=1\}\to\Cl(R);\ V\mapsto[M(V)].\]
For $V, V'\in\simp l*G$ with $\dim_lV=\dim_lV'=1$, we know $[M(V)]\cdot[M(V')]=[M(V\otimes_lV')]\in\Cl(R)$ by Proposition \ref{tenhomcor}.

(1) This follows by the fact that one-dimensional representations in $\simp lG$ corresponds bijectively to $\Hom_\mathbb{Z}(G^{\rm ab},l^\times)$. Observe that this correspondence preserves products.

(2) We  can regard $\Cl(R)$ as a subset of $\Cl(S^H)$ through
\[\{V\in\simp l*G\mid\dim_lV=1\}\to\{W\in\simp lH\mid\dim_lW=1\};\ V\mapsto\Res V.\]
Moreover, this turns out to be a subgroup since $\Res(V\otimes_lV')=(\Res V)\otimes_l(\Res V')$ holds. %The explicit form of $\Cl(R)\to\Cl(S^H)$ is obtained by Proposition \ref{indrescm}.
\end{proof}

Let $V\in\simp l*G$ and $\Res V\cong (W_1\oplus\cdots\oplus W_s)^{\oplus a}$ be the decomposition to simple modules in $\mod lH$. Then it is obvious that $\dim_lV=1$ is equivalent to that $s=1$, $a=1$ and $\dim_lW_1=1$ hold. Thus we can determine easily which elements of $\Cl(S^H)$ belong to $\Cl(R)$.

\section{Examples}\label{Examples}
%%%%%%%%%%%%%%%%%%%%%%%%%%%%%%%%%%%%%%%%%%%%%%%%%%%%%%%%%%%%%

In this section, we explain the method for drawing Auslander-Reiten quivers using Theorems \ref{tab} and \ref{ARseqsum} through several explicit examples. The calculations are divided into the following three steps.

Step(I) Determine the set $\simp lH$ and the action of $\Gal(l/k)$ on $\simp lH$.

Step(I\hspace{-1.2pt}I) Determine $a$'s and $b$'s.

Step(I\hspace{-1.2pt}I\hspace{-1.2pt}I) Determine $(d-1)$-almost split sequences and $(d-1)$-fundamental sequences.

As seen later, it is important to determine the image of $\Res\colon\mod l*G\to\mod lH$ in the Step(I\hspace{-1.2pt}I). We exhibit explicitly a necessary and sufficient condition for this.

\begin{Prop}\label{existV}
For a given $W\in\mod lH$ with $m=\dim_lW$, we denote by $\rho\colon H\to GL_m(l)$ the corresponding group homomorphism. Then the following conditions are equivalent.
\begin{enumerate}
\item There exists $V\in\mod l*G$ such that $W=\Res V$ holds.
\item There exists a group homomorphism $G\to GL_m(l)\rtimes\Gal(l/k)$ making the following diagram commute.
\[\xymatrix{
H \ar@{^{(}->}[r] \ar[d]_\rho & G \ar@{.>}[d] \ar[rd]\\
GL_m(l) \ar@{^{(}->}[r] & GL_m(l)\rtimes\Gal(l/k) \ar[r] & \Gal(l/k)
}\]
\end{enumerate}
\end{Prop}
\begin{proof}
It follows from Proposition \ref{extend} easily.
\end{proof}

%%%%%%%%%%%%%%%%%%%%%%%%%%%%%%%%%%%%%%%%%%%%%%%%%%%%%%%%%%%%%
\subsection{Two-dimensional examples}
%%%%%%%%%%%%%%%%%%%%%%%%%%%%%%%%%%%%%%%%%%%%%%%%%%%%%%%%%%%%%

First, we see two-dimensional examples where $H=G\cap GL_2(l)$ is cyclic.

\begin{Ex}($d=2$, Gorenstein, $\widetilde{CL}$)\label{typeCL}
Let $n$ be a positive integer and $l/k$ a field extension with $\ch k$ not dividing $2(2n+1)$ and $[l:k]=2$. We assume $\zeta:=\zeta_{2n+1}\in l$. For example, $\mathbb{C}/\mathbb{R}$ and $\mathbb{Q}(\zeta_{2n+1})/\mathbb{Q}(\cos\frac{2\pi}{2n+1})$ satisfy these conditions. We denonte $\Gal(l/k)=\{{\rm id},\sigma\}$. Let
\[\alpha:=\Big(\Big[\begin{smallmatrix}\zeta&0\\0&\zeta^{-1}\end{smallmatrix}\Big], {\rm id}\Big),\ \beta:=\Big(\Big[\begin{smallmatrix}0&1\\1&0\end{smallmatrix}\Big],\sigma\Big)\in GL_2(l)\rtimes\Gal(l/k)\]
and $G:=\langle\alpha,\beta\rangle\subseteq GL_2(l)\rtimes\Gal(l/k)$. Then we can write $\ind(\C(R))$ as $\{M_0=R, M_{\pm1}, \cdots, M_{\pm n}\}$ and draw $\AR(\C(R))$, which is of type $\widetilde{CL}_n$,  as follows. In addition, $\Cl(R)=1$ holds.
\[\xymatrix{
M_0 \ar@<0.5ex>[r]^-{(2,1)} & M_{\pm1} \ar@<0.5ex>[l]^-{(1,2)} \ar@<0.5ex>[r] & \cdots \ar@<0.5ex>[l] \ar@<0.5ex>[r] & M_{\pm n} \ar@<0.5ex>[l] \ar@(ur,dr)
}\]
Compare with $\AR(\C(S^H))$.
\[\xymatrix@R=1mm{
& N_1 \ar@<0.5ex>[dl] \ar@<0.5ex>[r] & \cdots \ar@<0.5ex>[l] \ar@<0.5ex>[r] & N_n \ar@<0.5ex>[l] \ar@<0.5ex>[dd] \\
N_0 \ar@<0.5ex>[ur] \ar@<0.5ex>[dr] & & & \\
& N_{-1} \ar@<0.5ex>[ul] \ar@<0.5ex>[r] & \cdots \ar@<0.5ex>[l] \ar@<0.5ex>[r] & N_{-n} \ar@<0.5ex>[l] \ar@<0.5ex>[uu]
}\]
\end{Ex}
\begin{proof}
Step(I) Put $\sigma(\zeta)=\zeta^m$. By easy calculations, we can check $G=\langle\alpha,\beta\mid\alpha^{2n+1}=\beta^2=1,\ \beta\alpha=\alpha^{-m}\beta\rangle$. Thus $H:=G\cap GL_2(l)=\langle\alpha\rangle\cong C_{2n+1}$ holds and so $H\subseteq SL_2(l)$. Let $\simp lH=\{W_j\}_{j\in\mathbb{Z}/(2n+1)\mathbb{Z}}$. Here, $W_j=l$ as an $l$-linear space and $\alpha\cdot-=\zeta^j\colon W_j\to W_j$. Then the map $W_{-j}\to\beta\otimes W_j;\lambda\mapsto\beta\otimes\sigma(\lambda)$ is an $lH$-isomorphism, so $\beta\otimes W_j=W_{-j}$ holds. Since $G=H\sqcup\beta H$, the orbit decomposition of $\simp lH$ with respect to the $G$-action is given by
\[\simp lH=\{W_0\}\sqcup\bigsqcup_{j=1}^s\{W_j, W_{-j}\}.\]
Thus by Theorem \ref{tab}, $\simp l*G$ can be written as $\{V_0, V_{\pm1}, \cdots, V_{\pm n}\}$, where $V_0$ and $V_{\pm j}$ correspond to $\{W_0\}$ and $\{W_j, W_{-j}\}$ respectively.

Step(I\hspace{-1.2pt}I) More explicitly, there exist $a_0, a_{\pm j}, b_0, b_{\pm j}\in\mathbb{N}\ (1\leq i\leq n)$ satisfying
\begin{gather}
a_0b_0=2a_{\pm j}b_{\pm j}=[l:k]=2\label{CLab}\\
\Res V_0=W_0^{\oplus a_0},\ \Res V_{\pm j}=(W_j\oplus W_{-j})^{\oplus a_{\pm j}}\notag\\
\Ind W_0=V_0^{\oplus b_0},\ \Ind W_j=\Ind W_{-j}=V_{\pm j}^{\oplus b_{\pm j}}.\notag
\end{gather}
We claim
\[a_0=a_{\pm j}=1.\]
In fact, by \eqref{CLab}, we have $a_{\pm j}=b_{\pm j}=1$. The trivial representation $l\in\simp l*G$ satisfies $\Res l=W_0$, so we have $V_0=l$ and hense $a_0=1, b_0=2$. From this, we can check $\Cl(R)\cong\{V_0\}$ easily.

Step(I\hspace{-1.2pt}I\hspace{-1.2pt}I) Remark that the minimal projective resolution of $W_j\in \mod S*H$ is given by
\[0\longrightarrow S\otimes_lW_j\longrightarrow S\otimes_l(W_{j-1}\oplus W_{j+1})\longrightarrow S\otimes_lW_j\longrightarrow W_j\longrightarrow 0.\]
By Theorem \ref{ARseqsum}, almost split sequences and fundamental sequences in $\C(R)$ have the following forms, where we put $M_0:=M(V_0)$ and $M_{\pm j}:=M(V_{\pm j})$.
\[\xymatrix@R=1mm@!C=18mm{
0 \ar[r] & M_0 \ar[r] & M_{\pm 1} \ar[r] & M_0\\
0 \ar[r] & M_{\pm 1} \ar[r] & M_0^{\oplus 2}\oplus M_{\pm 2} \ar[r] & M_{\pm 1} \ar[r] & 0\\
0 \ar[r] & M_{\pm i} \ar[r] & M_{\pm(i-1)}\oplus M_{\pm(i+1)} \ar[r] & M_{\pm i} \ar[r] & 0\ (1<i<n)\\
0 \ar[r] & M_{\pm n} \ar[r] & M_{\pm(n-1)}\oplus M_{\pm n} \ar[r] & M_{\pm n} \ar[r] & 0
}\]
Using these sequences, we can draw $\AR(\C(R))$.
\end{proof}

We see more examples where $H$ is cyclic.

\begin{Ex}($d=2$, Gorenstein, $\widetilde{C}$)\label{typeC}
Let $n\geq2$ be an integer and $l/k$ a field extension with $\ch k$ not dividing $2n$ and $[l:k]=2$. We assume $\zeta:=\zeta_{2n}\in l$. For example, $\mathbb{C}/\mathbb{R}$ and $\mathbb{Q}(\zeta_{2n})/\mathbb{Q}(\cos\frac{\pi}{n})$ satisfy these conditions. We denonte $\Gal(l/k)=\{{\rm id},\sigma\}$. Let
\[\alpha:=\Big(\Big[\begin{smallmatrix}\zeta&0\\0&\zeta^{-1}\end{smallmatrix}\Big], {\rm id}\Big),\ \beta:=\Big(\Big[\begin{smallmatrix}0&1\\1&0\end{smallmatrix}\Big],\sigma\Big)\in GL_2(l)\rtimes\Gal(l/k)\]
and $G:=\langle\alpha,\beta\rangle\subseteq GL_2(l)\rtimes\Gal(l/k)$. Then we can write $\ind(\C(R))$ as $\{M_0=R, M_{\pm1}, \cdots, M_{\pm n}\}$ and draw $\AR(\C(R))$, which is of type $\widetilde{C}_n$,  as follows. In addition, $\Cl(R)=1$ holds.
\[\xymatrix{
M_0 \ar@<0.5ex>[r]^-{(2,1)} & M_{\pm1} \ar@<0.5ex>[l]^-{(1,2)} \ar@<0.5ex>[r] & \cdots \ar@<0.5ex>[l] \ar@<0.5ex>[r] & M_{\pm(n-1)} \ar@<0.5ex>[l] \ar@<0.5ex>[r]^-{(1,2)} & M_n \ar@<0.5ex>[l]^-{(2,1)}
}\]
Compare with $\AR(\C(S^H))$.
\[\xymatrix@R=1mm{
& N_1 \ar@<0.5ex>[dl] \ar@<0.5ex>[r] & \cdots \ar@<0.5ex>[l] \ar@<0.5ex>[r] & N_{n-1} \ar@<0.5ex>[l] \ar@<0.5ex>[dr] \\
N_0 \ar@<0.5ex>[ur] \ar@<0.5ex>[dr] & & & &N_n \ar@<0.5ex>[ul] \ar@<0.5ex>[dl] \\
& N_{-1} \ar@<0.5ex>[ul] \ar@<0.5ex>[r] & \cdots \ar@<0.5ex>[l] \ar@<0.5ex>[r] & N_{-n+1} \ar@<0.5ex>[l] \ar@<0.5ex>[ur]
}\]
\end{Ex}
\begin{proof}
Step(I) Put $\sigma(\zeta)=\zeta^m$. By easy calculations, we can check $G=\langle\alpha,\beta\mid\alpha^{2n}=\beta^2=1,\ \beta\alpha=\alpha^{-m}\beta\rangle$. Thus $H:=G\cap GL_2(l)=\langle\alpha\rangle\cong C_{2n}$ holds and so $H\subseteq SL_2(l)$. Let $\simp lH=\{W_j\}_{j\in\mathbb{Z}/2n\mathbb{Z}}$. Here, $W_j=l$ as an $l$-linear space and $\alpha\cdot-=\zeta^j\colon W_j\to W_j$. Then the map $W_{-j}\to\beta\otimes W_j;\lambda\mapsto\beta\otimes\sigma(\lambda)$ is an $lH$-isomorphism, so $\beta\otimes W_j=W_{-j}$ holds. Since $G=H\sqcup\beta H$, the orbit decomposition of $\simp lH$ with respect to the $G$-action is given by
\[\simp lH=\{W_0\}\sqcup\bigsqcup_{j=1}^{n-1}\{W_j, W_{-j}\}\sqcup\{W_n\}.\]
Thus by Theorem \ref{tab}, $\simp l*G$ can be written as $\{V_0, V_{\pm1}, \cdots, V_{\pm(n-1)}, V_n\}$, where $V_0, V_{\pm j}$ and $V_n$ correspond to $\{W_0\}, \{W_j, W_{-j}\}$ and $\{W_n\}$ respectively.

Step(I\hspace{-1.2pt}I) More explicitly, there exist $a_0, a_{\pm j}, a_n, b_0, b_{\pm j}, b_n\in\mathbb{N}\ (1\leq j\leq n-1)$ satisfying
\begin{gather}
a_0b_0=2a_{\pm j}b_{\pm j}=a_nb_n=[l:k]=2\label{Cab}\\
\Res V_0=W_0^{\oplus a_0},\ \Res V_{\pm j}=(W_j\oplus W_{-j})^{\oplus a_{\pm j}},\ \Res V_n=W_n^{\oplus a_n}\notag\\
\Ind W_0=V_0^{\oplus b_0},\ \Ind W_j=\Ind W_{-j}=V_{\pm j}^{\oplus b_{\pm j}},\ \Ind W_n=V_n^{\oplus b_n}.\notag
\end{gather}
Let us calculate $a$'s. By \eqref{Cab}, we have $a_{\pm j}=b_{\pm j}=1$. As in the previous case, we have $V_0=l$, $a_0=1$ and $b_0=2$. Here, there exists $V\in\simp l*G$ such that $\Res V=W_n$ holds. In fact the $H$-action on $W_n$ is given by the group homomorphism $H\to l^\times;\alpha\mapsto -1$, and this can be extended to $G\to l^\times\rtimes\Gal(l/k);\alpha\mapsto(-1,{\rm id}),\ \beta\mapsto(1,\sigma)$ (see Proposition \ref{existV}). Thus we have $V_n=V$ and hence $a_n=1$, $b_n=2$. From this, we can check $\Cl(R)\cong\{V_0, V_n\}$ easily.

Step(I\hspace{-1.2pt}I\hspace{-1.2pt}I) Remark that the minimal projective resolution of $W_j\in\mod S*H$ is given by
\[0\longrightarrow S\otimes_lW_j\longrightarrow S\otimes_l(W_{j-1}\oplus W_{j+1})\longrightarrow S\otimes_lW_j\longrightarrow W_j\longrightarrow 0.\]
By Theorem \ref{ARseqsum}, almost split sequences and fundamental sequences in $\C(R)$ have the following forms, where we put $M_i:=M(V_i)\ (i=0,n)$ and $M_{\pm j}:=M(V_{\pm j})$.
\[\xymatrix@R=1mm@!C=18mm{
0 \ar[r] & M_0 \ar[r] & M_{\pm 1} \ar[r] & M_0\\
0 \ar[r] & M_{\pm 1} \ar[r] & M_0^{\oplus 2}\oplus M_{\pm 2} \ar[r] & M_{\pm 1} \ar[r] & 0\\
0 \ar[r] & M_{\pm i} \ar[r] & M_{\pm(i-1)}\oplus M_{\pm (i+1)} \ar[r] & M_{\pm i} \ar[r] & 0\ (1<i<n-1)\\
0 \ar[r] & M_{\pm(n-1)} \ar[r] & M_{\pm(n-2)}\oplus M_n^{\oplus 2} \ar[r] & M_{\pm(n-1)} \ar[r] & 0\\
0 \ar[r] & M_n \ar[r] & M_{\pm(n-1)} \ar[r] & M_n \ar[r] & 0
}\]
Using these sequences, we can draw $\AR(\C(R))$.
\end{proof}

\begin{Ex}($d=2$, Gorenstein, $\widetilde{A}_{11}$ and $\widetilde{BC}$)\label{typeBC}
Let $n$ be a positive integer and $l/k$ a field extension with $\ch k$ not dividing $2n$ and $[l:k]=2$. Write $2n=2^eN$ where $e$ is a positive integer and $N$ is an odd integer. We assume $\zeta_{2n}\in l$ and the image of the norm map $N_{l/k}\colon l\to k$ does not contain $-1\in k$. For example, $\mathbb{C}/\mathbb{R}$ and $\mathbb{Q}(\zeta_{2n})/\mathbb{Q}(\cos\frac{\pi}{n})$ satisfy these conditions. We denote $\Gal(l/k)=\{{\rm id},\sigma\}$. Let \[\alpha:=\Big(\Big[\begin{smallmatrix}\zeta_{2n}&0\\0&\zeta_{2n}^{-1}\end{smallmatrix}\Big], {\rm id}\Big),\ \beta:=\Big(\Big[\begin{smallmatrix}0&\zeta_{2^e}\\1&0\end{smallmatrix}\Big],\sigma\Big)\in GL_2(l)\rtimes\Gal(l/k)\]
and $G:=\langle\alpha,\beta\rangle\subseteq GL_2(l)\rtimes\Gal(l/k)$. Then we can write $\ind(\C(R))$ as $\{M_0=R, M_{\pm1}, \cdots, M_{\pm(n-1)}, M_n\}$ and draw $\AR(\C(R))$ as follows.

\[
\begin{array}{c|c|c|c}
\text{Case}&\AR(\C(R))&\text{Type}&\Cl(R)\\ \hline\hline
n=1&\xymatrix{M_0 \ar@<0.5ex>[r]^-{(4,1)} & M_1 \ar@<0.5ex>[l]^-{(1,4)}}&\widetilde{A}_{11}&1\\ \hline
n\geq2&\xymatrix{M_0 \ar@<0.5ex>[r]^-{(2,1)} & M_{\pm1} \ar@<0.5ex>[l]^-{(1,2)} \ar@<0.5ex>[r] & \cdots \ar@<0.5ex>[l] \ar@<0.5ex>[r] & M_{\pm(n-1)} \ar@<0.5ex>[l] \ar@<0.5ex>[r]^-{(2,1)} & M_n \ar@<0.5ex>[l]^-{(1,2)}}&\widetilde{BC}_n&1
\end{array}
\]
Compare with $\AR(\C(S^H))$.
\[\xymatrix@R=1mm{
& N_1 \ar@<0.5ex>[dl] \ar@<0.5ex>[r] & \cdots \ar@<0.5ex>[l] \ar@<0.5ex>[r] & N_{n-1} \ar@<0.5ex>[l] \ar@<0.5ex>[dr] \\
N_0 \ar@<0.5ex>[ur] \ar@<0.5ex>[dr] & & & &N_n \ar@<0.5ex>[ul] \ar@<0.5ex>[dl] \\
& N_{-1} \ar@<0.5ex>[ul] \ar@<0.5ex>[r] & \cdots \ar@<0.5ex>[l] \ar@<0.5ex>[r] & N_{-n+1} \ar@<0.5ex>[l] \ar@<0.5ex>[ur]
}\]
\begin{proof}
Step(I) By our assumption $-1\notin \Im N_{l/k}$, we can deduce $\sigma(\zeta_{2^e})=\zeta_{2^e}^{-1}$. Put $\sigma(\zeta_{2n})=\zeta_{2n}^m$. By easy calculations, we can check $G=\langle\alpha,\beta\mid\alpha^{2n}=1,\ \alpha^N=\beta^2,\ \beta\alpha=\alpha^{-m}\beta\rangle$. Thus $H:=G\cap GL_2(l)=\langle\alpha\rangle\cong C_{2N}$ holds and so $H\subseteq SL_2(l)$. Let $\simp lH=\{W_j\}_{j\in\mathbb{Z}/{2N}\mathbb{Z}}$. Here, $W_j=l$ as an $l$-linear space and $\alpha\cdot-=\zeta_{2N}^j\colon W_j\to W_j$. Then the map $W_{-j}\to\beta\otimes W_j;\lambda\mapsto\beta\otimes\sigma(\lambda)$ is an $lH$-isomorphism, so $\beta\otimes W_j=W_{-j}$ holds. Since $G=H\sqcup\beta H$, the orbit decomposition of $\simp lH$ with respect to the $G$-action is given by
\[\simp lH=\{W_0\}\sqcup\bigsqcup_{j=1}^{n-1}\{W_j, W_{-j}\}\sqcup\{W_n\}.\]
Thus by Theorem \ref{tab}, $\simp l*G$ can be written as $\{V_0, V_{\pm1}, \cdots, V_{\pm(n-1)}, V_n\}$, where $V_0, V_{\pm j}$ and $V_n$ correspond to $\{W_0\}, \{W_j, W_{-j}\}$ and $\{W_n\}$ respectively.

Step(I\hspace{-1.2pt}I) More explicitly, since $[l:k]=2$, we have
\[\Res V_0=W_0^{\oplus 1 {\rm or} 2},\ \Res V_{\pm j}=W_j\oplus W_{-j},\ \Res V_n=W_n^{\oplus 1 {\rm or} 2}.\]
As usual, $V_0$ is the trivial representation and $\Res V_0=W_0$ holds. In addition, we have $\Res V_n=W_n^{\oplus2}$. Otherwise, we have $\Res V_n=W_n$. Since the $H$-action on $W_n$ is given by $H\to l^\times;\alpha\mapsto-1$, Proposition \ref{existV} implies that there exists a group homomorphism $\rho:G\to l^\times\rtimes\Gal(l/k)$ such that $\rho(\alpha)=(-1,{\rm id})$ and $\rho(\beta)=(\lambda,\sigma)$ hold for some $\lambda\in l^\times$. Then we have $(-1,{\rm id})=\rho(\alpha^N)=\rho(\beta^2)=(N_{l/k}(\lambda),{\rm id})$, which contradicts to our assumption $-1\notin \Im N_{l/k}$. From this, we can check $\Cl(R)\cong\{V_0\}$ easily.

Step(I\hspace{-1.2pt}I\hspace{-1.2pt}I) Remark that the minimal projective resolution of $W_j\in\mod S*H$ is given by
\[0\longrightarrow S\otimes_lW_j\longrightarrow S\otimes_lW_{j+1}^{\oplus2}\longrightarrow S\otimes_lW_j\longrightarrow W_j\longrightarrow 0.\]
By Theorem \ref{ARseqsum}, almost split sequences and fundamental sequences in $\C(R)$ have the following forms, where we put $M_i:=M(V_i)\ (i=0,n)$ and $M_{\pm j}:=M(V_{\pm j})$.

(Case 1) $n=1$.
\[\xymatrix@R=1mm@!C=18mm{
0 \ar[r] & M_0 \ar[r] & M_1 \ar[r] & M_0\\
0 \ar[r] & M_1 \ar[r] & M_0^{\oplus4} \ar[r] & M_1 \ar[r] & 0
}\]

(Case 2) $n\geq2$.
\[\xymatrix@R=1mm@!C=18mm{
0 \ar[r] & M_0 \ar[r] & M_{\pm 1} \ar[r] & M_0\\
0 \ar[r] & M_{\pm 1} \ar[r] & M_0^{\oplus 2}\oplus M_{\pm 2} \ar[r] & M_{\pm 1} \ar[r] & 0\\
0 \ar[r] & M_{\pm i} \ar[r] & M_{\pm(i-1)}\oplus M_{\pm (i+1)} \ar[r] & M_{\pm i} \ar[r] & 0\ (1<i<n-1)\\
0 \ar[r] & M_{\pm(n-1)} \ar[r] & M_{\pm(n-2)}\oplus M_n \ar[r] & M_{\pm(n-1)} \ar[r] & 0\\
0 \ar[r] & M_n \ar[r] & M_{\pm(n-1)}^{\oplus 2} \ar[r] & M_n \ar[r] & 0
}\]
Using these sequences, we can draw $\AR(\C(R))$.
\end{proof}
\end{Ex}

Next we see an example whose Auslander-Reiten quiver has $(3,1)$ and $(1,3)$ as valuations.

\begin{Ex}($d=2$, Gorenstein, $\widetilde{G}_{22}$)\label{typeG22}
Let $l/k$ be a Galois extension with $[l:k]=3$ and $\ch k\neq2,3$. We assume $\zeta:=\zeta_{24}\in k$. For example, $\mathbb{Q}(\zeta_{72})/\mathbb{Q}(\zeta_{24})$ and $\mathbb{F}_{5^6}/\mathbb{F}_{5^2}$ satisfy these conditions. We let
\[\alpha:=\Big(\Big[\begin{smallmatrix} i&0\\0&-i\end{smallmatrix}\Big], {\rm id}\Big),\ \beta:=\Big(\Big[\begin{smallmatrix} 0&1\\-1&0\end{smallmatrix}\Big], {\rm id}\Big),\ \gamma:=\Big(\frac{1}{\sqrt{2}}\Big[\begin{smallmatrix} \zeta^7&\zeta^7\\\zeta^{13}&\zeta\end{smallmatrix}\Big], \sigma\Big)\in GL_2(l)\rtimes\Gal(l/k)\]
and $G:=\langle\alpha,\beta,\gamma\rangle\subseteq GL_2(l)\rtimes\Gal(l/k)$ where $\Gal(l/k)=\{{\rm id},\sigma,\sigma^2\}$. Then we can write $\ind(\C(R))$ as $\{M_0=R, M_1, M'\}$ and draw $\AR(\C(R))$, which is of type $\widetilde{G}_{22}$, as follows. In addition, $\Cl(R)=1$ holds.
\[\xymatrix{
M_0 \ar@<0.5ex>[r] & M' \ar@<0.5ex>[r]^{(3,1)} \ar@<0.5ex>[l] & M_1 \ar@<0.5ex>[l]^{(1,3)}
}\]
Compare with $\AR(\C(S^H))$.
\[\xymatrix@R=1mm{
& & N_{0,1} \ar@<0.5ex>[dl]\\
N_{0,0} \ar@<0.5ex>[r] & N' \ar@<0.5ex>[r] \ar@<0.5ex>[dr] \ar@<0.5ex>[ur] \ar@<0.5ex>[l] & N_{1,1} \ar@<0.5ex>[l]\\
& & N_{1,0} \ar@<0.5ex>[ul]
}\]
\begin{proof}
Step(I) Easy calculations show $G=\langle\alpha,\beta,\gamma\mid\alpha^4=1, \alpha^2=\beta^2, \gamma^3=1, \beta\alpha=\alpha\beta^3,\gamma\alpha=\beta\gamma,\gamma\beta=\alpha\beta\gamma\rangle$. Thus $H:=G\cap GL_2(l)=\langle\alpha,\beta\rangle$ holds and so $H\subseteq SL_2(l)$. Observe that this $H$ is binary dihedral. We can see that $H\to C_2\times C_2;\alpha\mapsto(1,0),\beta\mapsto(1,1)$ gives the Abelianization of $H$. Let $\{W_{j,k}\}_{j,k\in\mathbb{Z}/2\mathbb{Z}}\subseteq\simp lH$ be one-dimensional representations where $W_j=l$ as an $l$-linear space and $\alpha\cdot-=(-1)^j,\beta\cdot-=(-1)^{j+k}\colon W_{j,k}\to W_{j,k}$. We let $W'\in\mod lH$ be the two-dimensional representation determined by the inclusion $H\hookrightarrow GL_2(l)$. Then we can verify $\simp lH=\{W_{0,0},W_{0,1},W_{1,0},W_{1,1},W'\}$. Now we can see that $W_{k,j+k}\to\gamma\otimes W_{j,k};\lambda\mapsto\gamma\otimes\sigma^{-1}(\lambda)$ is an $lH$-isomorphism, so $\gamma\otimes W_{j,k}=W_{k,j+k}$ holds. In addition, $\gamma\otimes W'=W'$ holds since $W'$ is the unique two-dimensional irreducible representation. Since $G=H\sqcup\gamma H\sqcup\gamma^2 H$, the orbit decomposition of $\simp lH$ is given by
\[\simp lH=\{W_{0,0}\}\sqcup\{W_{0,1},W_{1,0},W_{1,1}\}\sqcup\{W'\}.\]
Thus by Theorem \ref{tab}, $\simp l*G$ can be written as $\{V_0, V_1, V'\}$, where $V_0, V_1$ and $V'$ correspond to $\{W_0\},\{W_{0,1},W_{1,0},W_{1,1}\}$ and $\{W'\}$ respectively.

Step(I\hspace{-1.2pt}I) More explicitly, since $[l:k]=3$, we have
\[\Res V_0=W_0^{\oplus 1 {\rm or} 3},\ \Res V_1=W_{0,1}\oplus W_{1,0}\oplus W_{1,1},\ \Res V'=W'^{\oplus 1 {\rm or} 3}.\]
As usual, $V_0$ is the trivial representation and $\Res V_0=W_0$ holds. Since $\Res\mathfrak{n}/\mathfrak{n}^2=W'$, $V'$ must be $\mathfrak{n}/\mathfrak{n}^2$ and $\Res V'=W'$ holds. From this, we can check $\Cl(R)\cong\{V_0\}$ easily.

Step(I\hspace{-1.2pt}I\hspace{-1.2pt}I) Remark that the minimal projective resolutions of $W_{j,k},W'\in\mod S*H$ are given as follows.
\[0\longrightarrow S\otimes_lW_{j,k}\longrightarrow S\otimes_lW'\longrightarrow S\otimes_lW_{j,k}\longrightarrow W_{j,k}\longrightarrow 0\]
\[0\longrightarrow S\otimes_lW'\longrightarrow S\otimes_l(W_{0,0}\oplus W_{0,1}\oplus W_{1,0}\oplus W_{1,1})\longrightarrow S\otimes_lW'\longrightarrow W'\longrightarrow 0\]
By Theorem \ref{ARseqsum}, almost split sequences and fundamental sequences in $\C(R)$ have the following forms, where we put $M_i:=M(V_i)\ (i=0,1)$ and $M':=M(V')$.
\[\xymatrix@R=1mm@!C=15mm{
0 \ar[r] & M_0 \ar[r] & M' \ar[r] & M_0\\
0 \ar[r] & M_1 \ar[r] & M'^{\oplus3} \ar[r] & M_1 \ar[r] & 0\\
0 \ar[r] & M' \ar[r] & M_0\oplus M_1 \ar[r] & M' \ar[r] & 0
}\]
Using these sequences, we can draw $\AR(\C(R))$.
\end{proof}
\end{Ex}

Now we see that the technique of tensor product representations is also useful to determine the decomposition laws. Observe that the ring in this example is not Gorenstein.

\begin{Ex}($d=2$, non-Gorenstein)\label{nongor}
We denote by $\sigma_i\ (i\in(\mathbb{Z}/8\mathbb{Z})^\times)$ the element of $\Gal(\mathbb{Q}(\zeta)/\mathbb{Q})\ (\zeta:=\zeta_8)$ determined by $\sigma_i(\zeta)=\zeta^i$. Let
\[\alpha:=\Big(\Big[\begin{smallmatrix} 0&\zeta\\1&0\end{smallmatrix}\Big], \sigma_3\Big),\ \beta:=\Big(\Big[\begin{smallmatrix} 0&\zeta^7\\\zeta&0\end{smallmatrix}\Big], \sigma_5\Big)\in GL_2(\mathbb{Q}(\zeta))\rtimes\Gal(\mathbb{Q(\zeta)}/\mathbb{Q})\]
and $G:=\langle\alpha,\beta\rangle\subseteq GL_2(\mathbb{Q}(\zeta))\rtimes\Gal(\mathbb{Q(\zeta)}/\mathbb{Q})$. Then we can write $\ind(\C(R))$ as $\{M_0=R, M_4=\omega, M_{1,3}, M_{2,6}, M_{5,7}\}$ and draw $\AR(\C(R))$ as follows. In addition, $\Cl(R)\cong C_2$ holds.
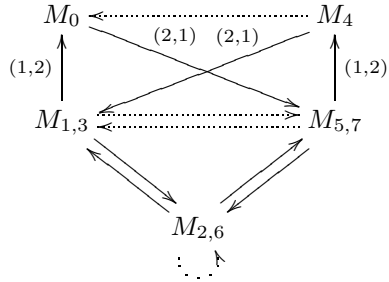
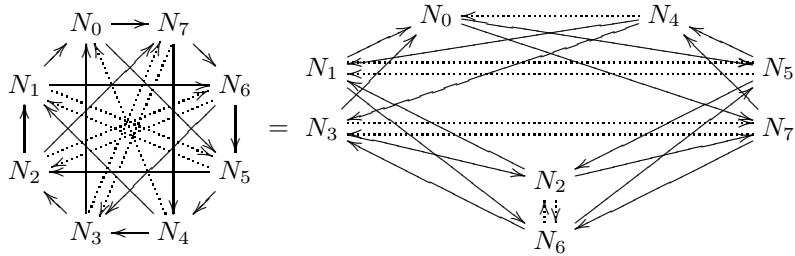
\[\xymatrix{
M_0 \ar[drr]^(0.37){(2,1)} & & M_4 \ar@<0.5ex>[dll]_(0.32){(2,1)} \ar@{.>}[ll]\\
M_{1,3} \ar[u]^-{(1,2)} \ar@<0.5ex>[dr] \ar@<0.5ex>@{.>}[rr] & & M_{5,7} \ar[u]_-{(1,2)} \ar@<0.5ex>[dl] \ar@<0.5ex>@{.>}[ll] \\
 & M_{2,6} \ar@<0.5ex>[ul] \ar@<0.5ex>[ur] \ar@(dl,dr)@{.>} & 
}\]
Compare with $\AR(\C(S^H))$.
\[
\begin{xy}
(-5.74,13.86)*+{N_0}="0",
(-13.86,5.74)*+{N_1}="1",
(-13.86,-5.74)*+{N_2}="2",
(-5.74,-13.86)*+{N_3}="3",
(5.74,-13.86)*+{N_4}="4",
(13.86,-5.74)*+{N_5}="5",
(13.86,5.74)*+{N_6}="6",
(5.74,13.86)*+{N_7}="7",
\ar "0";"7",
\ar "7";"6",
\ar "6";"5",
\ar "5";"4",
\ar "4";"3",
\ar "3";"2",
\ar "2";"1",
\ar "1";"0",

\ar "0";"5",
\ar "5";"2",
\ar "2";"7",
\ar "7";"4",
\ar "4";"1",
\ar "1";"6",
\ar "6";"3",
\ar "3";"0",

\ar@{.>} "4";"0",
\ar@<0.5ex>@{.>} "1";"5",
\ar@<0.5ex>@{.>} "5";"1",
\ar@<0.5ex>@{.>} "3";"7",
\ar@<0.5ex>@{.>} "7";"3",
\ar@<0.5ex>@{.>} "2";"6",
\ar@<0.5ex>@{.>} "6";"2",
\end{xy}
=
\begin{xy}
(-15,15)*+{N_0}="0",
(-30,8)*+{N_1}="1",
(0,-7)*+{N_2}="2",
(-30,0)*+{N_3}="3",
(15,15)*+{N_4}="4",
(30,8)*+{N_5}="5",
(0,-15)*+{N_6}="6",
(30,0)*+{N_7}="7",
\ar "0";"7",
\ar "7";"6",
\ar "6";"5",
\ar "5";"4",
\ar "4";"3",
\ar "3";"2",
\ar "2";"1",
\ar "1";"0",

\ar "0";"5",
\ar "5";"2",
\ar "2";"7",
\ar "7";"4",
\ar "4";"1",
\ar "1";"6",
\ar "6";"3",
\ar "3";"0",

\ar@{.>} "4";"0",
\ar@<0.5ex>@{.>} "1";"5",
\ar@<0.5ex>@{.>} "5";"1",
\ar@<0.5ex>@{.>} "3";"7",
\ar@<0.5ex>@{.>} "7";"3",
\ar@<0.5ex>@{.>} "2";"6",
\ar@<0.5ex>@{.>} "6";"2",
\end{xy}
\]

\begin{proof}
Step(I) Easy calculations such as $\alpha^2=\Big(\Big[\begin{smallmatrix} \zeta&0\\0&\zeta^3\end{smallmatrix}\Big], {\rm id}\Big)$ and $\beta^2=\Big(\Big[\begin{smallmatrix} -1&0\\0&-1\end{smallmatrix}\Big], {\rm id}\Big)$ show $G=\langle\alpha,\beta\mid\alpha^{16}=1, \alpha^8=\beta^2, \beta\alpha=\alpha^7\beta\rangle$. Thus $H:=G\cap GL_2(\mathbb{Q}(\zeta))=\langle\alpha^2\rangle\cong C_8$ holds and so $H\subseteq GL_2(\mathbb{Q}(\zeta))$ is small.

Let $\simp\mathbb{Q}(\zeta)H=\{W_j\}_{j\in\mathbb{Z}/8\mathbb{Z}}$. Here, $W_j=\mathbb{Q}(\zeta)$ as a $\mathbb{Q}(\zeta)$-linear space and $\alpha^2\cdot-=\zeta^j\colon W_j\to W_j$. By easy calculations as in Example \ref{typeCL}, we have $\alpha\otimes W_j=\beta\otimes W_j=W_{3j}$. Since $G=H\sqcup\alpha H\sqcup\beta H\sqcup\alpha\beta H$, the orbit decomposition of $\simp\mathbb{Q}(\zeta)H$ is given by
\[\simp\mathbb{Q}(\zeta)H=\{W_0\}\sqcup\{W_4\}\sqcup\{W_1, W_3\}\sqcup\{W_2, W_6\}\sqcup\{W_5, W_7\}.\]
Thus by Theorem \ref{tab}, $\simp\mathbb{Q}(\zeta)*G$ can be written as $\{V_0, V_4, V_{1,3}, V_{2,6}, V_{5,7}\}$, where $V_i$ and $V_{j,j'}$ correspond to $\{W_i\}$ and $\{W_j, W_{j'}\}$ respectively.

Step(I\hspace{-1.2pt}I) More explicitly, since $[\mathbb{Q}(\zeta):\mathbb{Q}]=4$, we have
\[\Res V_0=W_0^{\oplus 1 {\rm or} 2 {\rm or} 4},\ \Res V_4=W_4^{\oplus 1 {\rm or} 2 {\rm or} 4},\ \Res V_{1,3}=(W_1\oplus W_3)^{\oplus 1 {\rm or} 2},\]
\[\Res V_{2,6}=(W_2\oplus W_6)^{\oplus 1 {\rm or} 2},\ \Res V_{5,7}=(W_5\oplus W_7)^{\oplus 1 {\rm or} 2}.\]
As usual, $V_0$ is the trivial representation and $\Res V_0=W_0$ holds. Since $\Res\mathfrak{n}/\mathfrak{n}^2=W_1\oplus W_3$, $V_{1,3}$ must be $\mathfrak{n}/\mathfrak{n}^2$ and $\Res V_{1,3}=W_1\oplus W_3$ holds. Now let us calculate $\bigwedge^2 V_{1,3}$. Since
\[\Res\bigwedge^2V_{1,3}=\bigwedge^2\Res V_{1,3}=\bigwedge^2(W_1\oplus W_3)=W_4,\]
$V_4$ must be $\bigwedge^2V_{1,3}$ and $\Res V_4=W_4$ holds. Similarly, since
\[\Res(V_{1,3}\otimes V_{1,3})=(W_1\oplus W_3)\otimes(W_1\oplus W_3)=W_2\oplus W_6\oplus W_4^{\oplus 2},\]
$V_{1,3}\otimes V_{1,3}$ must be $V_{2,6}\oplus V_4^{\oplus 2}$ and $\Res V_{2,6}=W_2\oplus W_6$ holds. By the same argument, $V_{1,3}\otimes V_{2,6}=V_{1,3}\oplus V_{5,7}$ and $\Res V_{5,7}=W_5\oplus W_7$ hold. From this, we can check $\Cl(R)\cong\{V_0,V_4\}$ easily.

Step(I\hspace{-1.2pt}I\hspace{-1.2pt}I) Remark that the minimal projective resolution of $W_j\in\mod S*H$ is given by
\[0\longrightarrow S\otimes_{\mathbb{Q}(\zeta)}W_{j+4}\longrightarrow S\otimes_{\mathbb{Q}(\zeta)}(W_{j+1}\oplus W_{j+3})\longrightarrow S\otimes_{\mathbb{Q}(\zeta)}W_j\longrightarrow W_j\longrightarrow 0.\]
By Theorem \ref{ARseqsum}, almost split sequences and fundamental sequences in $\C(R)$ have the following forms, where we put $M_i:=M(V_i)\ (i=0,4)$ and $M_{j,j'}:=M(V_{j,j'})\ ((j,j')=(1,3),(2,6),(5,7))$.
\[\xymatrix@R=1mm@!C=15mm{
0 \ar[r] & M_4 \ar[r] & M_{1,3} \ar[r] & M_0\\
0 \ar[r] & M_0 \ar[r] & M_{5,7} \ar[r] & M_4 \ar[r] & 0\\
0 \ar[r] & M_{5,7} \ar[r] & M_4^{\oplus2}\oplus M_{2,6} \ar[r] & M_{1,3} \ar[r] & 0\\
0 \ar[r] & M_{2,6} \ar[r] & M_{1,3}\oplus M_{5,7} \ar[r] & M_{2,6} \ar[r] & 0\\
0 \ar[r] & M_{1,3} \ar[r] & M_0^{\oplus2}\oplus M_{2,6} \ar[r] & M_{5,7} \ar[r] & 0
}\]
Using these sequences, we can draw $\AR(\C(R))$ and confirm $\Cl(R)=\{[M_0],[M_4]\}$ easily.
\end{proof}
\end{Ex}

\subsection{Higher-dimensional examples}
%%%%%%%%%%%%%%%%%%%%%%%%%%%%%%%%%%%%%%%%%%%%%%%%%%%%%%%%%%%%%

Finally, we see higher-dimensional examples. The first example is a higher-dimensional generalization of Examples \ref{typeCL} and \ref{typeC}.

\begin{Ex}($d\geq2$, Gorenstein, isolated)
Let $d\geq2$ and $n\geq2$ be integers. Put $N:=\frac{n^d-1}{n-1}$. Let $l/k$ be a cyclic Galois extension with $\ch k$ not dividing $dN$ and $[l:k]=d$. We assume $\zeta:=\zeta_N\in l$. For example, if $d=3$ and $n=2$, then $\mathbb{Q}(\zeta)/\mathbb{Q}(\sqrt{-7})$ and $\mathbb{F}_8/\mathbb{F}_2$ satisfy these conditions. We denote $\Gal(l/k)=\langle\sigma\rangle$. Let
\[\alpha:=\big({\rm diag}(\zeta,\zeta^n,\zeta^{n^2},\cdots,\zeta^{n^{d-1}}), {\rm id}\big),\ \beta:=\big([e_2\ e_3\ \cdots\ e_d\ e_1],\sigma\big)\in GL_d(l)\rtimes\Gal(l/k)\]
and $G:=\langle\alpha,\beta\rangle\subseteq GL_d(l)\rtimes\Gal(l/k)$, where $e_i\in l^d$ denotes the $i$-th unit vector. Then the cyclic group $C_d$ acts on $\mathbb{Z}/N\mathbb{Z}$ by $i\cdot j:=n^ij$ and $\ind(\C(R))\cong C_d\backslash(\mathbb{Z}/N\mathbb{Z})$ holds. For example, if $d=3$ and $n=2$, then we can write $\ind(\C(R))=\{M_0=R, M_{1,2,4}, M_{3,5,6}\}$ and draw $\AR(\C(R))$ as follows. In addition, $\Cl(R)=1$ holds.
\[\xymatrix@R=1mm{
& M_0 \ar@<0.5ex>[ddr]^-{(3,1)} & \\
&&\\
M_{1,2,4}\ar@<0.5ex>[rr]\ar@<0.5ex>[uur]^-{(1,3)}\ar@(dl,ul)& & M_{3,5,6} \ar@<0.5ex>[ll]^-{(2,2)}\ar@(ur,dr)
}\]
Compare with $\AR(\C(S^H))$.
\[
\begin{xy}
(0,15)*+{N_0}="0",
(-11.73,9.35)*+{N_1}="1",
(-14.62,-3.338)*+{N_2}="2",
(-6.51,-13.51)*+{N_3}="3",
(6.51,-13.51)*+{N_4}="4",
(14.62,-3.338)*+{N_5}="5",
(11.73,9.35)*+{N_6}="6",
\ar "0";"6",
\ar "6";"5",
\ar "5";"4",
\ar "4";"3",
\ar "3";"2",
\ar "2";"1",
\ar "1";"0",

\ar "0";"5",
\ar "5";"3",
\ar "3";"1",
\ar "1";"6",
\ar "6";"4",
\ar "4";"2",
\ar "2";"0",

\ar "0";"3",
\ar "3";"6",
\ar "6";"2",
\ar "2";"5",
\ar "5";"1",
\ar "1";"4",
\ar "4";"0",
\end{xy}
=
\begin{xy}
(0,15)*+{N_0}="0",
(-30,1)*+{N_1}="1",
(-30,-7)*+{N_2}="2",
(20,-15)*+{N_3}="3",
(-20,-15)*+{N_4}="4",
(30,-7)*+{N_5}="5",
(30,1)*+{N_6}="6",
\ar "0";"6",
\ar "6";"5",
\ar "5";"4",
\ar "4";"3",
\ar "3";"2",
\ar "2";"1",
\ar "1";"0",

\ar "0";"5",
\ar "5";"3",
\ar "3";"1",
\ar "1";"6",
\ar "6";"4",
\ar "4";"2",
\ar "2";"0",

\ar "0";"3",
\ar "3";"6",
\ar "6";"2",
\ar "2";"5",
\ar "5";"1",
\ar "1";"4",
\ar "4";"0",
\end{xy}
\]

\begin{proof}
Step(I) Put $\sigma(\zeta)=\zeta^m$. Easy calculations show $G=\langle\alpha,\beta\mid\alpha^N=\beta^d=1, \beta\alpha=\alpha^{mn^{d-1}}\beta\rangle$. Thus $H:=G\cap GL_d(l)=\langle\alpha\rangle\cong C_N$ holds and so $H\subseteq SL_d(l)$. Let $\simp lH=\{W_j\}_{j\in\mathbb{Z}/N\mathbb{Z}}$. Here, $W_j=l$ as an $l$-linear space and $\alpha\cdot-=\zeta^j\colon W_j\to W_j$. By easy calculations as in Example \ref{typeCL}, we have $\beta\otimes W_j=W_{nj}$. Since $G=H\sqcup\beta H\sqcup\cdots\sqcup\beta^{d-1}H$, we can conclude $\ind(\C(R))\cong C_d\backslash(\mathbb{Z}/N\mathbb{Z})$. Below, we put $d=3$ and $n=2$ and continue calculations. Now the orbit decomposition of $\simp lH$ is given by
\[\simp lH=\{W_0\}\sqcup\{W_1, W_2,W_4\}\sqcup\{W_3, W_5, W_6\}.\]
Thus by Theorem \ref{tab}, $\simp l*G$ can be written as $\{V_0, V_{1,2,4}, V_{3,5,6}\}$.

Step(I\hspace{-1.2pt}I) Since $[l:k]=3$, we obtain $\Res V_0=W_0, \Res V_{1,2,4}=W_1\oplus W_2\oplus W_4$ and $\Res V_{3,5,6}=W_3\oplus W_5\oplus W_6$ as usual. From this, we can check $\Cl(R)\cong\{V_0\}$.

Step(I\hspace{-1.2pt}I\hspace{-1.2pt}I) Remark that the minimal projective resolution of $W_j\in\mod S*H$ is given by
\[0\longrightarrow S\otimes_lW_j\longrightarrow S\otimes_l(W_{j+3}\oplus W_{j+5}\oplus W_{j+6})\longrightarrow S\otimes_l(W_{j+1}\oplus W_{j+2}\oplus W_{j+4})\longrightarrow S\otimes_lW_j\longrightarrow W_j\longrightarrow 0.\]
By Theorem \ref{ARseqsum}, 2-almost split sequences and 2-fundamental sequences in $\C(R)$ have the following forms, where we put $M_0:=M(V_0), M_{1,2,4}:=M(V_{1,2,4})$ and $M_{3,5,6}:=M(V_{3,5,6})$.
\[
\begin{gathered}[b]
\xymatrix@R=1mm@C=5mm{
0 \ar[r] & M_0 \ar[r] & M_{3,5,6} \ar[r] & M_{1,2,4} \ar[r] & M_0\\
0 \ar[r] & M_{1,2,4} \ar[r] & M_0^{\oplus3}\oplus M_{1,2,4}\oplus M_{3,5,6} \ar[r] & M_{1,2,4}\oplus M_{3,5,6}^{\oplus2} \ar[r] & M_{1,2,4} \ar[r] & 0\\
0 \ar[r] & M_{3,5,6} \ar[r] & M_{1,2,4}^{\oplus2}\oplus M_{3,5,6} \ar[r] & M_0^{\oplus3}\oplus M_{1,2,4}\oplus M_{3,5,6} \ar[r] & M_{3,5,6} \ar[r] & 0
}
\\[-\dp\strutbox]
\end{gathered}
\qedhere\]
\end{proof}
\end{Ex}

In the next example, we deal with a ring which is neither Gorenstein nor an isolated singularity.

\begin{Ex}($d=3$, non-Gorenstein, non-isolated)
Let $n\geq2$ be an integer and $l/k$ a field extension with $\ch k$ not dividing $2n$. We assume $\zeta:=\zeta_{2n}\in l\backslash k$ and $[l:k]=2$. We denote $\Gal(l/k)=\{{\rm id},\sigma\}$. Let
\[\alpha:=\bigg(\bigg[\begin{smallmatrix}\zeta&&\\&-1&\\&&\zeta^{-1}\end{smallmatrix}\bigg], {\rm id}\bigg),\ \beta:=\bigg(\bigg[\begin{smallmatrix}&&1\\&1&\\1&&\end{smallmatrix}\bigg],\sigma\bigg)\in GL_3(l)\rtimes\Gal(l/k)\]
and $G:=\langle\alpha,\beta\rangle\subseteq GL_3(l)\rtimes\Gal(l/k)$. Then we can write $\ind(\C(R))=\{M_0=R, M_{\pm1}, \cdots, M_{\pm(n-1)},M_n=\omega\}$ and $\AR(\C(R))$ is as below. In addition, $\Cl(R)\cong C_2$ holds.
\[
\begin{array}{c|c}
\text{Case}&\AR(\C(R))\\ \hline\hline
n\colon{\rm even}&\xymatrix@R=1mm{M_0 \ar@<0.5ex>[r]^-{(2,1)} \ar@<0.5ex>[dd] & M_{\pm1} \ar@<0.5ex>[l]^-{(1,2)} \ar@<0.5ex>[r] \ar@<0.5ex>[dd] \ar@<1.5ex>@{.>}[dd] & \cdots \ar@<0.5ex>[l] \ar@<0.5ex>[r] & M_{\pm(\frac{n}{2}-1)} \ar@<0.5ex>[l] \ar@<0.5ex>[dr] \ar@<0.5ex>[dd] \ar@<1.5ex>@{.>}[dd] & \\
&&&&M_{\pm\frac{n}{2}} \ar@<0.5ex>[ul] \ar@<0.5ex>[dl] \ar@(ur,dr) \ar@(dr,dl)@{.>} \\
M_n \ar@<0.5ex>[r]^-{(2,1)} \ar@<0.5ex>[uu] \ar@<1.5ex>@{.>}[uu] & M_{\pm(n-1)} \ar@<0.5ex>[l]^-{(1,2)} \ar@<0.5ex>[r] \ar@<0.5ex>[uu] \ar@<1.5ex>@{.>}[uu] & \cdots \ar@<0.5ex>[l] \ar@<0.5ex>[r]  & M_{\pm(\frac{n}{2}+1)} \ar@<0.5ex>[l] \ar@<0.5ex>[ur] \ar@<0.5ex>[uu] \ar@<1.5ex>@{.>}[uu] & }\\ \hline
n\colon{\rm odd}&\xymatrix@R=5mm{M_0 \ar@<0.5ex>[r]^-{(2,1)} \ar@<0.5ex>[d] & M_{\pm1} \ar@<0.5ex>[l]^-{(1,2)} \ar@<0.5ex>[r] \ar@<0.5ex>[d] \ar@<1.5ex>@{.>}[d] & \cdots \ar@<0.5ex>[l] \ar@<0.5ex>[r] & M_{\pm\frac{n-3}{2}} \ar@<0.5ex>[l] \ar@<0.5ex>[r] \ar@<0.5ex>[d] \ar@<1.5ex>@{.>}[d] & M_{\pm\frac{n-1}{2}} \ar@<0.5ex>[l] \ar@<0.5ex>[d]^-{(2,2)} \ar@<1.5ex>@{.>}[d] \\
M_n \ar@<0.5ex>[r]^-{(2,1)} \ar@<0.5ex>[u] \ar@<1.5ex>@{.>}[u] & M_{\pm(n-1)} \ar@<0.5ex>[l]^-{(1,2)} \ar@<0.5ex>[r] \ar@<0.5ex>[u] \ar@<1.5ex>@{.>}[u] & \cdots \ar@<0.5ex>[l] \ar@<0.5ex>[r] & M_{\pm\frac{n+3}{2}} \ar@<0.5ex>[l] \ar@<0.5ex>[r] \ar@<0.5ex>[u] \ar@<1.5ex>@{.>}[u] & M_{\pm\frac{n+1}{2}} \ar@<0.5ex>[l] \ar@<0.5ex>[u]^-{(2,2)} \ar@<1.5ex>@{.>}[u] }
\end{array}
\]
\begin{proof}
Step(I) Easy calculations show $G=\langle\alpha,\beta\mid\alpha^{2n}=\beta^2=1, \beta\alpha=\alpha\beta\rangle$. Thus $H:=G\cap GL_3(l)=\langle\alpha\rangle\cong C_{2n}$ holds and so $H\subseteq GL_3(l)$ is small. Let $\simp lH=\{W_j\}_{j\in\mathbb{Z}/2n\mathbb{Z}}$. Here, $W_j=l$ as an $l$-linear space and $\alpha\cdot-=\zeta^j\colon W_j\to W_j$. By easy calculations as in Example \ref{typeCL}, we have $\beta\otimes W_j=W_{-j}$. Since $G=H\sqcup\beta H$, the orbit decomposition of $\simp lH$ is given by
\[\simp lH=\{W_0\}\sqcup\bigsqcup_{j=1}^{n-1}\{W_j, W_{-j}\}\sqcup\{W_n\}.\]
Thus by Theorem \ref{tab}, $\simp l*G$ can be written as $\{V_0, V_{\pm1}, \cdots, V_{\pm(n-1)}, V_n\}$, where $V_i$ and $V_{\pm j}$ corresponds to $\{W_i\}$ and $\{W_j, W_{-j}\}$ respectively.

Step(I\hspace{-1.2pt}I) As usual, we obtain $\Res V_0=W_0, \Res V_{\pm j}=W_j\oplus W_{-j}\ (1\leq j\leq n-1)$ and $\Res V_n=W_n$ easily. From this, we can check $\Cl(R)\cong\{V_0,V_n\}$.

Step(I\hspace{-1.2pt}I\hspace{-1.2pt}I) Remark that the minimal projective resolution of $W_j\in\mod S*H$ is given by
\[0\longrightarrow S\otimes_lW_{j+n}\longrightarrow S\otimes_l(W_j\oplus W_{j+n-1}\oplus W_{j-n+1})\longrightarrow S\otimes_l(W_{j+1}\oplus W_{j+n}\oplus W_{j-1})\longrightarrow S\otimes_lW_j\longrightarrow W_j\longrightarrow 0.\]
By Theorem \ref{ARseqsum}, 2-almost split sequences and 2-fundamental sequences in $\C(R)$ have the following forms, where we put $M_i:=M(V_i)\ (i=0,n)$ and $M_{\pm j}:=M(V_{\pm j})\ (1\leq j\leq n-1)$.
\[\xymatrix@R=1mm@C=5mm{
0 \ar[r] & M_n \ar[r] & M_0\oplus M_{\pm(n-1)} \ar[r] & M_{\pm 1}\oplus M_n \ar[r] & M_0\\
0 \ar[r] & M_{\pm(n-1)} \ar[r] & M_{\pm1}\oplus M_{\pm(n-2)}\oplus M_n^{\oplus2} \ar[r] & M_0^{\oplus2}\oplus M_{\pm2}\oplus M_{\pm(n-1)} \ar[r] & M_{\pm 1} \ar[r] & 0\\
0 \ar[r] & M_{\pm(n-j)} \ar[r] & M_{\pm j}\oplus M_{\pm(n-j-1)}\oplus M_{\pm(n-j+1)} \ar[r] & M_{\pm(j-1)}\oplus M_{\pm(j+1)}\oplus M_{\pm(n-j)} \ar[r] & M_{\pm j} \ar[r] & 0\\
0 \ar[r] & M_{\pm(n-1)} \ar[r] & M_0^{\oplus2}\oplus M_{\pm2}\oplus M_{\pm(n-1)} \ar[r] & M_{\pm1}\oplus M_{\pm(n-2)}\oplus M_n^{\oplus2} \ar[r] & M_{\pm(n-1)} \ar[r] & 0\\
0 \ar[r] & M_n \ar[r] & M_{\pm 1}\oplus M_n \ar[r] & M_0\oplus M_{\pm(n-1)} \ar[r] & M_n \ar[r] & 0
}\]
\rightline{$(1<j<\frac{n-1}{2},\frac{n+1}{2}<j<n-1)$}

(Case 1) $n$ is even.
\[\xymatrix@R=1mm@C=8.45mm{
0 \ar[r] & M_{\pm\frac{n}{2}} \ar[r] & M_{\pm(\frac{n}{2}-1)}\oplus M_{\pm\frac{n}{2}}\oplus M_{\pm(\frac{n}{2}+1)} \ar[r] & M_{\pm(\frac{n}{2}-1)}\oplus M_{\pm\frac{n}{2}}\oplus M_{\pm(\frac{n}{2}+1)} \ar[r] & M_{\pm\frac{n}{2}} \ar[r] & 0
}\]

(Case 2) $n$ is odd.
\[\xymatrix@R=1mm@C=13.1mm{
0 \ar[r] & M_{\pm\frac{n+1}{2}} \ar[r] & {M_{\pm\frac{n+1}{2}}}^{\oplus2}\oplus M_{\pm\frac{n+3}{2}} \ar[r] & M_{\pm\frac{n-3}{2}}\oplus {M_{\pm\frac{n+1}{2}}}^{\oplus2} \ar[r] & M_{\pm\frac{n-1}{2}} \ar[r] & 0\\
0 \ar[r] & M_{\pm\frac{n-1}{2}} \ar[r] & M_{\pm\frac{n-3}{2}}\oplus {M_{\pm\frac{n+1}{2}}}^{\oplus2} \ar[r] & {M_{\pm\frac{n+1}{2}}}^{\oplus2}\oplus M_{\pm\frac{n+3}{2}} \ar[r] & M_{\pm\frac{n+1}{2}} \ar[r] & 0
}\]

In both cases, we can draw $\AR(\C(R))$.
\end{proof}
\end{Ex}

%%%%%%%%%%%%%%%%%%%%%%%%%%%%%%%%%%%%%%%%%%%%%%%%%%%%%%%%%%%%%
\section{Two-dimensional rings of finite Cohen-Macaulay type of equicharacteristic zero}
%%%%%%%%%%%%%%%%%%%%%%%%%%%%%%%%%%%%%%%%%%%%%%%%%%%%%%%%%%%%%

The goal of this section is to prove that two-dimensional Cohen-Macaulay complete local rings of finite Cohen-Macaulay type of equicharacteristic zero are precisely quotient singularities admitting field extensions.

\begin{Thm}\label{frt}
Let $k$ be a field with $\ch k=0$ and $(R,\mathfrak{m},k)$ a two-dimensional Cohen-Macaulay complete local ring of finite Cohen-Macaulay type. Then there exist a finite Galois extension $l/k$ and a finite subgroup $G\subseteq GL_2(l)\rtimes\Gal(l/k)$ such that the following conditions are satisfied.
\begin{enumerate}
\item There exists a ring isomorphism $R\cong l[[x,y]]^G$.
\item The subgroup $G\cap GL_2(l)\subseteq GL_2(l)$ is small.
\item The natural group homomorphism $G\to\Gal(l/k)$ is surjective.
\end{enumerate}
\end{Thm}

In the proof of Theorem \ref{frt}, the following inequality is the most technical part which is rather non-trivial than one used in the classical case. Observe that we cannot prove this inequality without our description of irreducible representations of skew group algebras (Theorem \ref{tab}).

\begin{Lem}\label{ineq}
Let $G'$ be a finite group acting on a field $l'$ and $N\subseteq G'$ a normal subgroup. Put $k:=l'^{G'}$ and $l:=l'^N$. Observe that $G:=G'/N$ acts on $l$. Assume that $|G'|$ is not divided by $\ch l$ and that there exists an element of $N$ different from the unit which acts on $l'$ trivially. In this setting, we have the following inequality.
\[\sharp\simp l*G<\sharp\simp l'*G'\]
\end{Lem}
\begin{proof}
Let $H:=\Ker(G\to\Gal(l/k))$ and $H':=\Ker(G'\to\Gal(l'/k))$ be normal subgroups of $G$ and $G'$. Observe that we have the following commutative diagram whose rows and columns are exact.
\[\xymatrix{
& 1 \ar[d] & 1 \ar[d] & 1 \ar[d] & \\
1 \ar[r] & H'\cap N \ar[r] \ar[d] & H' \ar[r] \ar[d] & H \ar[r] \ar[d] & 1 \\
1 \ar[r] & N \ar[r] \ar[d] & G' \ar[r] \ar[d] & G \ar[r] \ar[d] & 1 \\
1 \ar[r] & \Gal(l'/l) \ar[r] \ar[d] & \Gal(l'/k) \ar[r] \ar[d] & \Gal(l/k) \ar[r] \ar[d] & 1 \\
& 1 & 1 & 1 &
}\]

Let $e:=\frac{1}{|N|}\sum_{n\in N}n\in l'*G'$ be an idempotent. Then we can see that
\[e(l'*G')e=\bigg\{\sum_{g'\in G'}\lambda'_{g'}g'\in l'*G'\mid \lambda'_{g'}\in l, g'_1N=g'_2N\Rightarrow\lambda'_{g'_1}=\lambda'_{g'_2}\bigg\}\cong l*G\]
holds. Therefore we obtain an injective map $(l'*G')e\otimes_{l*G}-\colon\simp l*G\to\simp l'*G'$. From now on, we show that this map is not surjective. By Theorem \ref{tab}, we have the following bijection.
\[\simp l'*G'\cong\Gal(l'/k)\backslash\simp l'H'\]
Since we have the natural surjective ring homomorphism $l'H'\to l'H$, we may view $\mod l'H\subseteq\mod l'H'$. Recall that the action $\Gal(l'/k)\curvearrowright\mod l'H'$ is induced by the action $G'\curvearrowright\mod l'H'$. By easy calculations, we can see that the action $\Gal(l'/k)\curvearrowright\mod l'H'$ restricts to $\Gal(l'/k)\curvearrowright\mod l'H$, so we can regard $\Gal(l'/k)\backslash\simp l'H\subseteq\Gal(l'/k)\backslash\simp l'H'$. Since the surjective ring homomorphism $l'H'\to l'H$ between semisimple algebras is not an isomorphism due to $H'\cap N\neq1$, the inclusion map $\Gal(l'/k)\backslash\simp l'H\hookrightarrow\Gal(l'/k)\backslash\simp l'H'$ is not surjective. Thus it is enough to show that the map $\simp l*G\to\simp l'*G'\cong\Gal(l'/k)\backslash\simp l'H'$ factors through the inclusion map $\Gal(l'/k)\backslash\simp l'H\hookrightarrow\Gal(l'/k)\backslash\simp l'H'$.
\[\xymatrix@C=6em{
\simp l*G \ar^-{(l'*G')e\otimes_{l*G}-}[rr] \ar@{-}_\cong[d] & & \simp l'*G' \ar@{-}^\cong[d] \\
\Gal(l/k)\backslash\simp lH \ar@{.>}[r] & \Gal(l'/k)\backslash\simp l'H\ar@{^{(}->}[r] & \Gal(l'/k)\backslash\simp l'H'
}\]
Since the bijection $\simp l'*G'\cong\Gal(l'/k)\backslash\simp l'H'$ is obtained by the restriction functor ${\rm Res}_{\it H'}\colon\mod l'*G'\to\mod l'H'$, it is enough to show that ${\rm Res}_{\it H'}(l'*G')e\in\mod l'H\subseteq\mod l'H'$ holds. For $\sum_{h'\in H'}\mu'_{h'}h'\in l'H'$ and $(\lambda'*g')e\in(l'*G')e$, we have
\[\bigg(\sum_{h'\in H'}\mu'_{h'}h'\bigg)(\lambda'*g')e=\sum_{h'\in H'}(\mu'_{h'}\lambda'*h'g')e.\]
Observe that if $h'_1,h'_2\in H'$ satisfies $\overline{h_1'}=\overline{h_2'}\in H$, then $h_1'g'e=h_2'g'e$ holds. Thus if $\sum_{h'\in H'}\mu'_{h'}h'\in l'H'$ belongs to $\Ker(l'H'\to l'H)$, we have $\bigg(\sum_{h'\in H'}\mu'_{h'}h'\bigg)(\lambda'*g')e=0$. This means ${\rm Res}_{\it H'}(l'*G')e\in\mod l'H$. In conclusion, we obtain the desired inequality.
\end{proof}

The following fact due to \cite{Fle} is crucial to our main theorem.

\begin{Thm}\cite[1.7]{Fle}\label{Flenner}
For a two-dimensional complete local normal domain $(S,\mathfrak{n},l)$ with $\ch l=0$, the following conditions are equivalent.
\begin{enumerate}
\item The ring $S$ is regular.
\item The ring $S$ is pure, i.e. the natural group homomorphism $\pi_1^{et}(\Spec S\backslash \{\mathfrak{n}\})\to\pi_1^{et}(\Spec S)$ is an isomorphism.
\item Let $L:=\Frac S$ be the field of fractions. For an arbitrary a field extension $L'/L$ of finite degree, let $S'\subseteq L'$ be the integral closure of $S$ in $L'$. If the extension $S'/S$ is unramified in codimension one, then it is unramified.
\end{enumerate}
\end{Thm}

Remark that (3) is just a rephrasing of (2) (see, for instance, \cite[3.3.6]{Fu}). The implication $(1)\Rightarrow(3)$ is nothing but the purity of the branch locus (see \cite[1.4]{Au62}). Now we can prove Theorem \ref{frt} using the implication $(3)\Rightarrow(1)$. The proof appears, at first glance, similar to the classical one, but it is actually quite different.

\begin{proof}[Proof of Theorem \ref{frt}]
Observe that $R$ is a normal domain since is of finite Cohen-Macaulay type (see, for instance, \cite{AU86}\cite[7.12]{LW}). Let $K$ be the field of fractions of $R$ and fix an algebraic closure $\overline{K}$ of $K$. For a finite Galois extension $L/K$ in $\overline{K}$, we denote by $S_L$ the integral closure of $R$ in $L$. Remark that $S_L^{\Gal(L/K)}=R$ holds since $R$ is normal. By the finiteness of the integral closure, $S_L$ is finitely generated as an $R$-module, so is a Noetherian normal domain. In addition, $S_L\in\CM R$ holds. Moreover, since $R$ is complete local, so is $S_L$. Let $\mathfrak{n}_L$ be the maximal ideal of $S_L$ and $l_L:=S_L/\mathfrak{n}_L$. Since $l_L^{\Gal(L/K)}=k$ holds, the extension $l_L/k$ is Galois and the natural group homomorphism $\Gal(L/K)\to\Gal(l_L/k)$ is surjective by Artin's theorem. We let $T(L/K):=\Ker(\Gal(L/K)\to\Gal(l_L/k))$ be the inertia group. Consider the following condition $(*)$ for $L$.
\[(*)\ \text{The extension }S_L/R\text{ is unramified in codimension one.}\]
We let $X$ be the set of finite Galois extensions $L$ of $K$ in $\overline{K}$ satisfying condition $(*)$.

Take $L\in X$. By Theorem \ref{end}, we have $\End_R(S_L)\cong S_L*G_L$, where we put $G_L:=\Gal(L/K)$. This means $\add_RS_L\simeq\proj(S_L*G_L)^{\op}\simeq\proj S_L*G_L$ (see Proposition \ref{projskew}). Since $S_L*G_L/J_{S_L*G_L}=l_L*G_L$ holds as in Lemma \ref{Jac}, we have the following bijections.
\[\ind\add_RS_L\cong\ind\proj S_L*G_L\cong\simp l_L*G_L\]

We define a partial order $\prec$ on $X$ as
\[L\prec L':\Leftrightarrow L\subseteq L',\ T(L'/L)\neq1\ (L,L'\in X).\]
Take $L,L'\in X$ with $L\prec L'$. Observe that we have the following commutative diagram whose rows and columns are exact (see \cite[I.7.22]{S}).
\[\xymatrix{
& 1 \ar[d] & 1 \ar[d] & 1 \ar[d] & \\
1 \ar[r] & T(L'/L) \ar[r] \ar[d] & T(L'/K) \ar[r] \ar[d] & T(L/K) \ar[r] \ar[d] & 1 \\
1 \ar[r] & \Gal(L'/L) \ar[r] \ar[d] & \Gal(L'/K) \ar[r] \ar[d] & \Gal(L/K) \ar[r] \ar[d] & 1 \\
1 \ar[r] & \Gal(l_{L'}/l_L) \ar[r] \ar[d] & \Gal(l_{L'}/k) \ar[r] \ar[d] & \Gal(l_L/k) \ar[r] \ar[d] & 1 \\
& 1 & 1 & 1 &
}\]
Since $T(L'/L)\neq1$, we can apply Lemma \ref{ineq} and obtain $\sharp\simp(l_L*\Gal(L/K))<\sharp\simp(l_{L'}*\Gal(L'/K))$. Thus we obtain the following inequality.
\[\sharp\ind\add_RS_L<\sharp\ind\add_RS_{L'}\]

By the above arguments and the assumption $\sharp\ind\CM R<\infty$, we can conclude that the partial ordered set $(X,\prec)$ has a maximal element $\Omega\in X$. Now we show the condition (3) of Theorem \ref{Flenner} for $S_\Omega$. Take a field extension $\Omega'/\Omega$ of finite degree in $\overline{K}$ with the extension $S_{\Omega'}/S_\Omega$ unramified in codimension one. By taking the Galois closure, we may assume that the extension $\Omega'/K$ is Galois. Now the extension $S_{\Omega'}/R$ is unramified in codimension one since so are $S_{\Omega'}/S_\Omega$ and $S_\Omega/R$. Thus we have $\Omega'\in X$. By the maximality of $\Omega$, we can conclude $T(\Omega'/\Omega)=1$. Thus the extension $S_{\Omega'}/S_\Omega$ is unramified by Proposition \ref{inertia}. Therefore we obtain the desired condition and thus $S_\Omega$ is regular by Theorem \ref{Flenner}. By Cohen's structure theorem, we have $S_\Omega\cong l[[x,y]]\ (l:=l_\Omega)$. Since $S_\Omega^{\Gal(\Omega/K)}=R$ holds, there exists a finite subgroup $G\subseteq GL_2(l)\rtimes\Gal(l/k)$ such that $R\cong l[[x,y]]^G$ holds and $G\cap GL_2(l)\subseteq GL_2(l)$ is small by Theorem \ref{makesmall}.
\end{proof}

%%%%%%%%%%%%%%%%%%%%%%%%%%%%%%%%%%%%%%%%%%%%%%%%%%%%%%%%%%%%%
\section{Two-dimensional Gorenstein rings of finite Cohen-Macaulay type of equicharacteristic zero}\label{Gorquiv}
%%%%%%%%%%%%%%%%%%%%%%%%%%%%%%%%%%%%%%%%%%%%%%%%%%%%%%%%%%%%%

In this section, we determine all quivers that may appear as $\AR(\CM R)$ where $(R,\mathfrak{m},k)$ is a two-dimensional Gorenstein complete local ring of finite Cohen-Macaulay type with $\ch k=0$ (Theorem \ref{classfyquiver}). First, we consider the case where $k$ is algebraically closed. Then it is well-known that $R$ is isomorphic to $k[[x,y]]^G$ where $G\subseteq SL_2(k)$ is a finite subgroup. Here, we have the following Klein's classification of finite subgroups of $SL_2(k)$ up to conjugacy by the ADE Dynkin diagrams.
\[
\begin{array}{c|c}
	\text{Type}&\text{Generator(s)}\\ \hline\hline
	A_n&\begin{bmatrix}\zeta_{n+1}&0\\0&\zeta_{n+1}^{-1}\end{bmatrix}\\ \hline
	D_n&A_{2n-5}, \begin{bmatrix}0&\zeta_4\\ \zeta_4&0\end{bmatrix}\\ \hline
	E_6&D_4, \dfrac{1}{\sqrt{2}}\begin{bmatrix}\zeta_8&\zeta_8^3\\ \zeta_8&\zeta_8^7\end{bmatrix}\\ \hline
	E_7&E_6, \begin{bmatrix}\zeta_8^3&0\\ 0&\zeta_8^5\end{bmatrix}\\ \hline
	E_8&\dfrac{1}{\sqrt{5}}\begin{bmatrix}\zeta_5^4-\zeta_5&\zeta_5^2-\zeta_5^3\\ \zeta_5^2-\zeta_5^3&\zeta_5-\zeta_5^4\end{bmatrix}, \dfrac{1}{\sqrt{5}}\begin{bmatrix}\zeta_5^2-\zeta_5^4&\zeta_5^4-1\\ 1-\zeta_5&\zeta_5^3-\zeta_5\end{bmatrix}
\end{array}
\]

It is also well-known that the quiver $\AR(\CM k[[x,y]]^G)$ coincides with the double of the corresponding extended Dynkin diagram. Therefore in this case, we can say that quivers that may appear as $\AR(\CM R)$ are just doubles of simply laced extended Dynkin diagrams.

Next, towards Theorem \ref{classfyquiver}, we exhibit Auslander-Reiten quivers for some examples without proofs.

\begin{Ex}($d=2$, Gorenstein)\label{mathbbR}
Let $H\subseteq SL_2(\mathbb{C})$ be a finite subgroup and $G:=\langle H,\Gal(\mathbb{C}/\mathbb{R})\rangle\subseteq GL_2(\mathbb{C})\rtimes\Gal(\mathbb{C}/\mathbb{R})$. By the above list, $H$ is classified by the ADE Dynkin diagrams. We exhibit $\AR(\C(R))$ and $\Cl(R)$.
\[
\begin{array}{c|c|c|c}
\text{Type of }H&\AR(\C(R))&\text{Type of }\AR(\C(R))&\Cl(R)\\ \hline\hline
A_n&\xymatrix@R=5mm{& & R \ar@<0.5ex>[dll] \ar@<0.5ex>[drr] & & \\
\circ \ar@<0.5ex>[urr] \ar@<0.5ex>[r] & \circ \ar@<0.5ex>[l] \ar@<0.5ex>[r] & \cdots \ar@<0.5ex>[l] \ar@<0.5ex>[r] & \circ \ar@<0.5ex>[l] \ar@<0.5ex>[r] & \circ \ar@<0.5ex>[ull] \ar@<0.5ex>[l]}&\widetilde{A}_n&C_{n+1}\\ \hline
D_n\ (n\colon{\rm even})&\xymatrix@R=1mm{R \ar@<0.5ex>[dr] & & & & \circ \ar@<0.5ex>[dl]\\
& \circ \ar@<0.5ex>[ul] \ar@<0.5ex>[dl] \ar@<0.5ex>[r] & \cdots \ar@<0.5ex>[l] \ar@<0.5ex>[r] & \circ \ar@<0.5ex>[l]\ar@<0.5ex>[ur] \ar@<0.5ex>[dr] & \\
\circ \ar@<0.5ex>[ur] & & & & \circ \ar@<0.5ex>[ul]}&\widetilde{D}_n&C_2\times C_2\\ \hline
D_n\ (n\colon{\rm odd})&\xymatrix@R=1mm{R \ar@<0.5ex>[dr] & & & & \\
& \circ \ar@<0.5ex>[ul] \ar@<0.5ex>[dl] \ar@<0.5ex>[r] & \cdots \ar@<0.5ex>[l] \ar@<0.5ex>[r] & \circ \ar@<0.5ex>[l]\ar@<0.5ex>[r]^-{(2,1)} & \circ \ar@<0.5ex>[l]^-{(1,2)}\\
\circ \ar@<0.5ex>[ur] & & & &}&\widetilde{BD}_{n-1}&C_2\\ \hline
E_6&\xymatrix{R \ar@<0.5ex>[r] & \circ \ar@<0.5ex>[l] \ar@<0.5ex>[r] & \circ \ar@<0.5ex>[l] \ar@<0.5ex>[r]^-{(2,1)} & \circ \ar@<0.5ex>[l]^-{(1,2)} \ar@<0.5ex>[r] &\circ \ar@<0.5ex>[l]}&\widetilde{F}_{42}&1\\ \hline
E_7&\xymatrix@R=5mm{&&&\circ \ar@<0.5ex>[d]&&&\\
R \ar@<0.5ex>[r] & \circ \ar@<0.5ex>[l] \ar@<0.5ex>[r] & \circ \ar@<0.5ex>[l] \ar@<0.5ex>[r] & \circ \ar@<0.5ex>[l] \ar@<0.5ex>[r] \ar@<0.5ex>[u] & \circ \ar@<0.5ex>[l] \ar@<0.5ex>[r] & \circ \ar@<0.5ex>[l] \ar@<0.5ex>[r] & \circ \ar@<0.5ex>[l]}&\widetilde{E}_7&C_2\\ \hline
E_8&\xymatrix@R=5mm{&&&&&\circ \ar@<0.5ex>[d]&&\\
R \ar@<0.5ex>[r] & \circ \ar@<0.5ex>[l] \ar@<0.5ex>[r] & \circ \ar@<0.5ex>[l] \ar@<0.5ex>[r] & \circ \ar@<0.5ex>[l] \ar@<0.5ex>[r] & \circ \ar@<0.5ex>[l] \ar@<0.5ex>[r] & \circ \ar@<0.5ex>[l] \ar@<0.5ex>[r] \ar@<0.5ex>[u] & \circ \ar@<0.5ex>[l] \ar@<0.5ex>[r] & \circ \ar@<0.5ex>[l]}&\widetilde{E}_8&1
\end{array}
\]
We can see that if $H$ is of type $D_n$ ($n\colon$odd) or $E_6$, then the Auslander-Reiten quiver has non-trivial valuations and thus is different from the usual one.
\end{Ex}

\begin{Ex}($d=2$, Gorenstein, $\widetilde{BD}$)\label{typeBD}
Let $n\geq2$ be an integer and $l/k$ a Galois extension with $\ch k$ not dividing $2n$ and $[l:k]=2$ such that $l=k(\zeta)$ where $\zeta:=\zeta_{4n}$. We denote $\Gal(l/k)=\{{\rm id},\sigma\}$ and assume $\sigma(\zeta)=-\zeta^{-1}$. For example, $\mathbb{Q}(\zeta)/\mathbb{Q}(i\sin\frac{\pi}{2n})$ satisfies these conditions. We let
\[\alpha:=\Big(\Big[\begin{smallmatrix} 0&1\\-1&0\end{smallmatrix}\Big], {\rm id}\Big),\ \beta:=\Big(\Big[\begin{smallmatrix} 0&-\zeta^3\\\zeta&0\end{smallmatrix}\Big], \sigma\Big)\in GL_2(l)\rtimes\Gal(l/k)\]
and $G:=\langle\alpha,\beta\rangle\subseteq GL_2(l)\rtimes\Gal(l/k)$. Then we can write $\ind(\C(R))$ as $\{M_0=R, M_1, M_{2,3}, M'_1, \cdots, M'_{N-1}\}$ and draw $\AR(\C(R))$, which is of type $\widetilde{BD}_{n+1}$, as follows. In addition, $\Cl(R)\cong C_2$ holds.
\[\xymatrix@R=1mm{
M_0 \ar@<0.5ex>[dr] & & & & \\
& M'_1 \ar@<0.5ex>[ul] \ar@<0.5ex>[dl] \ar@<0.5ex>[r] & \cdots \ar@<0.5ex>[l] \ar@<0.5ex>[r] & M'_{n-1} \ar@<0.5ex>[l]\ar@<0.5ex>[r]^-{(2,1)} & M_{2,3} \ar@<0.5ex>[l]^-{(1,2)}\\
M_1 \ar@<0.5ex>[ur] & & & &
}\]
Compare with $\AR(\C(S^H))$.
\[\xymatrix@R=1mm{
N_0 \ar@<0.5ex>[dr] & & & & N_2 \ar@<0.5ex>[dl]\\
& N'_1 \ar@<0.5ex>[ul] \ar@<0.5ex>[dl] \ar@<0.5ex>[r] & \cdots \ar@<0.5ex>[l] \ar@<0.5ex>[r] & N'_{n-1} \ar@<0.5ex>[l]\ar@<0.5ex>[dr] \ar@<0.5ex>[ur] & \\
N_1 \ar@<0.5ex>[ur] & & & & N_3 \ar@<0.5ex>[ul]
}\]

\end{Ex}

Here, we exhibit the list of all extended Dynkin diagrams with their imaginary roots labeling the vertices (see, for instance, \cite{Ka}).
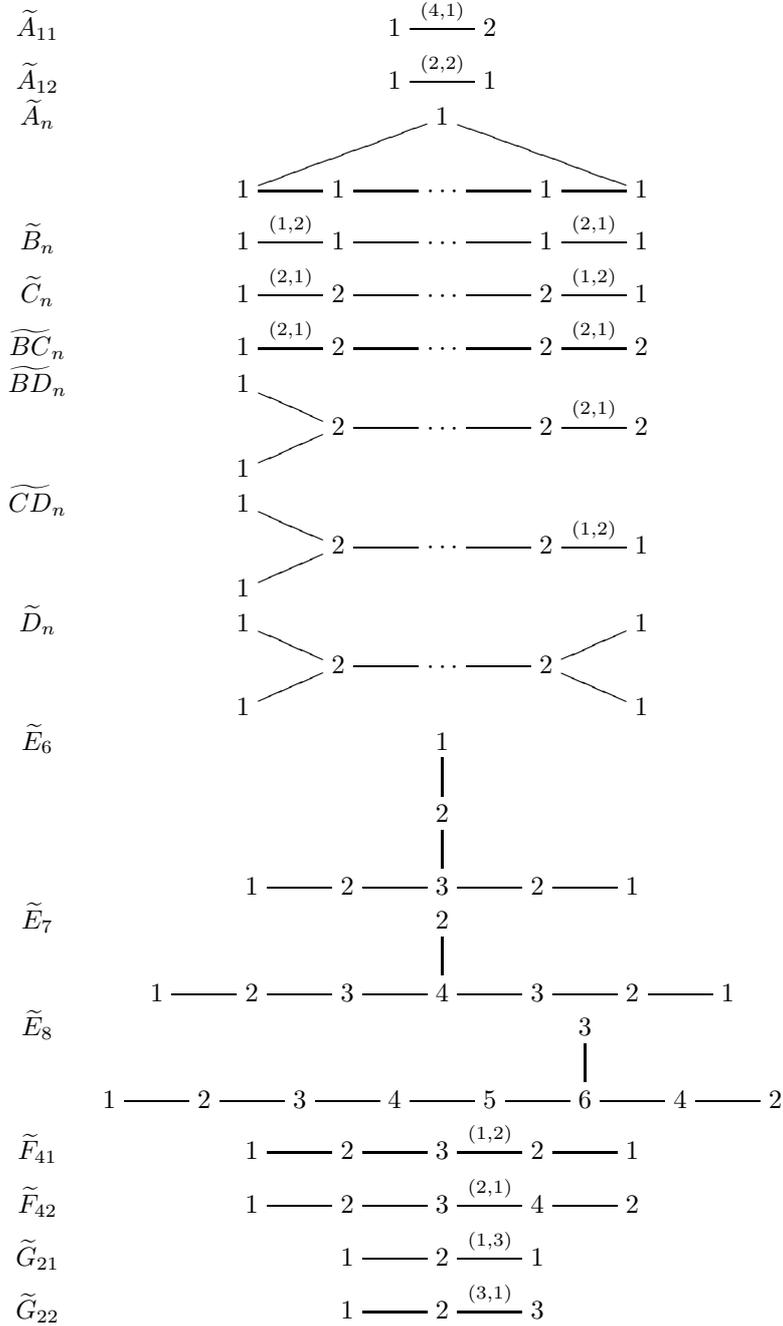
\[
\begin{array}{cc}
\widetilde{A}_{11}&\xymatrix{1 \ar@{-}^{(4,1)}[r] & 2}\\
\widetilde{A}_{12}&\xymatrix{1 \ar@{-}^{(2,2)}[r] & 1}\\
\widetilde{A}_n&\xymatrix@R=5mm{
& & 1 \ar@{-}[dll] \ar@{-}[drr] & & \\
1 \ar@{-}[r] & 1 \ar@{-}[r] & \cdots \ar@{-}[r] & 1 \ar@{-}[r] & 1
}\\
\widetilde{B}_n&\xymatrix{1 \ar@{-}^{(1,2)}[r] & 1 \ar@{-}[r] & \cdots \ar@{-}[r] & 1 \ar@{-}^{(2,1)}[r] & 1}\\
\widetilde{C}_n&\xymatrix{1 \ar@{-}^{(2,1)}[r] & 2 \ar@{-}[r] & \cdots \ar@{-}[r] & 2 \ar@{-}^{(1,2)}[r] & 1}\\
\widetilde{BC}_n&\xymatrix{1 \ar@{-}^{(2,1)}[r] & 2 \ar@{-}[r] & \cdots \ar@{-}[r] & 2 \ar@{-}^{(2,1)}[r] & 2}\\
\widetilde{BD}_n&\xymatrix@R=1mm{
1 \ar@{-}[dr] & & & & \\
& 2 \ar@{-}[dl] \ar@{-}[r] & \cdots \ar@{-}[r] & 2 \ar@{-}[r]^-{(2,1)} & 2\\
1 & & & &
}\\
\widetilde{CD}_n&\xymatrix@R=1mm{
1 \ar@{-}[dr] & & & & \\
& 2 \ar@{-}[dl] \ar@{-}[r] & \cdots \ar@{-}[r] & 2 \ar@{-}[r]^-{(1,2)} & 1\\
1 & & & &
}\\
\widetilde{D}_n&\xymatrix@R=1mm{
1 \ar@{-}[dr] & & & & 1 \\
& 2 \ar@{-}[dl] \ar@{-}[r] & \cdots \ar@{-}[r] & 2 \ar@{-}[ur] \ar@{-}[dr] &\\
1 & & & & 1
}\\
\widetilde{E}_6&\xymatrix@R=5mm{
&&1 \ar@{-}[d]&&\\
&&2 \ar@{-}[d]&&\\
1 \ar@{-}[r] & 2 \ar@{-}[r] & 3\ar@{-}[r] & 2 \ar@{-}[r] & 1
}\\
\widetilde{E}_7&\xymatrix@R=5mm{
&&&2 \ar@{-}[d]&&&\\
1 \ar@{-}[r] & 2 \ar@{-}[r] & 3 \ar@{-}[r] & 4 \ar@{-}[r] & 3 \ar@{-}[r] & 2 \ar@{-}[r] & 1
}\\
\widetilde{E}_8&\xymatrix@R=5mm{
&&&&&3 \ar@{-}[d]&&\\
1 \ar@{-}[r] & 2 \ar@{-}[r] & 3 \ar@{-}[r] & 4 \ar@{-}[r] & 5 \ar@{-}[r] & 6 \ar@{-}[r] & 4 \ar@{-}[r] & 2
}\\
\widetilde{F}_{41}&\xymatrix{
1 \ar@{-}[r] & 2 \ar@{-}[r] & 3 \ar@{-}[r]^-{(1,2)} & 2 \ar@{-}[r] & 1
}\\
\widetilde{F}_{42}&\xymatrix{
1 \ar@{-}[r] & 2 \ar@{-}[r] & 3 \ar@{-}[r]^-{(2,1)} & 4 \ar@{-}[r] & 2
}\\
\widetilde{G}_{21}&\xymatrix{
1 \ar@{-}[r] & 2 \ar@{-}[r]^-{(1,3)} & 1
}\\
\widetilde{G}_{22}&\xymatrix{
1 \ar@{-}[r] & 2 \ar@{-}[r]^-{(3,1)} & 3
}
\end{array}
\]
Our theorem below gives a complete classification of the quivers which may appear as $\AR(\CM R)$. These are either doubles of extended Dynkin diagrams or the quivers $\widetilde{A}_0$ or $\widetilde{CL}_n$ (see \cite{BES}). Note that some of extended Dynkin diagrams (type $\widetilde{B}_n, \widetilde{CD}_n, \widetilde{F}_{41}$ and $\widetilde{G}_{21}$) do not appear.

\begin{Thm}\label{classfyquiver}
Let $(R,\mathfrak{m},k)$ be a two-dimensional Gorenstein complete local ring of finite Cohen-Macaulay type with $\ch k=0$. Then quivers which may appear as $\AR(\CM R)$ are precisely listed below with trivial Auslander-Reiten translations. They are either doubles of all extended Dynkin diagrams except for type $\widetilde{B}_n, \widetilde{CD}_n, \widetilde{F}_{41}$ and $\widetilde{G}_{21}$, or the quivers $\widetilde{A}_0$ or $\widetilde{CL}_n$ having loops. Moreover, we can also determine the divisor class group $\Cl(R)$.
\[
\begin{array}{c|c|c}
Type&\AR(\CM R)&\Cl(R)\\ \hline\hline
\widetilde{A}_0&\xymatrix{
R \ar@(ur,dr)^{(2,2)}
}&1\\ \hline
\widetilde{A}_{11}&\xymatrix{
R \ar@<0.5ex>[r]^-{(4,1)} & \circ \ar@<0.5ex>[l]^-{(1,4)}
}&1\\ \hline
\widetilde{A}_{12}&\xymatrix{
R \ar@<0.5ex>[r]^-{(2,2)} & \circ \ar@<0.5ex>[l]^-{(2,2)}
}&C_2\\ \hline
\widetilde{A}_n(n\geq2)&\xymatrix@R=5mm{
& & R \ar@<0.5ex>[dll] \ar@<0.5ex>[drr] & & \\
\circ \ar@<0.5ex>[urr] \ar@<0.5ex>[r] & \circ \ar@<0.5ex>[l] \ar@<0.5ex>[r] & \cdots \ar@<0.5ex>[l] \ar@<0.5ex>[r] & \circ \ar@<0.5ex>[l] \ar@<0.5ex>[r] & \circ \ar@<0.5ex>[ull] \ar@<0.5ex>[l]
}&C_{n+1}\\ \hline
\widetilde{C}_n&\xymatrix{
R \ar@<0.5ex>[r]^-{(2,1)} & \circ \ar@<0.5ex>[l]^-{(1,2)} \ar@<0.5ex>[r] & \cdots \ar@<0.5ex>[l] \ar@<0.5ex>[r] & \circ \ar@<0.5ex>[l] \ar@<0.5ex>[r]^-{(1,2)} & \circ \ar@<0.5ex>[l]^-{(2,1)}
}&C_2\\ \hline
\widetilde{BC}_n&\xymatrix{
R \ar@<0.5ex>[r]^-{(2,1)} & \circ \ar@<0.5ex>[l]^-{(1,2)} \ar@<0.5ex>[r] & \cdots \ar@<0.5ex>[l] \ar@<0.5ex>[r] & \circ \ar@<0.5ex>[l] \ar@<0.5ex>[r]^-{(2,1)} & \circ \ar@<0.5ex>[l]^-{(1,2)}
}&1\\ \hline
\widetilde{BD}_n&\xymatrix@R=1mm{
R \ar@<0.5ex>[dr] & & & & \\
& \circ \ar@<0.5ex>[ul] \ar@<0.5ex>[dl] \ar@<0.5ex>[r] & \cdots \ar@<0.5ex>[l] \ar@<0.5ex>[r] & \circ \ar@<0.5ex>[l]\ar@<0.5ex>[r]^-{(2,1)} & \circ \ar@<0.5ex>[l]^-{(1,2)}\\
\circ \ar@<0.5ex>[ur] & & & &
}&C_2\\ \hline
\widetilde{D}_n\ (n\colon{\rm odd})&\xymatrix@R=1mm{
R \ar@<0.5ex>[dr] & & & & \circ \ar@<0.5ex>[dl]\\
& \circ \ar@<0.5ex>[ul] \ar@<0.5ex>[dl] \ar@<0.5ex>[r] & \cdots \ar@<0.5ex>[l] \ar@<0.5ex>[r] & \circ \ar@<0.5ex>[l]\ar@<0.5ex>[ur] \ar@<0.5ex>[dr] & \\
\circ \ar@<0.5ex>[ur] & & & & \circ \ar@<0.5ex>[ul]
}&C_4\\ \hline
\widetilde{D}_n\ (n\colon{\rm even})&\xymatrix@R=1mm{
R \ar@<0.5ex>[dr] & & & & \circ \ar@<0.5ex>[dl]\\
& \circ \ar@<0.5ex>[ul] \ar@<0.5ex>[dl] \ar@<0.5ex>[r] & \cdots \ar@<0.5ex>[l] \ar@<0.5ex>[r] & \circ \ar@<0.5ex>[l]\ar@<0.5ex>[ur] \ar@<0.5ex>[dr] & \\
\circ \ar@<0.5ex>[ur] & & & & \circ \ar@<0.5ex>[ul]
}&C_2\times C_2\\ \hline
\widetilde{E}_6&\xymatrix@R=5mm{
&&\circ \ar@<0.5ex>[d]&&\\
&&\circ \ar@<0.5ex>[d] \ar@<0.5ex>[u]&&\\
R \ar@<0.5ex>[r] & \circ \ar@<0.5ex>[l] \ar@<0.5ex>[r] & \circ \ar@<0.5ex>[l] \ar@<0.5ex>[r] \ar@<0.5ex>[u] & \circ \ar@<0.5ex>[l] \ar@<0.5ex>[r] & \circ \ar@<0.5ex>[l]
}&C_3\\ \hline
\widetilde{E}_7&\xymatrix@R=5mm{
&&&\circ \ar@<0.5ex>[d]&&&\\
R \ar@<0.5ex>[r] & \circ \ar@<0.5ex>[l] \ar@<0.5ex>[r] & \circ \ar@<0.5ex>[l] \ar@<0.5ex>[r] & \circ \ar@<0.5ex>[l] \ar@<0.5ex>[r] \ar@<0.5ex>[u] & \circ \ar@<0.5ex>[l] \ar@<0.5ex>[r] & \circ \ar@<0.5ex>[l] \ar@<0.5ex>[r] & \circ \ar@<0.5ex>[l]
}&C_2\\ \hline
\widetilde{E}_8&\xymatrix@R=5mm{
&&&&&\circ \ar@<0.5ex>[d]&&\\
R \ar@<0.5ex>[r] & \circ \ar@<0.5ex>[l] \ar@<0.5ex>[r] & \circ \ar@<0.5ex>[l] \ar@<0.5ex>[r] & \circ \ar@<0.5ex>[l] \ar@<0.5ex>[r] & \circ \ar@<0.5ex>[l] \ar@<0.5ex>[r] & \circ \ar@<0.5ex>[l] \ar@<0.5ex>[r] \ar@<0.5ex>[u] & \circ \ar@<0.5ex>[l] \ar@<0.5ex>[r] & \circ \ar@<0.5ex>[l]
}&1\\ \hline
\widetilde{F}_{42}&\xymatrix{
R \ar@<0.5ex>[r] & \circ \ar@<0.5ex>[l] \ar@<0.5ex>[r] & \circ \ar@<0.5ex>[l] \ar@<0.5ex>[r]^-{(2,1)} & \circ \ar@<0.5ex>[l]^-{(1,2)} \ar@<0.5ex>[r] &\circ \ar@<0.5ex>[l]
}&1\\ \hline
\widetilde{G}_{22}&\xymatrix{
R \ar@<0.5ex>[r] & \circ \ar@<0.5ex>[r]^{(3,1)} \ar@<0.5ex>[l] & \circ \ar@<0.5ex>[l]^{(1,3)}
}&1\\ \hline
\widetilde{CL}_n&\xymatrix{
R \ar@<0.5ex>[r]^-{(2,1)} & \circ \ar@<0.5ex>[l]^-{(1,2)} \ar@<0.5ex>[r] & \cdots \ar@<0.5ex>[l] \ar@<0.5ex>[r] & \circ \ar@<0.5ex>[l] \ar@(ur,dr)
}&1

\end{array}
\]
\end{Thm}
\begin{proof}
Observe that these quivers certainly appear as $\AR(\CM R)$ (see Examples \ref{typeCL}, \ref{typeC}, \ref{typeBC}, \ref{typeG22}, \ref{mathbbR} and \ref{typeBD}). Below we show the necessity.

By Theorem \ref{frt}, there exists a finite Galois extension $l/k$ and a finite subgroup $G\subseteq GL_2(l)\rtimes\Gal(l/k)$ such that $R\cong l[[x,y]]^G$ holds. By Theorem \ref{Gor}, we know that $H:=G\cap\Gal(l/k)$ is contained in $SL_2(l)$. Let $\bar{l}$ be an algebraic closure of $l$. If we view $H\subseteq SL_2(\bar{l})$, then there exists $P\in GL_2(\bar{l})$ such that $H':=P^{-1}HP\subseteq SL_2(\bar{l})$ coincides with one in the list exhibited at the first of this section. Let $l'/l$ be a field extension of finite degree in $\bar{l}$ such that $P\in GL_2(l')$ and $H'\subseteq SL_2(l')$ holds. By taking the Galois closure, we may assume that the extension $l'/k$ is Galois. If we let $G'\subseteq GL_2(l)\rtimes\Gal(l'/k)\subseteq GL_2(l')\rtimes\Gal(l'/k)$ be the preimage of $G$ under the natural group homomorphism $GL_2(l)\rtimes\Gal(l'/k)\to GL_2(l)\rtimes\Gal(l/k)$, then we have $G'\cap GL_2(l')=H$ and $l[[x,y]]^G=l'[[x,y]]^{G'}$. Therefore by expanding the field $l$, we may assume that $H$ is of the form of one in the list exhibited at the first of this section. Moreover, by expanding $l$ more, we may also assume that $H$ splits over $l$, i.e. the endomorphism ring of any simple $lH$-module is isomorphic to $l$.

If $H$ is trivial, then $\AR(\CM S^H)$ is 
\[\xymatrix{
S^H \ar@(ur,dr)^{(2,2)}
}.\]
Thus $\AR(\CM R)$ is also of this form.

Assume that $H$ is of type $A_{2n}$, i.e. $H=\bigg\langle\begin{bmatrix}\zeta_{2n+1}&0\\0&\zeta_{2n+1}^{-1}\end{bmatrix}\bigg\rangle$. Then $\simp lH=\{W_j\}_{j\in\mathbb{Z}/(2n+1)\mathbb{Z}}$ holds and $\AR(\CM S^H)$ is as follows, where $N_0=S^H$.
\[\xymatrix@R=5mm{
& & N_0 \ar@<0.5ex>[dll] \ar@<0.5ex>[drr] & & \\
N_1 \ar@<0.5ex>[urr] \ar@<0.5ex>[r] & N_2 \ar@<0.5ex>[l] \ar@<0.5ex>[r] & \cdots \ar@<0.5ex>[l] \ar@<0.5ex>[r] & N_{2n-1} \ar@<0.5ex>[l] \ar@<0.5ex>[r] & N_{2n} \ar@<0.5ex>[ull] \ar@<0.5ex>[l]
}\]
Since the action of $G$ on $\AR(\CM S^H)$ fixes $N_0$, it must be either: the trivial action or the action of swapping the left and right sides.

(Case 1) The action $G\curvearrowright\AR(\CM S^H)$ is trivial.

We have $\simp l*G=\{V_i\}_{i\in\mathbb{Z}/(2n+1)\mathbb{Z}}$ and $\Res V_i=W_i^{\oplus a_i}$. Let us calculate $a_i$'s. First, $V_0$ is the trivial representation and $a_0=1$ holds as in Example \ref{typeCL}. Second, $V_1\oplus V_{2n}=\mathfrak{n}/\mathfrak{n}^2$ and $a_1=a_{2n}=1$ hold since $\Res\mathfrak{n}/\mathfrak{n}^2=W_1\oplus W_{2n}$. Next, we get $a_2=1$ by $\Res(V_1\otimes_lV_1)=W_0\oplus W_2$. Continuing this operation, we have $a_j=1$ for all $j$. Thus the minimal projective resolution of $V_i\in\mod S*G$ must be
\[0\to S\otimes_lV_i\to S\otimes_l(V_{i-1}\oplus V_{i+1})\to S\otimes_lV_i\to V_i\to 0.\]
Therefore $\AR(\CM R)$ is as follows.
\[\xymatrix@R=5mm{
& & M_0 \ar@<0.5ex>[dll] \ar@<0.5ex>[drr] & & \\
M_1 \ar@<0.5ex>[urr] \ar@<0.5ex>[r] & M_2 \ar@<0.5ex>[l] \ar@<0.5ex>[r] & \cdots \ar@<0.5ex>[l] \ar@<0.5ex>[r] & M_{2n-1} \ar@<0.5ex>[l] \ar@<0.5ex>[r] & M_{2n} \ar@<0.5ex>[ull] \ar@<0.5ex>[l]
}\]

(Case 2) The action $G\curvearrowright\AR(\CM S^H)$ swaps the left and right sides.

We have $\simp l*G=\{V_0, V_{\pm1}, \cdots, V_{\pm n}\}$, $\Res V_0=W_0^{\oplus a_0}$ and $\Res V_{\pm i}=(W_i\oplus W_{-i})^{\oplus a_i}$. Similar as in the previous case, we can conclude $a_0=a_1=\cdots a_n=1$. Thus the minimal projective resolution of $V_i\in\mod S*G$ must be
\[0\to S\otimes_lV_0\to S\otimes_lV_{\pm1}\to S\otimes_lV_0\to V_0\to 0,\]
\[0\to S\otimes_lV_{\pm1}\to S\otimes_l(V_0^{\oplus2}\oplus V_{\pm2})\to S\otimes_lV_{\pm1}\to V_{\pm1}\to 0,\]
\[0\to S\otimes_lV_{\pm i}\to S\otimes_l(V_{\pm(i-1)}\oplus V_{\pm(i+1)})\to S\otimes_lV_{\pm i}\to V_{\pm i}\to 0\ (1<i<n)\]
and
\[0\to S\otimes_lV_{\pm n}\to S\otimes_l(V_{\pm(n-1)}\oplus V_{\pm n})\to S\otimes_lV_{\pm n}\to V_{\pm n}\to 0.\]
Therefore $\AR(\CM R)$ is as follows.
\[\xymatrix{
M_0 \ar@<0.5ex>[r]^-{(2,1)} & M_{\pm1} \ar@<0.5ex>[l]^-{(1,2)} \ar@<0.5ex>[r] & \cdots \ar@<0.5ex>[l] \ar@<0.5ex>[r] & M_{\pm n} \ar@<0.5ex>[l] \ar@(ur,dr)
}\]

When $H$ is of another type, we can deduce that $\AR(\CM R)$ must be any in the list by classifying possible actions of $G$ on $\AR(\CM S^H)$. We present the results without detailed proofs.

If $H$ is of type $A_{2n-1}$, then $\AR(\CM S^H)$ is as follows.
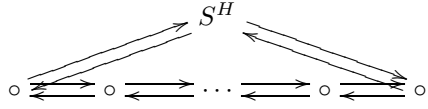
\[\xymatrix@R=5mm{
& & S^H \ar@<0.5ex>[dll] \ar@<0.5ex>[drr] & & \\
\circ \ar@<0.5ex>[urr] \ar@<0.5ex>[r] & \circ \ar@<0.5ex>[l] \ar@<0.5ex>[r] & \cdots \ar@<0.5ex>[l] \ar@<0.5ex>[r] & \circ \ar@<0.5ex>[l] \ar@<0.5ex>[r] & \circ \ar@<0.5ex>[ull] \ar@<0.5ex>[l]
}\]
Since the action of $G$ on $\AR(\CM S^H)$ fixes $S^H$, it must be either: the trivial action or the action of swapping the left and right sides. If the action is trivial, then $\AR(\CM R)$ is any one of the list below.
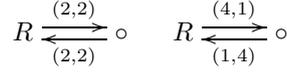
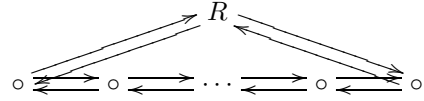
\[
\begin{array}{cc}
\xymatrix{
R \ar@<0.5ex>[r]^-{(2,2)} & \circ \ar@<0.5ex>[l]^-{(2,2)}
}&\xymatrix{
R \ar@<0.5ex>[r]^-{(4,1)} & \circ \ar@<0.5ex>[l]^-{(1,4)}
}
\end{array}
\]
\[\xymatrix@R=5mm{
& & R \ar@<0.5ex>[dll] \ar@<0.5ex>[drr] & & \\
\circ \ar@<0.5ex>[urr] \ar@<0.5ex>[r] & \circ \ar@<0.5ex>[l] \ar@<0.5ex>[r] & \cdots \ar@<0.5ex>[l] \ar@<0.5ex>[r] & \circ \ar@<0.5ex>[l] \ar@<0.5ex>[r] & \circ \ar@<0.5ex>[ull] \ar@<0.5ex>[l]
}\]
If the action swaps the left and right sides, then $\AR(\CM R)$ is any one of the list below.
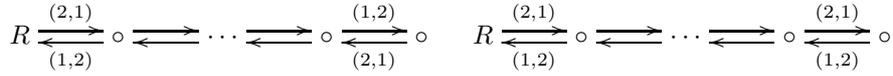
\[
\begin{array}{cc}
\xymatrix{
R \ar@<0.5ex>[r]^-{(2,1)} & \circ \ar@<0.5ex>[l]^-{(1,2)} \ar@<0.5ex>[r] & \cdots \ar@<0.5ex>[l] \ar@<0.5ex>[r] & \circ \ar@<0.5ex>[l] \ar@<0.5ex>[r]^-{(1,2)} & \circ \ar@<0.5ex>[l]^-{(2,1)}
}&\xymatrix{
R \ar@<0.5ex>[r]^-{(2,1)} & \circ \ar@<0.5ex>[l]^-{(1,2)} \ar@<0.5ex>[r] & \cdots \ar@<0.5ex>[l] \ar@<0.5ex>[r] & \circ \ar@<0.5ex>[l] \ar@<0.5ex>[r]^-{(2,1)} & \circ \ar@<0.5ex>[l]^-{(1,2)}
}
\end{array}
\]

If $H$ is of type $D_n$, then $\AR(\CM S^H)$ is as follows.
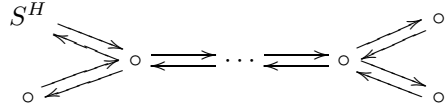
\[\xymatrix@R=1mm{
S^H \ar@<0.5ex>[dr] & & & & \circ \ar@<0.5ex>[dl]\\
& \circ \ar@<0.5ex>[ul] \ar@<0.5ex>[dl] \ar@<0.5ex>[r] & \cdots \ar@<0.5ex>[l] \ar@<0.5ex>[r] & \circ \ar@<0.5ex>[l]\ar@<0.5ex>[ur] \ar@<0.5ex>[dr] & \\
\circ \ar@<0.5ex>[ur] & & & & \circ \ar@<0.5ex>[ul]
}\]
Since the action of $G$ on $\AR(\CM S^H)$ fixes $S^H$, it must be one of the following: the trivial action, the action of swapping the rightmost two vertices or when $n=4$, the action of replacing the outer three vertices except for $S^H$. If the action is trivial, then $\AR(\CM R)$ is as follows.
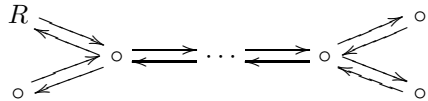
\[\xymatrix@R=1mm{
R \ar@<0.5ex>[dr] & & & & \circ \ar@<0.5ex>[dl]\\
& \circ \ar@<0.5ex>[ul] \ar@<0.5ex>[dl] \ar@<0.5ex>[r] & \cdots \ar@<0.5ex>[l] \ar@<0.5ex>[r] & \circ \ar@<0.5ex>[l]\ar@<0.5ex>[ur] \ar@<0.5ex>[dr] & \\
\circ \ar@<0.5ex>[ur] & & & & \circ \ar@<0.5ex>[ul]
}\]
If the action swaps the rightmost two vertices, then $\AR(\CM R)$ is as follows.
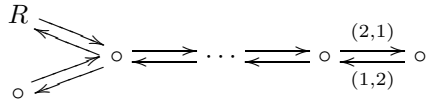
\[\xymatrix@R=1mm{
R \ar@<0.5ex>[dr] & & & & \\
& \circ \ar@<0.5ex>[ul] \ar@<0.5ex>[dl] \ar@<0.5ex>[r] & \cdots \ar@<0.5ex>[l] \ar@<0.5ex>[r] & \circ \ar@<0.5ex>[l]\ar@<0.5ex>[r]^-{(2,1)} & \circ \ar@<0.5ex>[l]^-{(1,2)}\\
\circ \ar@<0.5ex>[ur] & & & &
}\]
If $n=4$ and the action replaces the outer three vertices except for $S^H$, then $\AR(\CM R)$ is as follows.
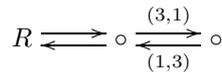
\[\xymatrix{
R \ar@<0.5ex>[r] & \circ \ar@<0.5ex>[r]^{(3,1)} \ar@<0.5ex>[l] & \circ \ar@<0.5ex>[l]^{(1,3)}
}\]

If $H$ is of type $E_6$, then $\AR(\CM S^H)$ is as follows.
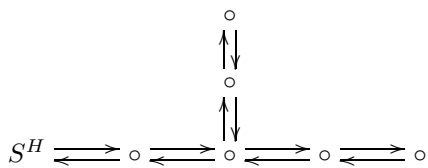
\[\xymatrix@R=5mm{
&&\circ \ar@<0.5ex>[d]&&\\
&&\circ \ar@<0.5ex>[d] \ar@<0.5ex>[u]&&\\
S^H \ar@<0.5ex>[r] & \circ \ar@<0.5ex>[l] \ar@<0.5ex>[r] & \circ \ar@<0.5ex>[l] \ar@<0.5ex>[r] \ar@<0.5ex>[u] & \circ \ar@<0.5ex>[l] \ar@<0.5ex>[r] & \circ \ar@<0.5ex>[l]
}\]
Since the action of $G$ on $\AR(\CM S^H)$ fixes $S^H$, it must be either: the trivial action or the action of swapping the two branches that do not contain $S^H$. If the action is trivial, then $\AR(\CM R)$ is as follows.
\[\xymatrix@R=5mm{
&&\circ \ar@<0.5ex>[d]&&\\
&&\circ \ar@<0.5ex>[d] \ar@<0.5ex>[u]&&\\
R \ar@<0.5ex>[r] & \circ \ar@<0.5ex>[l] \ar@<0.5ex>[r] & \circ \ar@<0.5ex>[l] \ar@<0.5ex>[r] \ar@<0.5ex>[u] & \circ \ar@<0.5ex>[l] \ar@<0.5ex>[r] & \circ \ar@<0.5ex>[l]
}\]
If the action swaps the two branches that do not contain $S^H$, then $\AR(\CM R)$ is as follows.
\[\xymatrix{
R \ar@<0.5ex>[r] & \circ \ar@<0.5ex>[l] \ar@<0.5ex>[r] & \circ \ar@<0.5ex>[l] \ar@<0.5ex>[r]^-{(2,1)} & \circ \ar@<0.5ex>[l]^-{(1,2)} \ar@<0.5ex>[r] &\circ \ar@<0.5ex>[l]
}\]

If $H$ is of type $E_7$, then $\AR(\CM S^H)$ is as follows.
\[\xymatrix@R=5mm{
&&&\circ \ar@<0.5ex>[d]&&&\\
S^H \ar@<0.5ex>[r] & \circ \ar@<0.5ex>[l] \ar@<0.5ex>[r] & \circ \ar@<0.5ex>[l] \ar@<0.5ex>[r] & \circ \ar@<0.5ex>[l] \ar@<0.5ex>[r] \ar@<0.5ex>[u] & \circ \ar@<0.5ex>[l] \ar@<0.5ex>[r] & \circ \ar@<0.5ex>[l] \ar@<0.5ex>[r] & \circ \ar@<0.5ex>[l]
}\]
Since the action of $G$ on $\AR(\CM S^H)$ fixes $S^H$, it must be trivial. Then $\AR(\CM S^H)$ is as follows.
\[\xymatrix@R=5mm{
&&&\circ \ar@<0.5ex>[d]&&&\\
R \ar@<0.5ex>[r] & \circ \ar@<0.5ex>[l] \ar@<0.5ex>[r] & \circ \ar@<0.5ex>[l] \ar@<0.5ex>[r] & \circ \ar@<0.5ex>[l] \ar@<0.5ex>[r] \ar@<0.5ex>[u] & \circ \ar@<0.5ex>[l] \ar@<0.5ex>[r] & \circ \ar@<0.5ex>[l] \ar@<0.5ex>[r] & \circ \ar@<0.5ex>[l]
}\]

If $H$ is of type $E_8$, then $\AR(\CM S^H)$ is as follows.
\[\xymatrix@R=5mm{
&&&&&\circ \ar@<0.5ex>[d]&&\\
S^H \ar@<0.5ex>[r] & \circ \ar@<0.5ex>[l] \ar@<0.5ex>[r] & \circ \ar@<0.5ex>[l] \ar@<0.5ex>[r] & \circ \ar@<0.5ex>[l] \ar@<0.5ex>[r] & \circ \ar@<0.5ex>[l] \ar@<0.5ex>[r] & \circ \ar@<0.5ex>[l] \ar@<0.5ex>[r] \ar@<0.5ex>[u] & \circ \ar@<0.5ex>[l] \ar@<0.5ex>[r] & \circ \ar@<0.5ex>[l]
}\]
Since the action of $G$ on $\AR(\CM S^H)$ fixes $S^H$, it must be trivial. Then $\AR(\CM S^H)$ is as follows.
\[
\begin{gathered}[b]
\xymatrix@R=5mm{
&&&&&\circ \ar@<0.5ex>[d]&&\\
R \ar@<0.5ex>[r] & \circ \ar@<0.5ex>[l] \ar@<0.5ex>[r] & \circ \ar@<0.5ex>[l] \ar@<0.5ex>[r] & \circ \ar@<0.5ex>[l] \ar@<0.5ex>[r] & \circ \ar@<0.5ex>[l] \ar@<0.5ex>[r] & \circ \ar@<0.5ex>[l] \ar@<0.5ex>[r] \ar@<0.5ex>[u] & \circ \ar@<0.5ex>[l] \ar@<0.5ex>[r] & \circ \ar@<0.5ex>[l]
}
\\[-\dp\strutbox]
\end{gathered}
\qedhere
\]
\end{proof}

\begin{appendix}

\end{appendix}

\bibliographystyle{amsplain} 
\bibliography{reference}

\begin{comment}

\end{comment}

\end{document}